\documentclass{article}
\usepackage{letltxmacro}
\usepackage{amsmath, amssymb, amsthm}
\usepackage{tikz}
\usetikzlibrary{decorations.pathreplacing}
\usepackage{subcaption}
\usepackage{float}
\usepackage{enumitem}

\tikzstyle{b_vertex}=[circle,fill=black!100,text=white,inner sep=0.8mm,draw]
\tikzstyle{w_vertex}=[circle,fill=white!100,text=white,inner sep=0.8mm,draw]
\tikzstyle{point}=[circle,fill=black,inner sep=0.1mm]
\tikzstyle{path_edge}=[thick]

\newtheorem{theorem}{Theorem}
\newtheorem{lemma}[theorem]{Lemma}
\newtheorem{claim}{Claim}

\newtheorem{corollary}{Corollary}

\newtheorem*{conjecture*}{Conjecture}

\newcommand{\dist}{\delta} 
\newcommand{\sh}{\Lambda} 
\newcommand{\CC}{S} 
\newcommand{\FF}{Q} 
\newcommand{\cc}{s} 
\newcommand{\ff}{q} 
\newcommand{\mm}{\mu} 
\newcommand{\GG}{\mathcal{G}} 
\newcommand{\BB}{\mathcal{B}} 
\newcommand{\RR}{\mathcal{R}} 

\newcommand{\compcent}[1]{\vcenter{\hbox{$#1\circ$}}}
\newcommand{\comp}{\mathbin{\mathchoice
  {\compcent\scriptstyle}{\compcent\scriptstyle}
  {\compcent\scriptscriptstyle}{\compcent\scriptscriptstyle}}}

\author{Aistis Atminas\thanks{School of Science and Technology, Nottingham Trent University, Nottingham NG11 8NS, UK. E-mail: Aistis.Atminas@ntu.ac.uk.} 
\and 
Viktor Zamaraev\thanks{Mathematics Institute, University of Warwick, Coventry CV4 7AL, UK. E-mail: V.Zamaraev@warwick.ac.uk. The author acknowledge support from EPSRC, grant EP/L020408/1; and
from Russian Foundation for Basic Research, grant 14-01-00515-a.}}

\date{}
\title{On forbidden induced subgraphs for unit disk graphs}

\makeatletter
\let\oldr@@t\r@@t
\def\r@@t#1#2{%
\setbox0=\hbox{$\oldr@@t#1{#2\kern 0.08em}$}\dimen0=\ht0
\advance\dimen0-0.2\ht0
\setbox2=\hbox{\vrule height\ht0 depth -\dimen0}%
{\box0\lower0.4pt\box2}}
\LetLtxMacro{\oldsqrt}{\sqrt}
\renewcommand*{\sqrt}[2][\ ]{\oldsqrt[#1]{#2}}
\makeatother

\begin{document}

\maketitle

\begin{abstract}
	A unit disk graph is the intersection graph of disks of equal radii in the plane. The class
	of unit disk graphs is hereditary, and therefore admits a characterization in terms of 
	minimal forbidden induced subgraphs. In spite of quite active study of unit disk graphs
	very little is known about minimal forbidden induced subgraphs for this class. 
	We found only finitely many minimal non unit disk graphs in the literature. In this paper
	we study in a systematic way forbidden induced subgraphs for the class of unit disk
	graphs. We develop several structural and geometrical tools, and use them to
	reveal infinitely many new minimal non unit disk graphs. Further we use these results
	to investigate structure of co-bipartite unit disk graphs.
	In particular, we give structural characterization of those co-bipartite unit disk graphs whose edges
	between parts form a $C_4$-free bipartite graph, and show that bipartite complements
	of these graphs are also unit disk graphs. Our results lead us to
	propose a conjecture that the class of co-bipartite unit disk graphs is closed under
	bipartite complementation.
\end{abstract}


\section{Introduction}

A graph is \textit{unit disk graph} (UDG for short) if its vertices can be represented
as points in the plane such that two vertices are adjacent if and only if the corresponding
points are at distance at most 1 from each other. 
Unit disk graphs has been very actively studied in recent decades. One of the reasons 
for this is that UDGs appear to be useful in number of applications. Perhaps a major
application area for UDGs is wireless networks. Here a UDG is used to model the topology
of a network consisting of nodes that communicates by means of omnidirectional 
antennas with equal transmission-reception range. Many research projects aimed at
designing algorithms for different graph optimization problems specifically on unit disk
graphs, as solutions to these problems are of practical importance for efficient operation
of modeled networks. We refer the reader to \cite{BB09,B96} and references therein 
for more details on applications of UDGs.

The class of unit disk graphs is \textit{hereditary}, that is closed under 
vertex deletion or, equivalently, closed under induced 
subgraphs\footnote{All subgraphs in this paper are induced and further we sometimes omit word `induced'.}.
It is well known and can be easily proved that every hereditary class of graphs admits
characterization in terms of minimal forbidden induced subgraphs. Formally, for
a hereditary class $\mathcal{X}$ there exists a unique minimal under inclusion set of 
graphs $M$ such that $\mathcal{X}$ coincides with the family $Free(M)$ of graphs none of which contains a graph from $M$ as an induced subgraph. Graphs in $M$ are called
\textit{minimal forbidden induced subgraphs} for $\mathcal{X}$. 
Such an obstructive specification of a hereditary class may be useful for investigation of 
its structural, algorithmic and combinatorial properties. For instance, forbidden subgraphs
characterization of a class may be helpful in testing whether a graph belongs to the 
class or not. In particular, if the set of minimal forbidden subgraphs is finite, then, clearly, 
the problem of recognizing graphs in the class is polynomially solvable.
However, describing a hereditary class in terms of its minimal
forbidden induced subgraphs may be extremely hard problem. For example, for the class
of perfect graphs it took more than 40 years to obtain forbidden subgraph characterization
\cite{CRST06}.

Despite extensive study of the class of unit disk graphs very little is known about its forbidden induced subgraphs. We found only few minimal
non unit disk graphs in the literature, namely, $K_{1,6}$, 
$K_{2,3}$, 
and five other graphs
(see Figure \ref{fig:knownMinForb}) \cite{H80, HS95}.
However, unless $\textup{P} = \textup{NP}$, the set of minimal forbidden induced subgraphs is infinite, 
since the problem of recognizing unit disk graphs is known to be NP-hard \cite{BK98}.
Interestingly, only the fact that unit disk graphs avoid $K_{1,6}$ already turned out to be useful 
in algorithms design.
For example, the fact was utilized in \cite{MBHRR95} for obtaining 3-approximation algorithm 
for the maximum independent set problem and 5-approximation algorithm for the dominating set problem.
In \cite{FFSM14} da Fonseca et al. used additional geometrical restrictions of UDGs to design an algorithm
for the latter problem with better approximation factor $44/9$. 
The authors pointed out that further improvement 
may require new information about forbidden induced subgraphs for UDGs, and in a subsequent paper \cite{FSMF15} they developed algorithm for recognizing UDGs. Unfortunately, (though, not surprising as 
the corresponding problem is NP-hard) in worst cases the algorithm works exponential time, and the experimental results are available only for small graphs and do not discover any new minimal forbidden
subgraphs.

In the present paper we systematically study forbidden induced subgraphs for the class
of unit disk graphs, and reveal infinitely many new minimal forbidden subgraphs. 
For example,
we show that all complements of even cycles with at least eight vertices are minimal
non-UDGs. In contrast, all complements of odd cycles are UDGs. 
We use the obtained results to investigate structure of co-bipartite unit disk graphs.
Specifically, we characterize the class of $C_4^*$-free co-bipartite
UDGs, that is co-bipartite UDGs whose edges between parts form a bipartite graph without
cycle on four vertices. Further we show that bipartite complement of every $C_4^*$-free 
co-bipartite UDG is also (co-bipartite) UDG. This fact and the structure of the set of 
found obstructions leads us to pose a conjecture that
the class of co-bipartite UDGs is closed under bipartite complementation.

The paper is organized as follows. In Section \ref{sec:term} we introduce necessary 
definitions and notation. In Section \ref{sec:tools} we develop auxiliary geometrical and 
structural tools that may be of their own interest. Using these tools we derive new minimal forbidden induced subgraphs in Section \ref{sec:forb}. In Section \ref{sec:structure} we give
structural characterization of certain classes of co-bipartite UDGs. 
In the last Section \ref{sec:conclusion} we discuss the results and open problems.

\begin{figure}[H]
	\centering
	\begin{tikzpicture}[scale=1,auto=left]
		\node[w_vertex] (1) at (0,0) { }; 	
		\node[w_vertex] (2) at (0,1) { };
		\node[w_vertex] (3) at (0,-1) { };
		\node[w_vertex] (4) at (0.866,0.5) { };
		\node[w_vertex] (5) at (-0.866,0.5) { };
		\node[w_vertex] (6) at (0.866,-0.5) { };
		\node[w_vertex] (7) at (-0.866,-0.5) { };

		\foreach \from/\to in {1/2,1/3,1/4,1/5,1/6, 1/7}
	    	\draw (\from) -- (\to);
	    	
	    	\coordinate [label=center:$K_{1,6}$] (S333) at (0,-1.75);

		\node[w_vertex] (1) at (4,0) { }; 	
		\node[w_vertex] (2) at (3,0) { };
		\node[w_vertex] (3) at (5,0) { };
		\node[w_vertex] (4) at (4,1) { };
		\node[w_vertex] (5) at (4,-1) { };

		\foreach \from/\to in {1/4,1/5,2/4,2/5,3/4,3/5}
	    	\draw (\from) -- (\to);
	    	
	    	\coordinate [label=center:$K_{2,3}$] (K23) at (4,-1.75);

		\node[w_vertex] (1) at (7.75,0) { }; 	
		\node[w_vertex] (2) at (7.75,1) { };
		\node[w_vertex] (3) at (7.75,-1) { };
		\node[w_vertex] (4) at (8.5,0.5) { };
		\node[w_vertex] (5) at (7,0.5) { };
		\node[w_vertex] (6) at (8.5,-0.5) { };
		\node[w_vertex] (7) at (7,-0.5) { };
		
		\foreach \from/\to in {1/2,1/3,2/4,2/5,3/6,3/7,4/6,5/7}
	    	\draw (\from) -- (\to);
	    	
	    	\draw (5) to[out=60] (8,1.3) to[out=-40,in=110] (4);
	    	
	    	\coordinate [label=center:$G_1$] (G1) at (7.75,-1.75);
	    	
		\node[w_vertex] (1) at (11.25,0) { }; 	
		\node[w_vertex] (2) at (11.25,1) { };
		\node[w_vertex] (3) at (11.25,-1) { };
		\node[w_vertex] (4) at (12,0.5) { };
		\node[w_vertex] (5) at (10.5,0.5) { };
		\node[w_vertex] (6) at (12,-0.5) { };
		\node[w_vertex] (7) at (10.5,-0.5) { };
		
		\foreach \from/\to in {1/2,1/3,2/4,2/5,3/6,3/7,4/6,5/7}
	    	\draw (\from) -- (\to);
	    	
	    	\draw (5) to[out=60](11.5,1.25) to (12.3,0.6) to[out=-60,in=60] (6);
	    	
	    	\coordinate [label=center:$G_2$] (G2) at (11.25,-1.75);

		\node[w_vertex] (1) at (1.75,-3) { }; 	
		\node[w_vertex] (2) at (1.75,-3.8) { };
		\node[w_vertex] (3) at (1.25,-4.5) { };
		\node[w_vertex] (4) at (2.5,-3.5) { };
		\node[w_vertex] (5) at (1,-3.5) { };
		\node[w_vertex] (6) at (2.25,-4.5) { };

		\foreach \from/\to in {1/4,1/5,2/4,2/5,3/6,3/5,4/6}
	    	\draw (\from) -- (\to);
	    	
	    	\coordinate [label=center:$G_3$] (G3) at (1.75,-5.25);	
	    	
		\node[w_vertex] (1) at (5.75,-3.5) { }; 	
		\node[w_vertex] (2) at (5.75,-2.5) { };
		\node[w_vertex] (3) at (5.75,-4.5) { };
		\node[w_vertex] (4) at (6.5,-3) { };
		\node[w_vertex] (5) at (5,-3) { };
		\node[w_vertex] (6) at (6.5,-4) { };
		\node[w_vertex] (7) at (5,-4) { };
		
		\foreach \from/\to in {1/4,1/5,2/4,2/5,3/6,3/7,4/6,5/7}
	    	\draw (\from) -- (\to);
	    	
	    	\coordinate [label=center:$G_4$] (G4) at (5.75,-5.25);

		\node[w_vertex] (1) at (8.5,-4.5) { }; 	
		\node[w_vertex] (2) at (9.5,-4.5) { };
		\node[w_vertex] (3) at (10.5,-4.5) { };
		\node[w_vertex] (4) at (8.5,-3.5) { };
		\node[w_vertex] (5) at (9.5,-3.5) { };
		\node[w_vertex] (6) at (10.5,-3.5) { };
		\node[w_vertex] (7) at (9.5,-2.75) { };
		
		\foreach \from/\to in {1/2,2/3,4/5,5/6,1/4,2/5,3/6,4/7,6/7}
	    	\draw (\from) -- (\to);
	    	
	    	\coordinate [label=center:$G_5$] (G5) at (9.5,-5.25); 
	\end{tikzpicture}
	\caption{Known minimal non unit disk graphs}
	\label{fig:knownMinForb}
\end{figure}
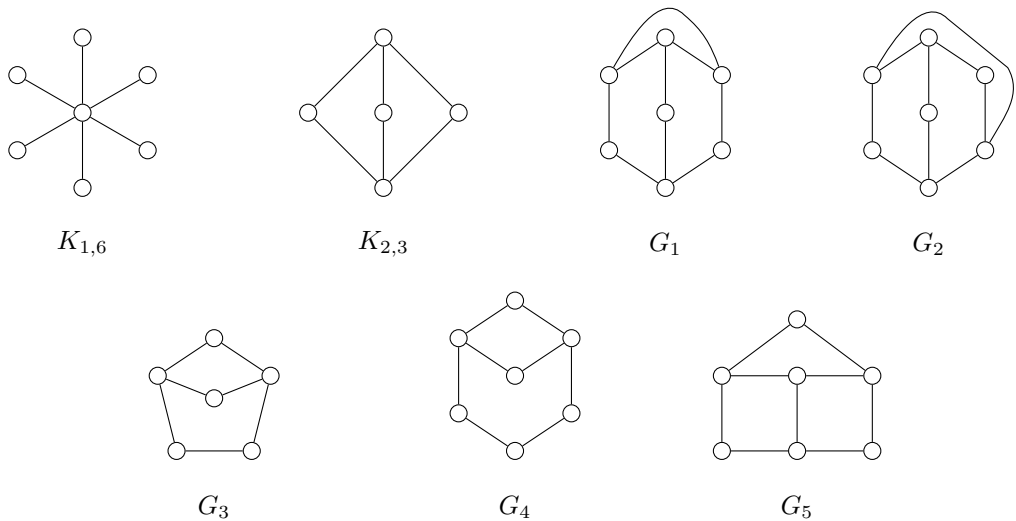


\section{Preliminaries}\label{sec:term}

Let $(V,E)$ denote a graph with vertex set $V$ and edge set $E$. An edge connecting
vertices $u$ and $v$ is denoted $uv$.
For a graph $G$ by $V(G)$ and $E(G)$ we denote the vertex set and the edge set
of $G$, respectively. The complement of a graph $G$ is denoted as $\overline{G}$.
For a vertex $v$  and a set $A \subseteq V(G)$, $N(v)$ denotes the set
of neighbours of $v$, and $N_A(v) = N(v) \cap A$.
Given a subset $A \subseteq V(G)$, $G[A]$ denotes the subgraph
of $G$ induced by $A$, and $G \setminus A$ denotes a graph 
obtained from $G$ by removing vertices in $A$. 
If $A = \{v\}$, then we omit braces and write $G \setminus v$.
A vertex of a graph $G$ is \textit{pendant} if it has exactly one neighbour in $G$.
A set of pairwise non-adjacent vertices in a graph is called an
\textit{independent set}, and a set of pairwise adjacent vertices is a \textit{clique}.
A graph is \textit{bipartite} if its vertex set can be partitioned into two independent sets.
By $(U,W,E)$ we denote a bipartite graph with fixed partition of its vertex set into two
independent sets $U$ and $W$, and edge set $E$.
A graph is \textit{co-bipartite} if its vertex set can be partitioned into two cliques. 
By $(U,W,E)_c$ we denote a co-bipartite graph with fixed partition of its vertex set into two
cliques $U$ and $W$, and set $E$ of edges connecting vertices in different parts of the graph.
Let $G$ be a bipartite graph $(U,W,E)$ (a co-bipartite graph $(U,W,E)_c$, respectively) with
fixed bipartition $U \cup W$, then by $\overline{G^b}$ we denote the \textit{bipartite complement} 
of $G$, that is the bipartite graph $(U,W, (U \times W) \setminus E)$ 
(the co-bipartite graph $(U,W, (U \times W) \setminus E)_c$, respectively).  
Also by $G^*$ we denote the graph obtained from $G$ 
by complementing its subgraphs $G[U]$ and $G[W]$, i.e. $G^* = (U,W,E)_c$ ($G^* = (U,W,E)$, respectively).
As usual, $K_n$, $P_n$ and $C_n$ denote a complete $n$-vertex graph, a chordless path on $n$ 
vertices and a chordless cycle on $n$ vertices, respectively.

A graph $G=(V,E)$ is a \textit{unit disk graph} (UDG for short) if there exists a function 
$f : V \rightarrow \mathbb{R}^2$ such that $uv \in E$ if and only if 
$\dist(f(u),f(v)) \leq 1$, where $\dist(a,b)$ is the Euclidean distance between two
points $a,b \in \mathbb{R}^2$. Function $f$ is called a \textit{UDG-representation} 
(or simply \textit{representation}) of $G$. 
For two vertices $u,v \in V(G)$ the distance $\dist(f(u),f(v))$ between the images
of $u$ and $v$ under a representation $f$ is denoted $\dist_f(u,v)$, or simply $\dist(u,v)$, 
when the context is clear. For a set of vertices $U \subseteq V(G)$, $f(U)$ denotes the set
of images of vertices in $U$, i.e. $f(U) = \{ f(u) : u \in U \}$.

Let $S$ be a finite set of points in $\mathbb{R}^2$. By $\textup{Conv}(S)$ we denote 
the convex hull of $S$. A point $x \in S$ that does not belong to the convex hull
$\textup{Conv}(S \setminus \{x\})$ is called an \textit{extreme point} 
of $\textup{Conv}(S)$. For two distinct points $a,b \in \mathbb{R}^2$ we denote by $L(a,b)$
the line through the points and by $[a,b]$ the line 
segment joining $a$ and $b$. 
The distance between two parallel lines $L_1$ and $L_2$ is denoted by $\dist(L_1,L_2)$.
We say that two line segments $[a,b]$ and $[c,d]$
\textit{cross} if their intersection consists of a single point different from $a,b,c$ and $d$.
For three non-collinear points $a,b,c$ the triangle with vertices $a,b,c$ is denoted 
by $\triangle abc$, and $\angle abc$ denotes the angle between sides $[a,b]$ and $[b,c]$
of the triangle.
We will denote a point in Cartesian coordinate system as $(x, y)$, and in polar as $(r,\alpha)_p$ such that
$(r, \alpha)_p=(r \sin(\alpha), r \cos(\alpha))$.

In Sections \ref{ss:C4rep}-\ref{ss:2K2rep} dealing with UDG-representations we will make frequent 
use of following basic inequalities and equations:
\begin{align}
1-\frac{x}{2}-\frac{x^2}{2} \leq  \sqrt{1-x}  \leq   1-\frac{x}{2} \label{eq1} \\
 x - \frac{x^3}{6} \leq \sin(x) \leq  x \label{eq2} \\
\cos(2\beta)=\cos^2(\beta)-\sin^2(\beta) \label{eq3}\\
\sin(2\beta)=2\sin(\beta)\cos(\beta) \label{eq4}\\
\dist(a,b)^2+\dist(b,c)^2-2\cos(\angle abc)\dist(a,b)\dist(b,c)=\dist(a,c)^2 \label{eq5}
\end{align} 

\noindent
The inequalities (1) and (2) hold for all $x \in [-1,1]$ and $x \geq 0$, respectively. Both are coming from truncated Taylor series expansions, but one can also find direct proofs of these facts, by squaring (1) and considering derivatives in (2). The equations (3) and (4) are standard facts and hold for all $\beta \in \mathbb{R}$. The equation (5) is known as the Law of cosines and holds for any triangle $abc$.


\section{Tools}\label{sec:tools}

In this section we develop several geometric and structural tools which are helpful in
further sections, though may be of their own interest.


\subsection{Basic tools}

We use the following obvious claim.
\begin{claim}\label{cl:triangle}
	Let $a,b,c \in \mathbb{R}^2$ be three non-collinear points such that 
	$\dist(a,b) \leq 1$ and $\dist(a,c) \leq 1$. Then $\dist(a,d) \leq 1$ for every point 
	$d \in \triangle abc$.
\end{claim}

Informally, the following lemma says that any UDG-representation of a $C_4$
is a convex quadrilateral with sides corresponding to the edges of the $C_4$.

\begin{lemma}[Convexity of $C_4$]\label{lem:convC4}
	Let $G=(V,E)$ be a UDG and let a subset $\{ v_1,v_2,v_3,v_4 \} \subseteq V$ 
	induces a $C_4$ in $G$ such that 
	$\{ v_1v_2, v_2v_3, v_3v_4, v_4v_1 \} \subseteq E$. Then
	for any representation $f$ of the graph, 
	$\textup{Conv}(p_1, p_2, p_3, p_4)$ 
	is a quadrangle, and $[p_1,p_3]$ and $[p_2,p_4]$ cross, where $p_i = f(v_i)$, $i=1, \ldots, 4$.
\end{lemma}
\begin{proof}
	
	First, let us show that no three points in $S = \{ p_1, p_2, p_3, p_4 \}$
	are collinear, i.e. no three points in $S$ lie on the same line. Indeed, 
	assume, that $p_1, p_2$ and $p_3$ lie on the same line. As $v_1v_2$ and $v_2v_3$ are edges 
	of $G$ and $v_1v_3$ is a non-edge, we know that $\dist(p_1, p_2) \leq 1$ and $\dist(p_2, p_3) \leq 1$,
  	while $\dist(p_1,p_3)>1$. From this it follows, that $p_2$ must lie between $p_1$ and $p_3$, and hence, in 
	particular, belongs to the triangle $\triangle p_4 p_1 p_3$. Since $v_4$ is adjacent to $v_1$ and $v_3$, 
	we have $\dist(p_4, p_1) \leq 1$ and $\dist(p_4, p_3) \leq 1$. Hence, Claim~\ref{cl:triangle} now 
	applies to triangle $\triangle p_4 p_1 p_3$ and we deduce that $\dist(p_4, p_2) \leq 1$. But this 
	contradicts the assumption that $v_2v_4$ is a non-edge. By symmetry the same conclusion follows for 
	the other three tripples of points from $S$.

	Suppose now that $\textup{Conv}(S)$ is a triangle. Without loss of generality
	let $p_1,p_2,p_3$ be the extreme points of the triangle. As $v_1v_2, v_2v_3$ are edges of $G$, 
	we have $\dist(p_2, p_1) \leq 1$ and $\dist(p_2,p_3) \leq 1$. By Claim \ref{cl:triangle} applied to 
	triangle $\Delta p_2 p_1 p_3$, we deduce that $\dist(p_2, p_4) \leq 1$. But this contradicts the
	the assumption that $v_2v_4$ is a non-edge. 
	
	Finally, suppose that $\textup{Conv}(S)$ is a quadrangle and $[p_1, p_3]$ 
	and $[p_2, p_4]$ do not cross, i.e. these segments are two opposite sides of the quadrangle.
	As these segments have both length greater than 1,  we'll show that this implies that one of the 
	diagonals of the quadrangle must be of size greater than 1 as well and hence a contradiction. 
	Consider the case when $[p_1, p_4]$, $[p_2, p_3]$ forms the diagonals
 	of the quadrilateral and crosses at some point 	$q$. Without loss of generality, let 
	$\dist(q, p_3) \leq \dist(q, p_4)$. 
	By triangle inequality 
	$$
		1 < \dist(p_1, p_3) \leq \dist(p_1,q) + 
		\dist(q, p_3) \leq \dist(p_1, q) + \dist(q, p_4) = 
		\dist(p_1, p_4) \leq 1,
	$$
	a contradiction. Similarly, we arrive at a contradiction if we assume that the diagonals
	 of the quadrangle are $[p_1,p_2]$ and $[p_3,p_4]$. These contradictions prove 
	that $[p_1, p_3]$ and $[p_2, p_4]$ must cross and finish the proof of the lemma.

\end{proof}

\begin{corollary}\label{cor:C4same_side}
	Let $G=(V,E)$ be a UDG and let a subset $\{ v_1,v_2,v_3,v_4 \} \subseteq V$ 
	induces a $C_4$ in $G$ such that 
	$\{ v_1v_2, v_2v_3, v_3v_4, v_4v_1 \} \subseteq E$. Then
	for any representation $f$ of the graph, $p_3$ and $p_4$ lie on the same side 
	of the line $L(p_1, p_2)$, where $p_i = f(v_i)$, $i=1, \ldots, 4$.
\end{corollary}

When we deal with UDG-representations of complements of  graphs the following form of Lemma
\ref{lem:convC4} is more convenient.

\begin{lemma}\label{lem:2K2cross}
	Let $G=(V,E)$ be a graph and vertices $v_1,v_2,v_3,v_4$ induce
	$2K_2$ in $G$ with edges $v_1v_3, v_2v_4 \in E$.
	If $\overline{G}$ is UDG, then for any representation $f$ of $\overline{G}$,
	$\textup{Conv}( p_1, p_2, p_3, p_4 )$ is a quadrangle and 
	$[p_1,p_3]$ and $[p_2,p_4]$ cross, where $p_i = f(v_i)$, $i=1, \ldots, 4$.
\end{lemma}

\begin{lemma}\label{lem:coP6}
	Let $G=(V,E)$ be a graph and let $\{ v_1, v_2, v_3, v_4, v_5, v_6 \} \subseteq V$
	induces a $P_6$ in $G$ with edges $v_iv_{i+1} \in E$ for $i = 1, \ldots, 5$.
	If $\overline{G}$ is a UDG then for any representation $f$ of $\overline{G}$
	convex hull $\textup{Conv}(p_2, p_3, p_4, p_5)$ is a quadrangle, and 
	$[p_2, p_3]$ and $[p_4, p_5]$ cross, where $p_i = f(v_i)$, $i=1, \ldots, 6$.
\end{lemma}
\begin{proof}
	First, let us note that neither $p_3$ nor $p_4$ lies on line $L = L(p_2,p_5)$. 
	Indeed, suppose $p_4$ lies on $L$, then $\textup{Conv}(p_1, p_2, p_4, p_5)$ is not 
	a quadrangle. However, it should be a quadrangle by Lemma \ref{lem:2K2cross}, 
	as $\{v_1, v_2, v_4, v_5\}$ induces a $2K_2$ in $G$. This contradiction proves that $p_4$
	does not belong to the line $L$. By symmetry the same conclusion holds for $p_3$.

	
	Further, we claim that $p_3$ and $p_4$ are on the same side of 
	$L$. Suppose to the contrary, $L$ separates $p_3$ and $p_4$.
	By Lemma \ref{lem:2K2cross}, $[p_5, p_6]$ crosses $[p_2, p_3]$,
	hence we deduce that $p_6$ must lie on the same side of $L$ as $p_3$ (see Figure \ref{fig:coP6_fig1}).
	Also, by Lemma \ref{lem:2K2cross}, $[p_1,p_2]$ crosses $[p_4,p_5]$,
	hence, $p_1$ must be on the same side of $L$ as $p_4$. 
	From this we deduce that $p_1$ and $p_6$ are separated by $L$
 	and hence $[p_1,p_2]$ and $[p_5,p_6]$ lie in different half-planes and do not cross. 
	The latter is impossible,
	since $[p_1,p_2]$ and $[p_5, p_6]$ cross by Lemma \ref{lem:2K2cross}.

	Let $S = \{ p_2, p_3, p_4, p_5 \}$ and suppose that $\textup{Conv}(S)$ 
	is a triangle.
	Since $p_3$ and $p_4$ are on the same side of $L$, either $p_3$ or 
	$p_4$ is not an extreme point of $\textup{Conv}(S)$. Without loss of generality, 
	assume $p_3$ is not an extreme point of $\textup{Conv}(S)$ (see Figure \ref{fig:coP6_fig2}). 
	Since $\dist(p_2, p_5) \leq 1$ 
	and $\dist(p_2, p_4) \leq 1$, by Claim \ref{cl:triangle} 
	we obtain $\dist(p_2, p_3) \leq 1$. This is a contradiction as $v_2v_3$ is a non-edge in $\overline{G}$.
	This shows that $\textup{Conv}(S)$ is a quadrangle.
	
	\begin{figure}[H]

       	\begin{subfigure}[b]{0.5\textwidth}
       	       	\centering
             		\includegraphics[scale=0.65]{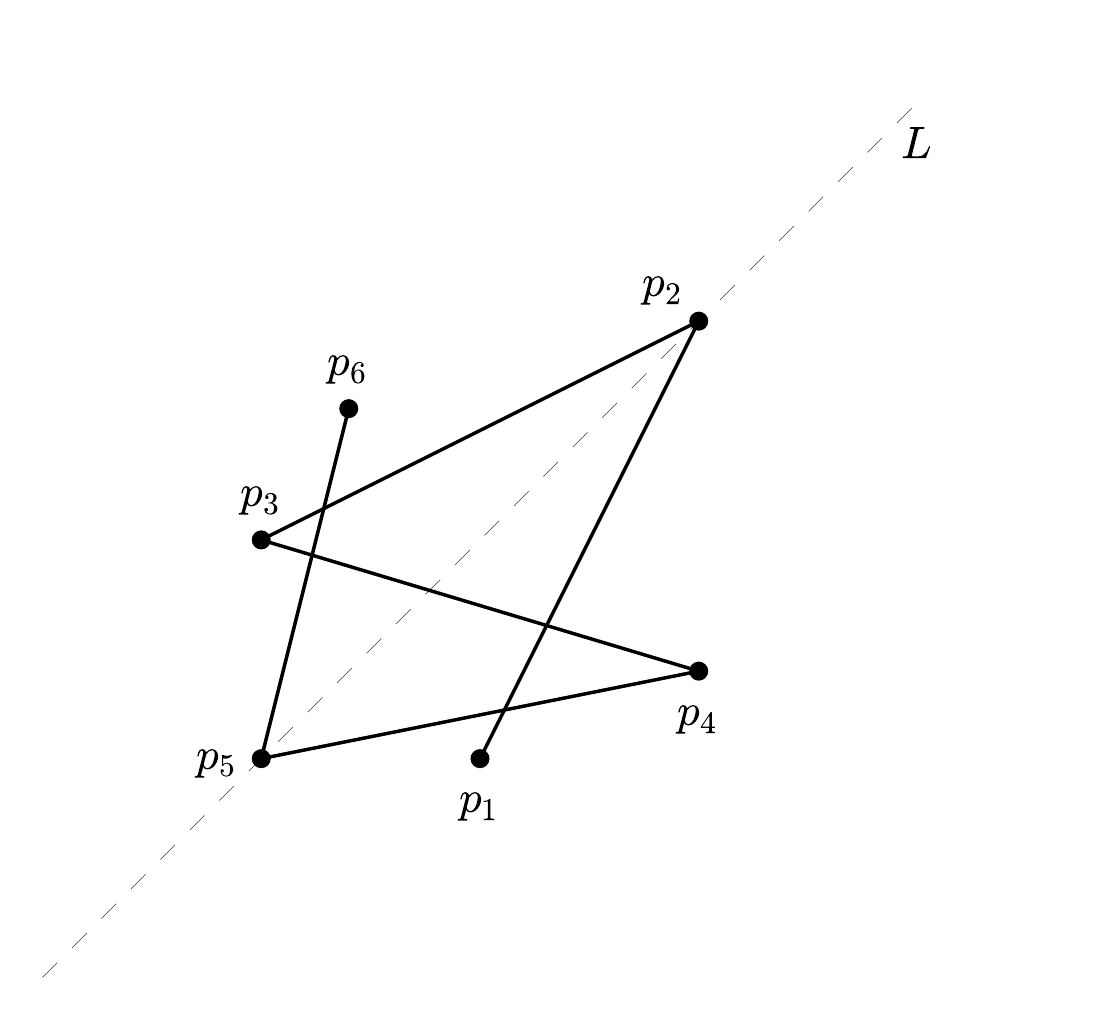}
                	\caption{}
                	\label{fig:coP6_fig1}
        	\end{subfigure}%
        ~ 
        	\begin{subfigure}[b]{0.5\textwidth}
        	       	\centering
              	\includegraphics[scale=0.65]{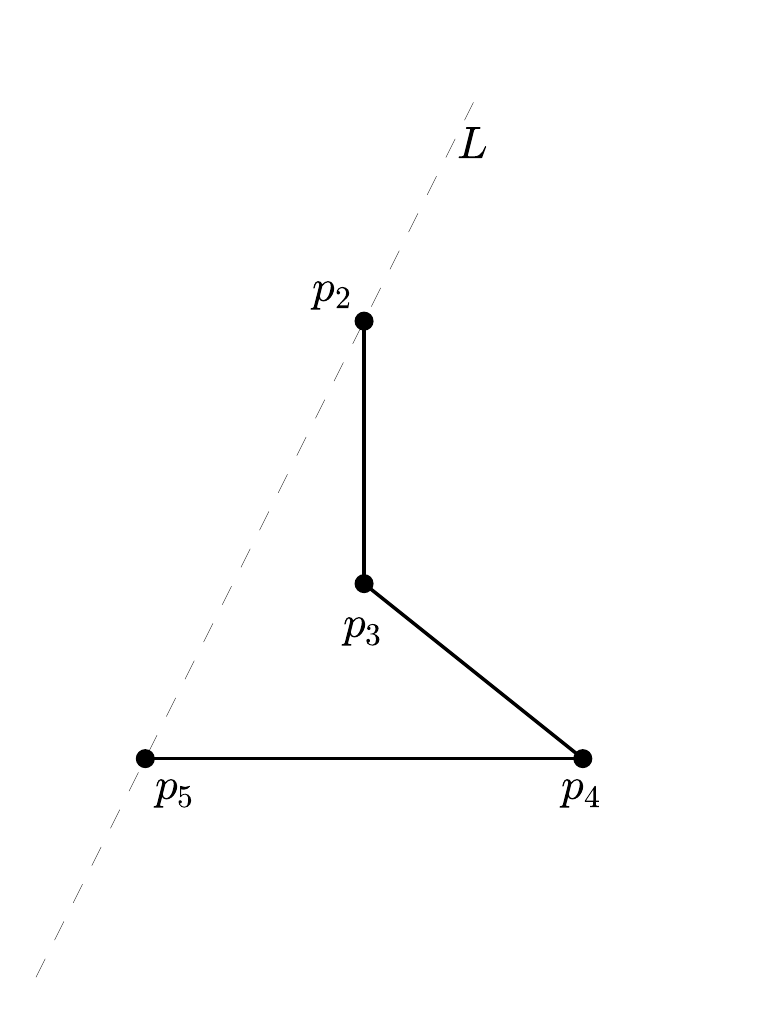}
                	\caption{}
                	\label{fig:coP6_fig2}
        	\end{subfigure}

        	\caption{}\label{fig:coP6}
	\end{figure}
	
	Finally, suppose that $\textup{Conv}(S)$ is a quadrangle, but
	$[p_2, p_3]$ and $[p_4,p_5]$ do not cross. 
	Since $p_3$ and $p_4$ are on 
	the same side of $L$, $[p_2,p_4]$ crosses $[p_3,p_5]$. Let $q$ be the 
	crossing point of these intervals.
	Without loss of generality, assume $\dist(p_3, q) \geq \dist(p_4, q)$. 
	Then
	$$
		1 < \dist(p_4, p_5) \leq \dist(p_4, q) + \dist(q, p_5) \leq 
		\dist(p_3, q) + \dist(q, p_5) = \dist(p_3, p_5) \leq 1,
	$$
	a contradiction. This finishes the proof of the lemma.
\end{proof}


\subsection{Edge-asteroid triples}

A set of three edges in a graph is called an {\it edge-asteroid triple} if for each pair of 
the edges, there is a path containing both of the edges that avoids the 
neighbourhoods of the end-vertices of the third edge.

\begin{lemma}\label{lem:co-eat}
	Let $G=(U, W, E)_c$ be a co-bipartite UDG.
	Then $\overline{G}$ contains no edge-asteroid triples.
\end{lemma}
\begin{proof}
	Let $f$ be a representation of the unit disk graph $G$, and for $v \in V(G)$ let $p_v = f(v)$.
	Suppose to the contrary that $\overline{G}$ contains an edge-asteroid triple
	$\{ e_1, e_2, e_3 \} \subset E$. Denote by 
	$u_i$ and $w_i$ the end-vertices of $e_i$, where $u_i \in U$, $w_i \in W$,
	$i \in \{ 1, 2, 3 \}$. 
	For distinct $i,j,k \in \{1,2,3\}$, let $P_i$ be a path in $\overline{G}$ that avoids 
	the neighbourhood of $u_i$ and the neighbourhood $w_i$, and whose terminal 
	edges are $e_j$ and $e_k$.
	By Lemma \ref{lem:2K2cross} the interval corresponding to an edge of $P_i$ 
	crosses $[p_{u_i},p_{w_i}]$. 
	Since $\overline{G}$ is bipartite, this implies that the images of the vertices in $V(P_i) \cap U$ 
	lie on one side of $L_i = L(p_{u_i},p_{w_i})$ and the images of the vertices in 
	$V(P_i) \cap W$ lie on the other side of $L_i$.
	In particular, 
	$p_{u_j}$ and $p_{u_k}$ lie on one side of $L_i$ and $p_{w_j}$ and $p_{w_k}$ lie 
	on the other side.
	
	On the other hand, since, by Lemma \ref{lem:2K2cross}, the intervals 
	corresponding to $e_1,e_2,e_3$ pairwise cross, there exists
	$i \in \{ 1,2,3 \}$ such that $p_{u_j}$ and $p_{w_k}$ are on the same side of $L_i$. Indeed,
	if, say, $p_{u_1}$ and $p_{u_2}$ lie on the same side of $L_3$ and $p_{w_1}$ and 
	$p_{w_2}$ lie on the other side, then necessarily either $L_1$ has 
	$p_{u_2}$ and $p_{w_3}$ on one of its sides or $L_2$ has $p_{u_1}$ and $p_{w_3}$
	on one of its sides (see Figure \ref{fig:coEAT_fig1}). This contradiction establishes the lemma. 
	
	\begin{figure}[H]
    		\begin{subfigure}[b]{0.47\textwidth}
    			\centering
   			\includegraphics[scale=0.65]{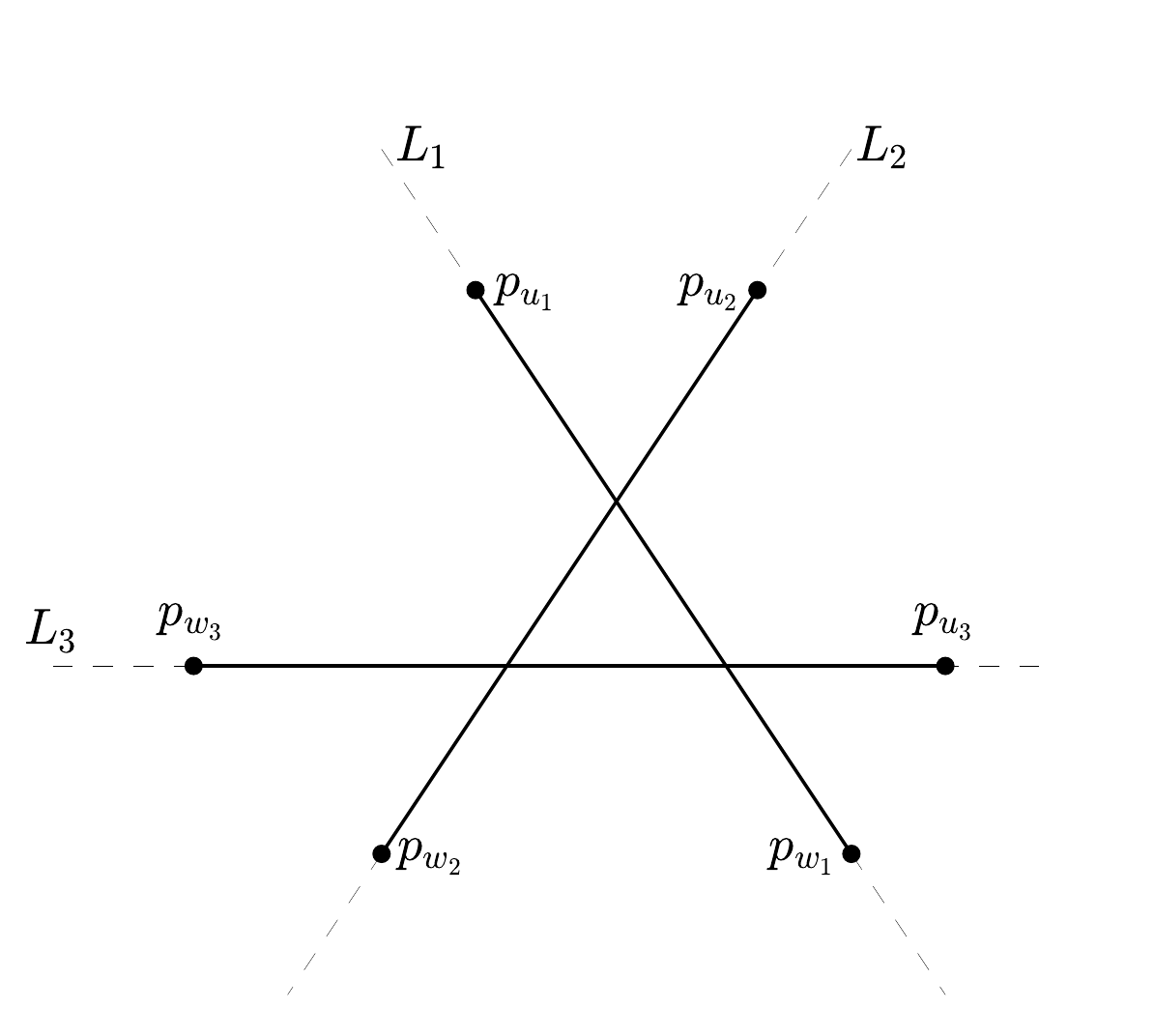}
    			\caption{}
    			\label{fig:coEAT_fig1}
        	\end{subfigure}%
        	~
        	\begin{subfigure}[b]{0.53\textwidth}
        		\centering
   			\includegraphics[scale=0.65]{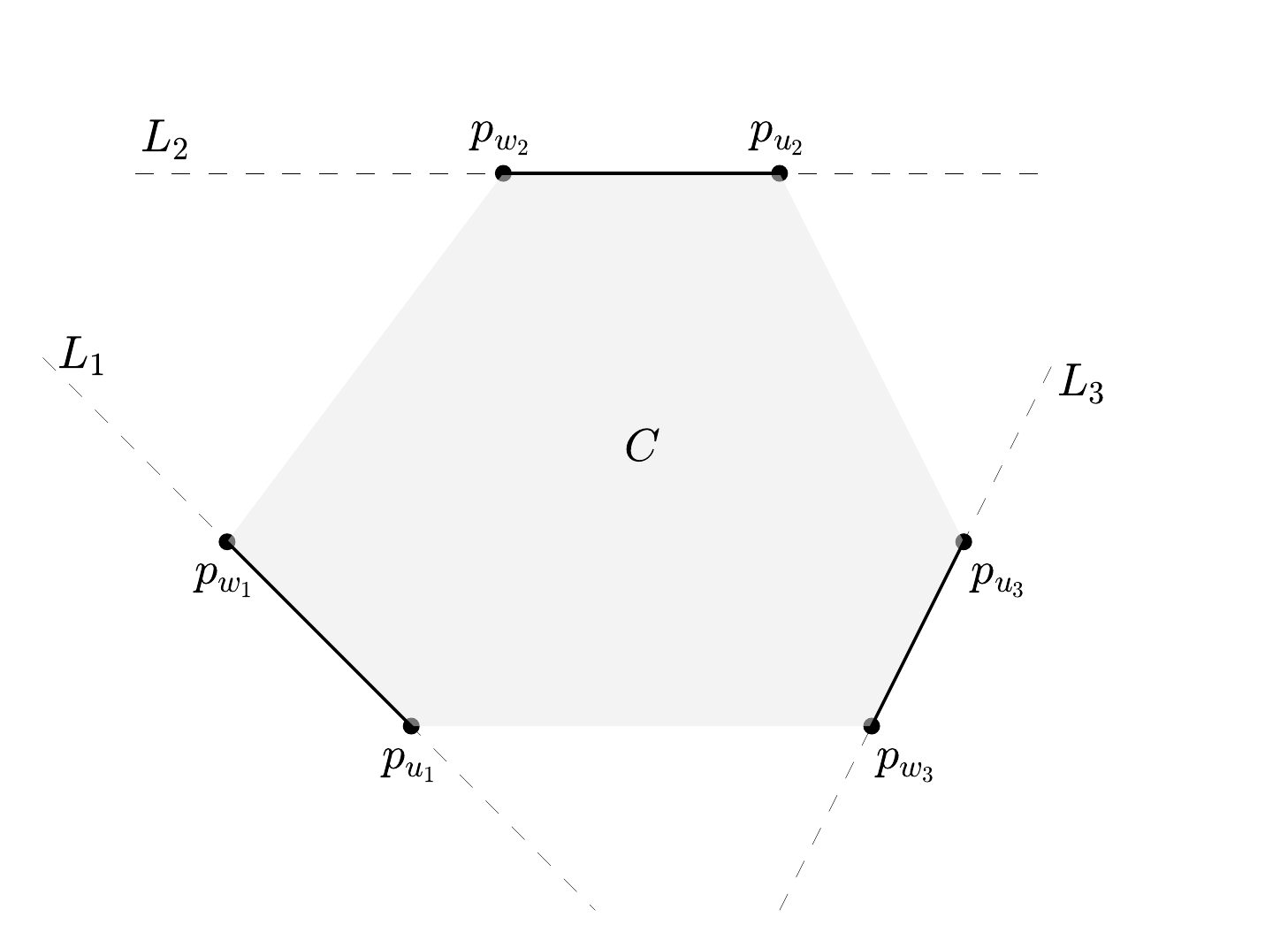}
    			\caption{}
    			\label{fig:starEAT_fig1}
		\end{subfigure}%
		
		\caption{}
	\end{figure}
\end{proof} 

\begin{lemma}\label{lem:star-eat}
	Let $G=(U,W,E)$ be a bipartite graph. If co-bipartite graph $G^*=(U,W,E)_c$ is UDG, 
	then $G$ contains no edge-asteroid triples.
\end{lemma}
\begin{proof}
	Let $f$ be a representation of unit disk graph $G^*$, and for $v \in V(G^*)$ let $p_v = f(v)$.
	Suppose to the contrary that $G$ contains an edge-asteroid triple
	$\{ e_1, e_2, e_3 \} \in E$. Denote by 
	$u_i$ and $w_i$ the end-vertices of $e_i$, where $u_i \in U$, $w_i \in W$,
	$i \in \{ 1, 2, 3 \}$. For distinct $i,j,k \in \{1,2,3\}$, let $P_i$ be a path in $G$ 
	that avoids the neighbourhoods of $u_i$ and $w_i$, and whose terminal edges 
	are $e_j$ and $e_k$.
	Corollary \ref{cor:C4same_side} implies that for 
	every edge 
	$vu$ of $P_i$ both $p_v$ and $p_u$ lie on the same side of 
	$L_i = L(p_{u_i},p_{w_i})$. 
	Therefore all the images of the vertices of $P_i$ lie on the same side of $L_i$.
	In particular, 
	$p_{u_j},p_{w_j},p_{u_k}$ and $p_{w_k}$ 
	lie on the same side of $L_i$. The latter fact means that 
	$p_{u_1},p_{w_1},p_{u_2},p_{w_2},p_{u_3},p_{w_3}$ are extreme
	points of $C = \textup{Conv}(p_{u_1},p_{w_1},p_{u_2},p_{w_2},p_{u_3},p_{w_3})$ 
	and for every $i \in \{1,2,3\}$ $p_{u_i}$ and $p_{w_i}$ are adjacent extreme points of the 
	convex hull (see Figure \ref{fig:starEAT_fig1}).
	
	
	Now we'll show that $p_{u_i}$ and $p_{w_j}$ for $j \neq i$ cannot be adjacent extreme points 
	of the convex hull. Indeed, assume for contradiction, $p_{u_i}$ is adjacent to $p_{w_j}$ for $j \neq i$.
	Then, as we proved above, $p_{w_i}, p_{u_i}, p_{w_j}, p_{u_j}$ must be a sequence of consecutive
	extreme points in the convex hull. However, $\{w_i, u_i, w_j, u_j\}$ forms a $C_4$ in $G^*$ and 
 	by Lemma~\ref{lem:convC4}, $[p_{w_i}, p_{u_j}]$ must be crossing $[p_{w_j}, p_{u_i}]$, a contradiction. 
	Hence, we deduce, that $p_{u_i}$ is adjacent to $p_{w_j}$ if and only if $i=j$.

	Now assume, without loss of generality, that $p_{w_1}$ is adjacent to $p_{w_2}$ in 
	$C$. This gives us a sequence of extremal points in the convex hull $p_{u_1}, p_{w_1}, p_{w_2}, p_{u_2}$.  
	But then $p_{w_3}$ is adjacent to either $p_{u_1}$ or to $p_{u_2}$ in $C$ 
	(see Figure \ref{fig:starEAT_fig1}), a contradiction.
	
\end{proof}


\section{Minimal forbidden induced subgraphs}\label{sec:forb}

\begin{theorem}\label{th:K2C2k+1}
	For every integer $k \geq 1$, $\overline{K_2 + C_{2k+1}}$ is a minimal non-UDG.
\end{theorem}
\begin{proof}
	
	Let $G=(V,E)$ be a graph
	isomorphic to $K_2 + C_{2k+1}$, where $V=\{ u,w, c_1, \ldots, c_{2k+1} \}$ and
	$E = \{ c_ic_j : |i-j| = 1 \} \cup \{ uw, c_1c_{2k+1} \} $.
	Suppose to the contrary $\overline{G}$ is a UDG and let $f$ be a representation of $\overline{G}$,
	and let $p_v$ denotes $f(v)$ for $v \in V$. 
	By Lemma \ref{lem:2K2cross} every linear interval corresponding to an
	edge of the cycle $C_{2k+1}$ crosses $[p_u,p_w]$. That means that the vertices of the cycle are partitioned
	into two parts, according to the side of line $L(p_u,p_w)$ the image of a vertex belongs to. 
	Moreover, there are no edges between vertices in the same part.
	This leads to the contradictory conclusion that $C_{2k+1}$ is a bipartite graph.
	
	To prove the minimality of the graphs it is sufficient to show that $\overline{K_1+C_{2k+1}}$ 
	is a UDG for
	any natural $k$. Indeed, notice that by removing a vertex from $\overline{K_2+C_{2k+1}}$ we 
	get a graph which is either $\overline{K_1+C_{2k+1}}$ or $\overline{K_2+P_{2k}}$. The latter one is, in turn,
	an induced subgraph of $\overline{K_1 + C_{2k+5}}$.
	To show that $\overline{K_1+C_{2k+1}}$ is a UDG, put $2k+1$ points 
	$p_0,p_1, \ldots, p_{2k}$ equally spaced on 
	the circle of radius $r$, i.e. in polar coordinates these points can be written as 
	$(r, 0)_p, (r, \frac{2\pi}{2k+1})_p, (r, 2\frac{2\pi}{2k+1})_p, \ldots, (r, 2k\frac{2\pi}{2k+1})_p$.
	We also add one point $p_c$ at the 
	center (0,0). Choose the radius $r$ of the circle such that 
	the distance between $p_0$ and $p_k$, and between $p_0$ and 
	$p_{k+1}$ is greater than 1, and the distances between $p_0$ and 
	the other points 
	is at most 1.
	It is easy to see that the UDG represented by these points is $\overline{K_1+C_{2k+1}}$. 
	See Figure \ref{fig:K2C2k_1_fig1} for an example of the representation of $\overline{K_1+C_{7}}$.
	
	\begin{figure}[H]
    		\centering
   		\includegraphics[scale=0.75]{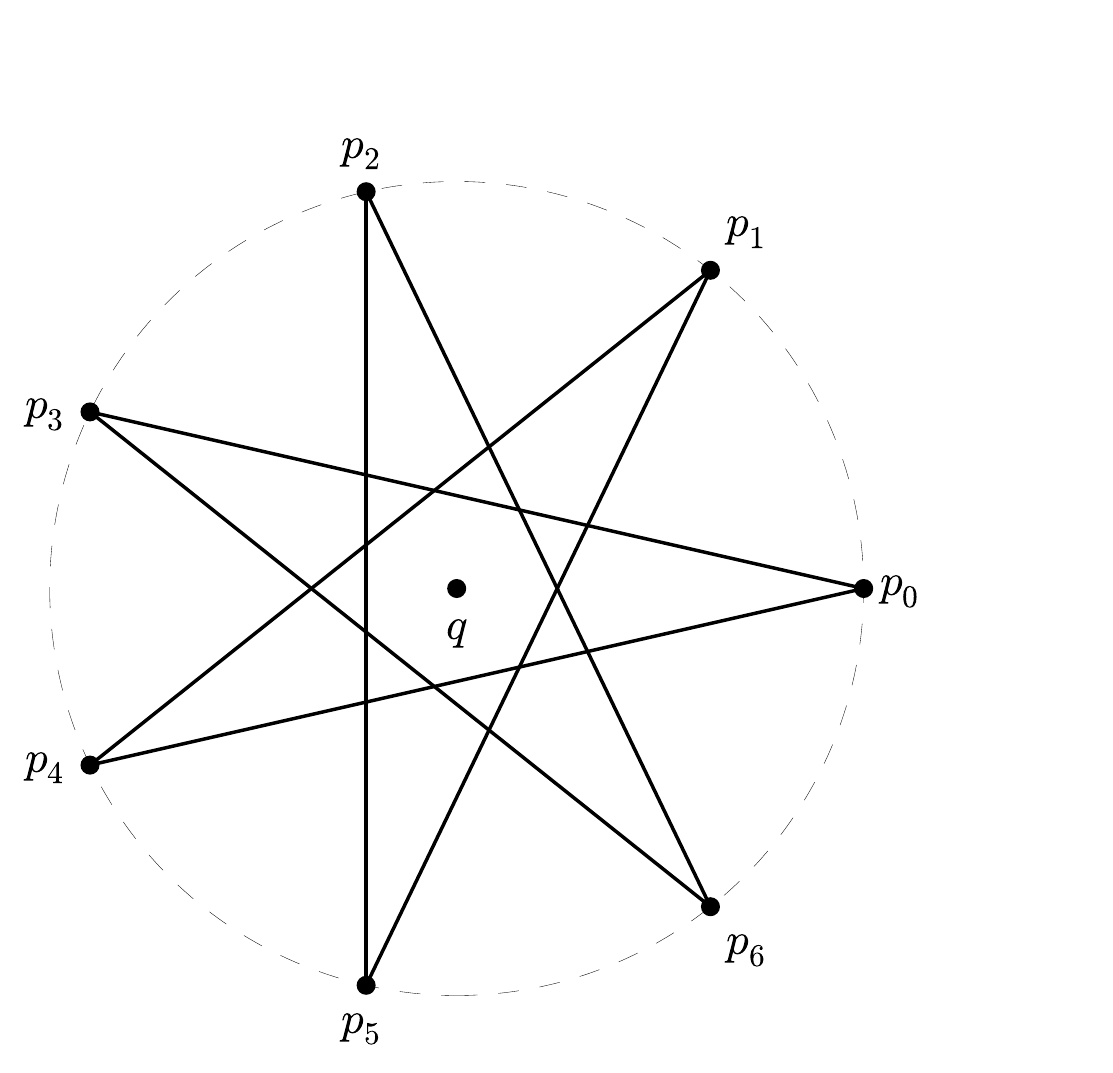}
    		\caption{The UDG-representation of $\overline{K_1+C_{7}}$}
    		\label{fig:K2C2k_1_fig1}
	\end{figure}
	
\end{proof}

\begin{corollary}\label{col:Pk}
	For every integer $k \geq 1$, $\overline{P_k}$ is UDG.
\end{corollary}

\begin{theorem}\label{minimalevencycles}
	For every integer $k \geq 4$, $\overline{C_{2k}}$ is a minimal non-UDG.
\end{theorem}
\begin{proof}
	Note that by removing a vertex from $\overline{C_{2k}}$ we get $\overline{P_{2k-1}}$,
	which is UDG by Corollary \ref{col:Pk}. Therefore it remains to show that
	$\overline{C_{2k}}$ is not UDG.
	 For $k \geq 5$ the desired result immediately follows from 
	Lemma \ref{lem:co-eat} and the fact that $C_{2k}$ contains an edge-asteroid triple. 
	To prove the result for $k=4$, consider $G=(V,E)$ with $V = \{ v_1, \ldots, v_8 \}$ and
	$E = \{ (v_1,v_8) \} \cup \{ (v_i,v_j) : |i-j| = 1\}$, and let $f$ be a representation of $\overline{G}$,
	and let $p_v$ denotes $f(v)$, as before.
	By Lemma \ref{lem:2K2cross} the linear interval corresponding to an edge of $G$,
	different from $v_1v_2$, $v_1v_8$ and $v_2v_3$,
	crosses $[p_{v_1},p_{v_2}]$. This leads to the conclusion that $p_{v_3}$ and $p_{v_8}$
	are on different sides of $L(p_{v_1},p_{v_2})$. Therefore $[p_{v_1},p_{v_8}]$ and
	$[p_{v_2},p_{v_3}]$ do not cross, which contradicts Lemma \ref{lem:coP6}.
\end{proof}

\begin{theorem}\label{th:starC2k}
	For every integer $k \geq 4$, $C_{2k}^*$ is a minimal non-UDG.
\end{theorem}
\begin{proof}
	For $k \geq 5$ the theorem immediately follows from 
	Lemma \ref{lem:star-eat} and the fact that $C_{2k}$ contains an edge-asteroid triple.
	Notice that $C_8^*=\overline{C_8}$ and hence the conclusion follows 
	from Theorem~\ref{minimalevencycles}.
	We  remark that one can also prove that $C_8^*$ is not a unit disk graph by 
	similar means as in Theorem~\ref{minimalevencycles} by proving a *-analog 
	of Lemma~\ref{lem:coP6}.  
	
	To prove the minimality of $C_{2k}^*$ it is sufficient to show
	that $P_{s}^*$ is UDG for every natural $s$. Such a representation could be seen
	in the Figure~\ref{fig:simple_lobster_star_repr} with a description in Theorem~\ref{basiclobster}. 
\end{proof}

Using Lemmas \ref{lem:co-eat} and \ref{lem:star-eat} one can find more forbidden
(not necessarily minimal) induced subgraphs for the class of unit disk graphs. 
For example, 
$\overline{S_{3,3,3}}$, $\overline{F_1}$, $\overline{F_2}$, $\overline{F_3}$,
$S_{3,3,3}^*$, $F_1^*$, $F_2^*$ and $F_3^*$ are forbidden, since each of 
the graphs $S_{3,3,3}$, $F_1$, $F_2$ and $F_3$ (see Figure \ref{fig:minForbWithEAT}) contains 
an edge-asteroid triple. Also, $\overline{F_4} $ and $F_4^*$ are forbidden, as they coincide with
$F_1^*$ and $\overline{F_1}$, respectively.
The results of the next section imply that all the mentioned forbidden graphs are in fact minimal.

\begin{figure}[H]
	\centering
	\begin{tikzpicture}[scale=.6,auto=left]
		\node[w_vertex] (1) at (1,0) { }; 	
		\node[w_vertex] (2) at (0,-1) { };
		\node[w_vertex] (3) at (1,-1) { };
		\node[w_vertex] (4) at (2,-1) { };
		\node[w_vertex] (5) at (0,-2) { };
		\node[w_vertex] (6) at (1,-2) { };
		\node[w_vertex] (7) at (2,-2) { };
		\node[w_vertex] (8) at (0,-3) { };
		\node[w_vertex] (9) at (1,-3) { };
		\node[w_vertex] (10) at (2,-3) { };
		
		\foreach \from/\to in {1/2,1/3,1/4, 2/5,3/6,4/7}
	    	\draw (\from) -- (\to);
	    	
	    	\foreach \from/\to in {5/8,6/9,7/10}
	    	\draw[very thick] (\from) -- (\to);
	    	
	    	\coordinate [label=center:$S_{3,3,3}$] (S333) at (1,-4);

		\node[w_vertex] (1) at (4,-2) { }; 	
		\node[w_vertex] (2) at (5,-3) { };
		\node[w_vertex] (3) at (6,-2) { };
		\node[w_vertex] (4) at (7,-3) { };
		\node[w_vertex] (5) at (8,-2) { };
		\node[w_vertex] (6) at (7,-1) { };
		\node[w_vertex] (7) at (5,-1) { };
		\node[w_vertex] (8) at (6,0) { };
		\node[w_vertex] (9) at (6,1) { };
		
		\foreach \from/\to in {2/3,3/4,5/6,6/3,3/7,7/1,7/8,8/6}
	    	\draw (\from) -- (\to);
	    	
	       \foreach \from/\to in {1/2,4/5,8/9}
	    	\draw[very thick] (\from) -- (\to);
	    	
	    	\coordinate [label=center:$F_1$] (F1) at (6,-4);

		\node[w_vertex] (1) at (10,-2) { }; 	
		\node[w_vertex] (2) at (11,-1) { };
		\node[w_vertex] (3) at (12,-1) { };
		\node[w_vertex] (4) at (13,-2) { };
		\node[w_vertex] (5) at (12,-3) { };
		\node[w_vertex] (6) at (11,-3) { };
		\node[w_vertex] (7) at (10,-1) { };
		\node[w_vertex] (8) at (14,-2) { };
		\node[w_vertex] (9) at (10,-3) { };
		
		\foreach \from/\to in {1/2,2/3,3/4,4/5,5/6,6/1}
	    	\draw (\from) -- (\to);
	    	
	       \foreach \from/\to in {2/7,4/8,6/9}
	    	\draw[very thick] (\from) -- (\to);
	    	
	    	\coordinate [label=center:$F_2$] (F2) at (11.5,-4);
	    	
		\node[w_vertex] (1) at (16,-2) { }; 	
		\node[w_vertex] (2) at (17,-1) { };
		\node[w_vertex] (3) at (18,-1) { };
		\node[w_vertex] (4) at (19,-2) { };
		\node[w_vertex] (5) at (18,-3) { };
		\node[w_vertex] (6) at (17,-3) { };
		
		\node[w_vertex] (7) at (17,0) { };
		\node[w_vertex] (8) at (17,1) { };
		\node[w_vertex] (9) at (18,0) { };
		\node[w_vertex] (10) at (18,1) { };
		
		\foreach \from/\to in {1/2,2/3,3/4,4/5,6/1,2/7,3/9}
	    	\draw (\from) -- (\to);
	    	
	       \foreach \from/\to in {7/8,9/10,5/6}
	    	\draw[very thick] (\from) -- (\to);
	    	
	    	\coordinate [label=center:$F_3$] (F3) at (17.5,-4);	
	
		\node[w_vertex] (1) at (21,-2) { }; 	
		\node[w_vertex] (2) at (22,-1) { };
		\node[w_vertex] (3) at (23,-1) { };
		\node[w_vertex] (4) at (24,-2) { };
		\node[w_vertex] (5) at (23,-3) { };
		\node[w_vertex] (6) at (22,-3) { };
		
		\node[w_vertex] (7) at (22,0) { };
		\node[w_vertex] (8) at (25, -2) {};	
		\node[w_vertex] (9) at (23,0) { };
		
		\foreach \from/\to in {1/2,2/3,3/4,4/5,6/1,2/7,3/9}
	    	\draw (\from) -- (\to);
	    	
	       \foreach \from/\to in {5/6, 4/8}
	    	\draw (\from) -- (\to);
	    	
	    	\coordinate [label=center:$F_4$] (F4) at (22.5,-4);

	\end{tikzpicture}
	\caption{Bipartite graphs $S_{3,3,3}$, $F_1$, $F_2$ and $F_3$ contain an edge-asteroid triple.
	Graph $F_4$ is the bipartite complementation of $F_1$.}
	\label{fig:minForbWithEAT}
\end{figure}
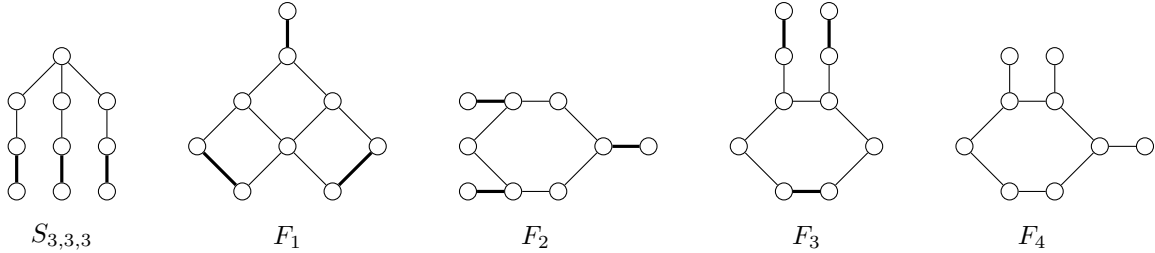


\section{Structure of some subclasses of co-bipartite unit disk graphs}
\label{sec:structure}


For easier reference, let $C_6^{+3c}$ denotes $F_4$ which reads 
``cycle on 6 vertices plus 3 consecutive (pendant) vertices'', 
$C_6^{+3nc}$ denotes $F_2$ which reads ``cycle on 6 vertices plus 3 non-consecutive (pendant) vertices''
and let $C_6^{+2l2}$ denotes $F_3$ which reads 
``cycle on 6 vertices plus 2 (consecutive pendant) paths of length 2''. 
It follows from the previous section that for every co-bipartite unit disk graph $G=(U,W,E)_c$, both
$G^*$ and $\overline{G}$ lie in the class 
$
	Free(S_{3,3,3}, C_3, C_5, C_6^{+3c}, C_6^{+3nc}, C_6^{+2l2}, C_7, C_8, \ldots ),
$
i.e. the class of bipartite graphs which do not contain $S_{3,3,3}, C_6^{+3c}, C_6^{+3nc}, C_6^{+2l2}$
and $C_k$ for $k \geq 8$ as induced subgraphs. Thus, obtaining the structure of the graphs 
in this class and showing which of them give rise to co-bipartite UDGs, would give complete 
characterization of the class of co-bipartite UDGs. 
As a step to the desired characterization of co-bipartite UDGs
we additionally forbid $C_4$ and get structural characterization of graphs in the resulting class 
$$
	\mathcal{X} = Free(S_{3,3,3}, C_3, C_4, C_5, C_6^{+3c}, C_6^{+3nc}, C_6^{+2l2}, C_7, C_8, \ldots ).
$$
Further, we show that for every graph $G=(U,W,E) \in \mathcal{X}$ both $G^*$ and $\overline{G}$
are UDGs. In other words we obtain both structural and forbidden induced subgraph characterizations for
the following two classes of co-bipartite UDGs:

\begin{itemize}
\item[$\mathcal{Y}$] -- the class of $C_4^*$-free co-bipartite UDGs, i.e. co-bipartite UDGs $G=(U,W,E)_c$ such that $G^*=(U,W,E)$ do not contain $C_4$;
\item[$\mathcal{Z}$] -- the class of $2K_2$-free co-bipartite UDGs.
\end{itemize}
In Section $5.1$ we describe the structure of the graphs in the class $\mathcal{X}$. By the results of 
the previous section it follows that 
$\mathcal{Y} \subseteq \mathcal{X}^*$ and
$\mathcal{Z} \subseteq \overline{\mathcal{X}}$, where $\mathcal{X}^* = \{ G^* : G=(U,W,E) \in \mathcal{X} \}$ and $\overline{\mathcal{X}} = \{ \overline{G} : G \in \mathcal{X} \}$. 
In Section 5.2 we use the structure of graphs in $\mathcal{X}$ to obtain a UDG-representation of 
every graph in $\mathcal{X}^*$.
This implies that $\mathcal{Y}=\mathcal{X}^*$, and gives both structural and induced forbidden subgraph characterization for the class $\mathcal{Y}$.
In Section 5.3 we show that a UDG-representation of $G^*$ can be transformed to a UDG-representation 
of $\overline{G}$, provided that the former representation satisfies certain conditions.
Finally, in section 5.4, we use this transformation to deduce UDG-representation for every graph
in $\overline{\mathcal{X}}$, which implies that $\mathcal{Z}=\overline{\mathcal{X}}$. As before this
gives both structural and forbidden subgraph characterization for the graphs in $\mathcal{Z}$.


\subsection{Structure of graphs in $\mathcal{X}$} \label{ss:structure}

Notice that the only cycle which is allowed in the class $\mathcal{X}$ is a $C_6$, 
which we call a \textit{hexagon}. 
It follows that a graph $G \in \mathcal{X}$ which do not contain a hexagon is a forest without $S_{3,3,3}$. 
It is not hard to convince oneself 
that every connected component of a $S_{3,3,3}$-free forest contains a path such that all other vertices are within distance 2 from the vertices of the path. Such graphs consist of caterpillar-like connected components which are known in the literature as \textit{lobsters}. 
\textit{Gluing vertices} of a lobster are the endpoints of a shortest path whose second
neighbourhood dominates the graph. See Figure \ref{fig:lobster} for an example of lobster with 
highlighted gluing vertices.
Now we turn to the general case, where $G \in \mathcal{X}$ is allowed to contain a hexagon. 

Let $H$ be a hexagon. We say that vertices of a set $S \subseteq V(H)$ of hexagon $H$ are
\textit{consecutive}, if $H[S]$ is
connected. 
Any two vertices of $H$ which are distance 3 away from each other we call a \textit{diagonal} of $H$.
Two hexagons $H_1$ and $H_2$ are disjoint if $S=V(H_1) \cap V(H_2)=\emptyset$, otherwise we say that they share the set $S$. If $|S|=2$ and the two vertices in $S$ are adjacent, we say that the hexagons share an edge. 

\begin{lemma}\label{hex}
If two hexagons $H_1$ and $H_2$ of $G \in \mathcal{X}$ are not disjoint then one of the following holds:
\begin{itemize}
\item They share exactly one vertex.
\item They share an edge. 
\item They share two vertices that form a diagonal in each of the hexagons.
\item They share 4 consecutive vertices, i.e. the intersection of two hexagons is a $P_4$. 
\end{itemize}
Further, $E(G[V(H_1) \cup V(H_2)])=E(G[V(H_1)]) \cup E(G[V(H_2)])$. 

\end{lemma}

\begin{proof}
It can be easily checked that in all the other cases a cycle of forbidden length 3, 4, 5, 7 or 8 would arise.
\end{proof}

For $k \geq 2$ let us define the graph $C_{6,k}$ to be a graph with 
$V(C_{6,k})=\{a, b, a_j, b_j : 1 \leq j \leq k\}$ and
$E(C_{6,k})=\{aa_j, a_jb_j, b_jb :1 \leq j \leq k\}$ (see Figure~\ref{fig:hexagon}). In particular,
$C_{6,2}$ is isomorphic to $C_6$.
A connected graph is \textit{2-connected} if there is no vertex whose removal disconnects the graph.
A maximal 2-connected subgraph of a graph is called \textit{2-connected component} of this graph.

\begin{lemma} \label{noshared}
Let $G \in \mathcal{X}$ be a 2-connected graph with no two hexagons sharing an edge. 
Then the graph $G$ is isomorphic to $C_{6,k}$ for some $k$. 
\end{lemma}

\begin{proof}
First we will show that there are no two hexagons sharing one vertex. Suppose, for contradiction, there are two hexagons
$H_1$ and $H_2$ with one vertex in common, say $V(H_1) \cap V(H_2)=\{v\}$ for some $v \in V(G)$. By Lemma~\ref{hex}, apart from the 12 edges forming two cycles of length 6, there are no other edges in $G[V(H_1) \cup V(H_2)]$. Further, one can observe that any vertex $w \in V(G)$ outside the hexagons is adjacent to at most one vertex in $V(H_1) \cup V(H_2)$. Indeed, if it has at least two neighbours in $H_1$ or at least two neighbours in $H_2$ then a cycle of length at most 5 arises. Also, if $w$ is adjacent to one vertex in $H_1 \setminus v$ and to one vertex in $H_2 \setminus v$, then either $w$ creates a cycle of length not equal to 6 or $w$ is adjacent to a neighbour of $v$ in one of $H_1$ and $H_2$, and to the vertex
which is diagonally opposite to $v$ in the other hexagon, 
in which case we have two hexagons sharing an edge, hence again a contradiction. 
Now, as the graph is 2-connected, there is a path from $V(H_1) \setminus \{v\}$ to $V(H_2) \setminus \{v\}$. We pick a path $p=h_1v_1v_2 \ldots v_k h_2$ of minimal length, where $h_1 \in V(H_1) \setminus \{v\}$, $h_2 \in V(H_2) \setminus \{v\}$, $v_1,v_2 \ldots, v_k \notin V(H_1)\cup V(H_2)$, and $k \geq 2$. 
Then, $v_i$ has at most one neighbour in $V(H_1) \cup V(H_2)$ with 
the neighbour of $v_1$ being $h_1$, neighbour of $v_k$ being $h_2$, and $v_2, v_3, \ldots, v_{k-1}$ can only be adjacent
to $v$ by minimality of the path. Also, by minimality, the path $p$ does not have
chords, i.e. edges connecting two non-consecutive vertices of $p$. Now, if $v_i$ is 
adjacent to $v$ for some $i$, then either a cycle of length not equal to 6 arises or there are two hexagons sharing the edge $vv_i$. 
Otherwise, $p$ together with the shortest path between $h_1$ and $h_2$ in $V(H_1) \cup V(H_2)$ either induce a cycle of length more than 6, or one of $h_1$ or $h_2$ is a neighbour of $v$ in which case 
we have two hexagons sharing an edge ($vh_1$ or $vh_2$).
The contradiction shows that there are no two hexagons sharing a vertex. 

Now, as $G$ is 2-connected it contains a cycle of length 6. Let us consider the maximal subgraph $G'$ isomorphic to $C_{6,k}$ containing this cycle. We will show that $G$ coincides with $G'$. Suppose not, 
then there is another hexagon $C$ sharing some vertices with some of the hexagons of $G'$. 
If $C$ shares 4 consecutive vertices with some hexagon, then it must share at least one vertex with 
each of the hexagons of $G'$, which is possible only if $V(C) \cup V(G')$ induces $C_{6,k+1}$
in $G$. But this contradicts maximality of $G'$.
Otherwise, if $C$ shares a diagonal with some of the hexagons of $G'$, then it either shares a diagonal with all hexagons or it shares one vertex with some hexagon. The latter case is impossible by the previous paragraph, and the former case proves that $V(C) \cup V(G')$ induces $C_{6,k+2}$ contradicting the maximality of $G'$. Thus, we deduce that $G$ is isomorphic to $C_{6,k}$.  
\end{proof}

\noindent
We say that an edge $xy$ of a graph $G$ is a \textit{cutset} if $G \setminus \{x,y\}$ has more connected
components than $G$.

\begin{lemma}\label{cutedge}
If $G \in \mathcal{X}$ has two hexagons $H_1$ and $H_2$ sharing an edge, then the edge is a cutset.
\end{lemma}
\begin{proof}
Let two hexagons share an edge, i.e. $V(H_1) \cap V(H_2)=\{v_1,v_2\}$ with $v_1v_2 \in E(G)$. 
Notice that each vertex in $V(G) \setminus (V(H_1) \cup V(H_2))$ has at most 1 neighbour in $V(H_1) \cup V(H_2)$.  
Indeed, if a vertex has two neighbours in one of the hexagons, then a cycle of length less than 6 arises.
If vertex is adjacent to a vertex $h_1$ in $H_1 \setminus \{v_1,v_2\}$ and a vertex $h_2$ in 
$H_2 \setminus \{v_1,v_2\}$, then the longer path from $h_1$ to $h_2$ in 
$G[V(H_1) \cup V(H_2)] \setminus \{v_1\}$ or in 
$G[V(H_1) \cup V(H_2)] \setminus \{v_2\}$ together with $v$ would make a chordless 
cycle of length more than 6. 
  
Now suppose to the contrary that $G \setminus \{v_1, v_2\}$ is connected.
Then, there is a path between $V(H_1) \setminus \{v_1, v_2\}$ and $V(H_2) \setminus \{v_1, v_2\}$. 
Let $p=h_1w_1w_2 \ldots w_kh_2$ be such a path of minimal length, where 
$h_1 \in V(H_1) \setminus \{v_1, v_2\}$ and $h_2 \in V(H_2) \setminus \{v_1, v_2\}$. 
The above discussion implies that $k \geq 2$.
Moreover, by minimality of $p$, none of the vertices $w_1, \ldots, w_k$ belongs to $V(H_1) \cup V(H_2)$;
for $i = 2, \ldots, k-1$, if $w_i$ has a neighbour in the hexagons, then this neighbour is either $v_1$ or $v_2$; and the path $p$ does not have chords.
%
Now, let us denote the vertices of $H_1$ by $v_1, v_2, \ldots, v_6$ and vertices of $H_2$ by $v_1, v_2, v'_3, v'_4, v'_5, v'_6$ with the edges $\{v_1v_2, v_2v_3, v_3v_4, v_4v_5, v_5v_6, v_6v_1, v_2v'_3, v'_3v'_4, v'_4v'_5, v'_5v'_6, v'_6v_1\}$. Then, 
note that $h_1 \notin \{ v_3, v_6 \}$ as otherwise $V(H_1) \cup \{w_1, v'_3, v'_6\}$ induce a $C_6^{+3c}$. 
Similarly, $h_2 \notin \{ v_3', v_6' \}$.
So without loss of generality we can assume $h_1=v_4$ and $h_2 \in \{v'_4, v'_5\}$. 
Then the paths connecting $h_1$ and $h_2$ in $G[V(H_1) \cup V(H_2)] \setminus \{v_1\}$
and in $G[V(H_1) \cup V(H_2)] \setminus \{v_2\}$ both have at least 5 vertices. Each of these paths together with the path $p$ form a cycle of length more than 6, and hence each of the cycles has a chord. Let
$v_1w_i$ be a chord in one of the cycles, and $v_2w_j$ be a chord in the other cycle, such that 
$i$ and $j$ are smallest possible. 
Then both $v_1v_6v_5v_4w_1w_2 \ldots w_i$ and $v_2v_3v_4w_1w_2 \ldots w_j$ are chordless cycles.
Since every chordless cycle in $G$ is a hexagon, we conclude that $i=2$ and $j=3$.
But then $v_1v_2w_2w_3$ induce a $C_4$. This contradiction finishes the proof. 
\end{proof}

Let $n \in \mathbb{N}$, and $k_i \in \mathbb{N}, k_i \geq 2$ for every $i=1, \ldots, n$, and
$d_1, \ldots, d_{n-1} \in \{1, -1 \}$.
Let $C^{i}_{6, k_i}$ be a graph isomorphic to $C_{6, k_i}$ with 
$V(C^{i}_{6,k_i})=\{a^i, b^i, a^i_j, b^i_j : 1 \leq j \leq k_i\}$ and
$E(C^{i}_{6,k_i})=\{a^ia^i_j, a^i_jb^i_j, b^i_jb^i :1 \leq j \leq k_i\}$. 
A \textit{hexagonal strip} $H(k_1, d_1, k_2, d_2, \ldots, k_{n-1}, d_{n-1}, k_n)$ is the graph obtained by gluing together $C^{1}_{6,k_1}, \ldots, C^{n}_{6,k_n}$ in such a way that the edge $b^i_2b^i$ is glued to $a^{i+1}a^{i+1}_1$ and the direction is described by $d_i$:
\begin{itemize}
	\item if $d_i=1$, then $b^i_2$ is identified with $a^{i+1}$ and $b^i$ is identified with $a^{i+1}_1$;
	\item if $d_i=-1$, then $b^i_2$ is identified with $a^{i+1}_1$ and $b^i$ is identified with $a^{i+1}$.
\end{itemize}


\begin{figure}[h]
    		
    	\begin{subfigure}[b]{0.22\textwidth}
       	\includegraphics[scale=0.42]{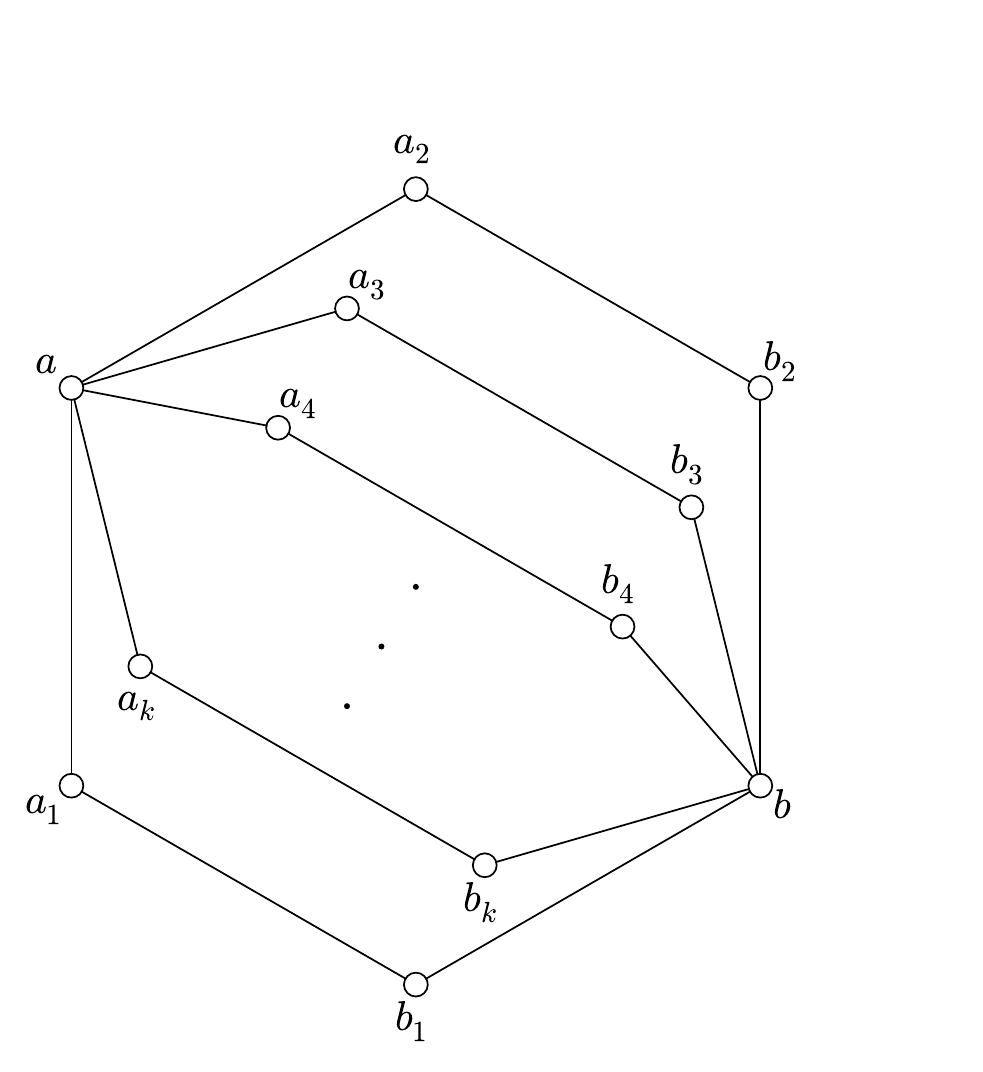}
       	\caption{The graph $C_{6,k}$}
		\label{fig:hexagon}
	\end{subfigure}
	~
	\begin{subfigure}[b]{0.78\textwidth}
              \includegraphics[scale=0.62]{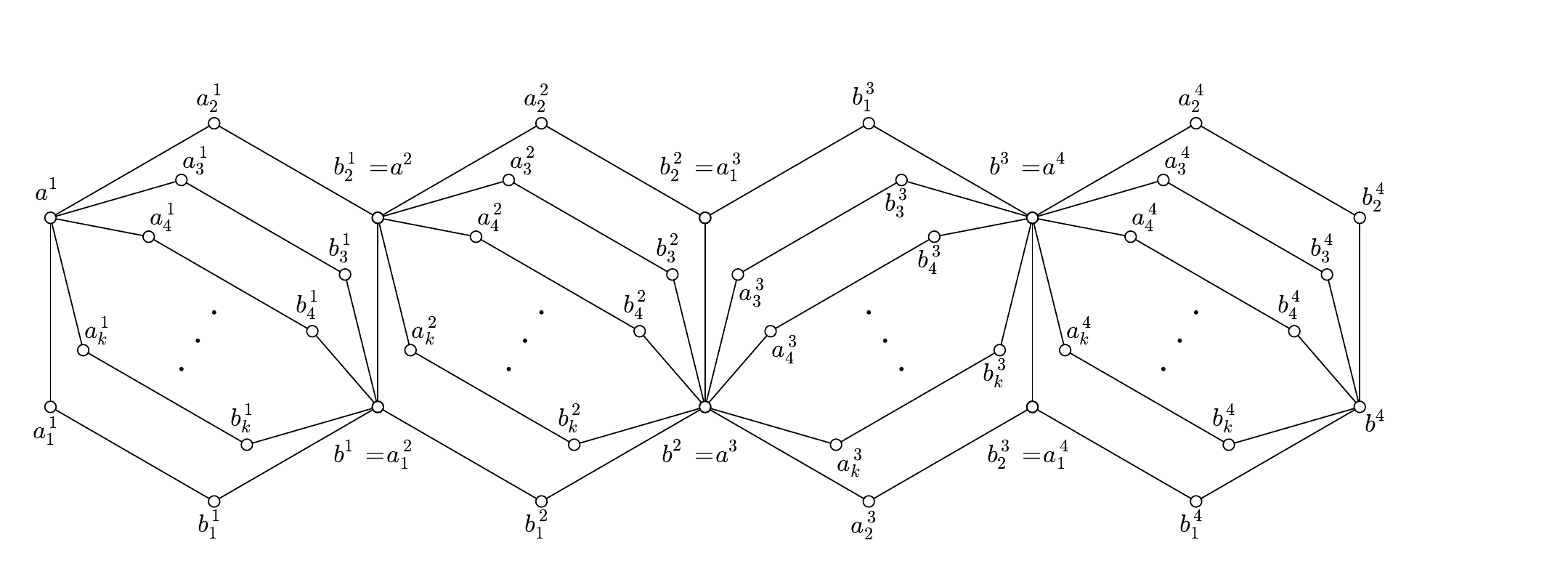}
              \caption{The graph $H(k,1,k,-1,k,-1,k)$}
		\label{fig:hexagonal_strip}
	\end{subfigure}
		
    	\caption{The graphs $C_{6,k}$ and $H(k,1,k,-1,k,-1,k)$}
    	\label{hexagon2con}
\end{figure}

\begin{lemma} \label{2con}
Let $G$ be a 2-connected graph in $\mathcal{X}$. Then $G$ is isomorphic to
$H(k_1, d_1, \ldots, k_{n-1}, d_{n-1}, k_n)$, for some $n \in \mathbb{N}$, and $k_i \in \mathbb{N}, k_i \geq 2, d_i \in \{1, -1 \}$, $i = 1, \ldots, n$.  
\end{lemma}

\begin{proof}
If two hexagons intersect at an edge then we call such an edge \textit{shared}. We prove the statement by induction on the number of shared edges. If there are no shared edges, then the conclusion follows from Lemma~\ref{noshared}. So suppose there is a shared edge $v_1v_2$. By the Lemma~\ref{cutedge} we know that such an edge is a cutset. Let $C'_1, C'_2, \ldots, C'_k$ be the components of 
$G \setminus \{v_1,v_2\}$, and let $C_i=G[V(C_i') \cup \{v_1,v_2\}]$. 
It is easy to see that each of $C_1, C_2, \ldots, C_k$ is 2-connected and has less shared edges
than $G$. Hence, by induction, each of this graphs is a hexagonal strip. Further, we conclude that 
$k=2$ and $C_1$ and $C_2$ are properly glued along $v_1v_2$ to form a hexagonal strip, 
because otherwise an induced copy of forbidden $C^{+2l2}$ would arise.

\end{proof}

\begin{lemma} \label{hexagonalcaterpillar}
Let $G$ be a graph in $\mathcal{X}$ and $H$ be a 2-connected component of $G$. Then $H$ is isomorphic to some hexagonal strip $H(k_1, d_1, k_2, d_2, \ldots, k_{n-1}, d_{n-1}, k_n)$ and vertices
$a^i, b^i, a^i_1, b^i_2$ can only be additionally adjacent to some pendant vertices of $G$, with exception 
of one of $\{a^1, a^1_1\}$ and $\{b^n, b^n_2\}$. Further, all the other vertices of $G$ do not have any 
more neighbours in $G \setminus V(H)$.
\end{lemma}

\begin{proof}
The structure of the $H$ follows from Lemma~\ref{2con}, so we only need to argue about the adjacencies
between the vertices in $V(G) \backslash V(H)$ and the vertices in $V(H)$. Consider a connected component $C$ of $G \setminus V(H)$. As $H$ is maximal 2-connected subgraph, the vertices of $C$ can only be adjacent to at most one vertex of $H$. If $C$ has some vertices adjacent to $v \in V(H)$, we will refer to $C$ as a \textit{$v$-component}. Since $G$ is $C_3$-free, we have that every $v$-component of size at least 2, 
must have two vertices $u, w$ such that  $uw, wv \in E(G)$, $u$ is non-adjacent to any vertex of $H$, 
and $w$ is non-adjacent to any vertex of $H$ other than $v$.

Let us first consider a $v$-component for $v \in \{a^i, b^i, a^i_1, b^i_2: 1 \leq i \leq n\} \backslash \{a^1, a^1_1, b^n, b^n_2\}$. Suppose the component has size at least 2, hence, by the above argument, the $v$-component has two vertices $u,w \in V(G) \backslash V(H)$ such that $uw, wv \in E(G)$. But then $u, w$ and the two hexagons of $H$, which share an edge containing $v$, form a subgraph containing an induced $C_6^{+2l2}$. 
We conclude that any $v \in \{a^i, b^i, a^i_1, b^i_2: 1 \leq i \leq n\} \backslash \{a^1, a^1_1, b^n, b^n_2\}$ is adjacent to pendant vertices of $G$ only. 

Now, consider a vertex $b^i_k$ for any $i \neq n$ and $k \neq 2$. Suppose to the contrary, 
that there is a vertex $w \in V(G) \backslash V(H)$ which is adjacent to $b^i_k$.  Then, $w, b^i_k, a_2^i, a^i$ together with $a^{i+1},a^{i+1}_1, a^{i+1}_2, b^{i+1}$, $b^{i+1}_1, b^{i+1}_2$ induce a $C^{+2l2}_6$. Hence, vertices $b^i_k$ for any $i \neq n$ and $k \neq 2$, have no neighbours outside $H$. 
Similarly, one can deduce that $a^i_k$ has no neighbours for any $i \neq 1$, $k \neq 1$. 

We are left to argue about adjacencies of the vertices $a^1, a^1_i$ and $b^n, b^n_i$. Consider the case when $n>1$. Notice that if $a^1_i$ and $a^1_j$ each have a neighbour outside $H$, for some $i \neq j$, then taking the two neighbours together with hexagon $G[a^1, a^1_i, a^1_j, b^1_i, b^1_j, b_1]$, and together with a neighbour of $b^1$ either $b^2_1$ or $a^2_2$ (depending on wether $a_1^2$ or $a_2$ gets identified with $b^1$, respectively), we get an induced $C_6^{+3nc}$. This contradiction proves that only one of $a^1_i$ may have a neighbour outside $H$. Moreover, $a^1_2$ does not have a neighbour outside $H$, as 
otherwise an induced $C_6^{+3c}$ would arise. Therefore, without loss of generality we can assume that if $a^1_i$ has a neighbour outside $H$, then $i=1$. It is clear that if there is an $a^1$-component and an $a^1_1$-component which both have sizes at least 2, then we have an induced $C_6^{+2l2}$. 
The analogous arguments holds for $b^n, b^n_i$. This finishes the proof for $n>1$. 
The case $n=1$ can be shown to hold by similar analysis.

\end{proof}

Let $G$ be a graph consisting of a hexagonal strip $H(k_1, d_1, k_2, d_2, \ldots, k_{n-1}, d_{n-1}, k_n)$ together with some pendant vertices attached to $a^i, b^i, a^i_1, b^i_2$ and with some radius 2 trees attached to a vertex $a \in \{a^1, a^1_1\}$ and a vertex $b \in \{b^n, b^n_2\}$. Then we call $G$ a 
\textit{hexagonal caterpillar with gluing vertices} $a$ and $b$.
Further let $H_1, H_2, \ldots, H_k$ be a set of vertex disjoint hexagonal caterpillars or lobsters,
with $a_i$ and $b_i$ being gluing vertices of $H_i$, for $i=1, \ldots, k$. 
Then the \textit{generalized hexagonal caterpillar} $(H_1, b_1, a_2, H_2, b_2, a_3, H_3, \ldots, b_{k-1}, a_k, H_k)$ is the graph obtained from $H_1, H_2, \ldots, H_k$ by identifying pairs of vertices $b_i$ and $a_{i+1}$ for every $i = 1, \ldots, k-1$. 

This description gives us a universal structure for the graphs in $\mathcal{X}$. One can deduce this by noting
that any graph in $\mathcal{X}$ should consist of 2-connected components provided by Lemma~\ref{hexagonalcaterpillar} and lobsters glued together, and that the generalized hexagonal caterpillars described above are the most general graphs we can obtain with this gluing without forming $S_{3,3,3}$. 
We state this as the main result of this section.
 

\begin{theorem}\label{th:Xstruct}
	Generalized hexagonal caterpillars are universal graphs for the class $\mathcal{X}$, 
	that is each such graph belongs to $\mathcal{X}$ and every graph $G \in \mathcal{X}$ is an 
	induced subgraph of some generalized hexagonal caterpillar.
\end{theorem}

In further sections we will use the structural characterization of graphs in $\mathcal{X}$ to show that for
every $G \in \mathcal{X}$ both $G^*$ and $\overline{G}$ are UDGs. First, by Theorem \ref{th:Xstruct}
it is enough to prove the result only for generalized hexagonal caterpillars.
Further, without loss of generality we can restrict our consideration to those graphs in $\mathcal{X}$ in 
which no vertex is adjacent to more than one pendant vertex. Indeed, assume a graph $G \in \mathcal{X}$
has a vertex with two pendant neighbours $a$ and $b$. Then $a$ and $b$ belong to the same part in $G$,
and therefore to the same part in both $G^*$ and $\overline{G}$, in particular $a$ and $b$ are adjacent in
these graphs. Moreover, in each of the graphs $N(a) \setminus \{ b \} = N(b) \setminus \{ a \}$. 
This implies that if we have a UDG-representation $f$ for $H \setminus \{b\}$, where $H$ is one of 
$G^*$ and $\overline{G}$, then an extension $f'$ of $f$ to $V(H)$ with $f'(b) = f(a)$ is the 
UDG-representation for $H$. 
Therefore, from now on when we refer to a graph in $\mathcal{X}$ we mean a generalized hexagonal
caterpillar which is constructed from hexagonal caterpillars or lobsters whose vertices have at most 
one pendant neighbour (see Figure \ref{fig:HexLobsterWithGluing}).

\begin{figure}[h]
    		
    	\begin{subfigure}[b]{\textwidth}
       	\includegraphics[scale=0.8]{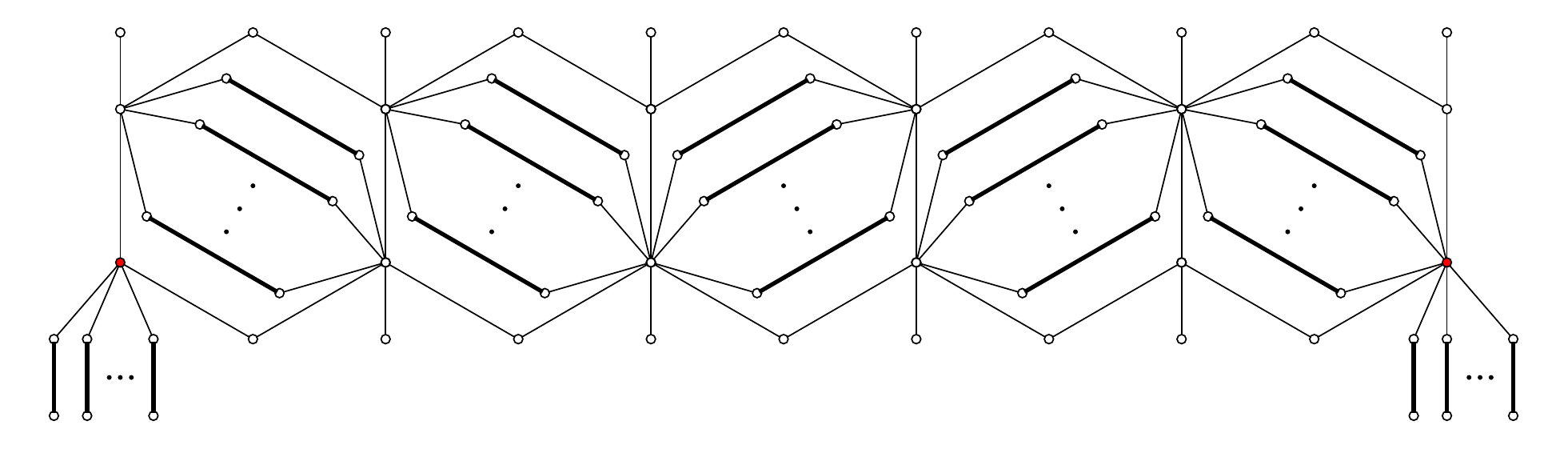}
       	\caption{Hexagonal caterpillar with gluing vertices (filled vertices)}
		\label{fig:hexCat}
	\end{subfigure}
	
	\vskip0.8cm
	
	\begin{subfigure}[b]{\textwidth}
		\centering
       	\includegraphics[scale=0.55]{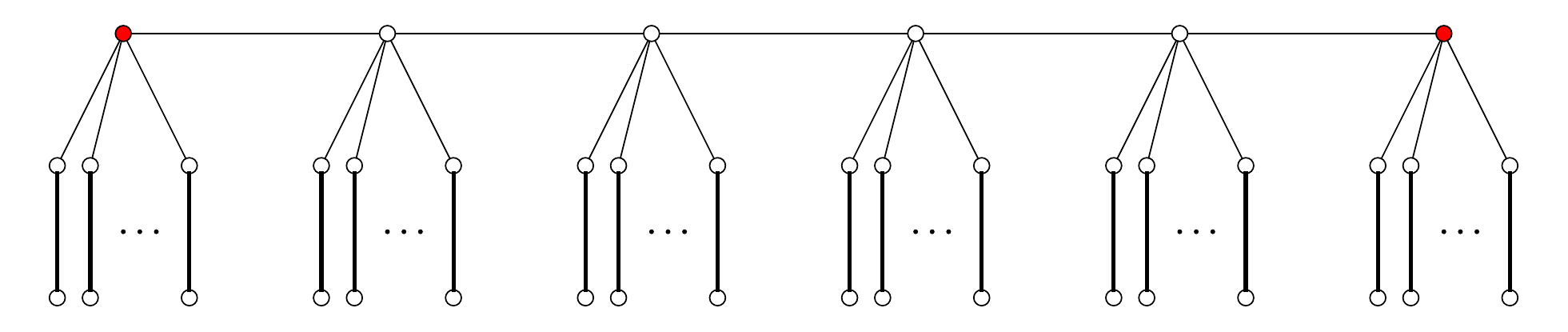}
              \caption{Lobsters with gluing vertices (filled vertices)}
		\label{fig:lobster}
	\end{subfigure}
		
    	\caption{}
    	\label{fig:HexLobsterWithGluing}
\end{figure}


\subsection{$C_4^*$-free co-bipartite unit disk graphs}\label{ss:C4rep}

In this section we show that for a graph $G \in \mathcal{X}$ the graph $G^*$ is UDG. 
We do this in two steps. 
First, we represent basic graphs in $\mathcal{X}^*$ and then show how representation of a general
graph in $\mathcal{X}^*$ can be obtained from a representation of a basic graph. To explain
this formally we introduce some definitions.

Let $G$ be a bipartite or co-bipartite graph with parts $U$ and $W$, and let $uw$ be an edge of 
$G$ with $u \in U$ and $w \in W$. An edge $u'w'$ of $G$ with $u' \in U$ and  $w' \in W$ is a 
\textit{twin} of $uw$ if
$N_{W}(u) \triangle N_W(u') = \{ w, w' \}$ and 
$N_{U}(w) \triangle N_U(w') = \{ u, u' \}$, where $P \triangle Q$ is the symmetric difference of 
sets $P$ and $Q$.
In this case we also say that the vertex $u'$ is a \textit{twin} of the vertex $u$ and the vertex $w'$ 
is a twin of the vertex $w$.
Notice that the relation of being twins is symmetric and transitive. 
The graph $G$ is \textit{basic} if it does not contain twin edges.
The operation of \textit{duplication} of the edge $uw$ is to add one or more new edges to 
$G$ each of which is a twin of $uw$. Note that $uw$ and $u'w'$ are twins in $G$ if and only if
they are twins in $G^*$.
Each of the thick edges in Figures \ref{fig:hexCat} and \ref{fig:lobster} is called \textit{parallel edge} of
hexagonal caterpillar or lobster, respectively. 
Let $H$ be a generalized hexagonal caterpillar obtained from $H_1, \ldots, H_k$, then an edge of $H$
is called \textit{parallel}, if it is parallel edge of one of the graphs $H_1, \ldots, H_k$. 
Similarly, an edge of $H^*$ is parallel, if it is parallel edge in $H$. 
It follows from the results of Section \ref{ss:structure} that a generalized hexagonal caterpillar is 
either basic or can be obtained from a basic one by duplicating some of  its parallel edges.
In Section \ref{sss:basic} we show how to represent graphs in $\mathcal{X}^*$ corresponding to
basic generalized hexagonal caterpillars,
and in Section \ref{sss:general} we extend this representation to the case of arbitrary 
generalized hexagonal caterpillars. 


\subsubsection{Representation of basic graphs}\label{sss:basic}

\begin{theorem} \label{basiclobster}
Let $G$ be a basic lobster in $\mathcal{X}$. Then $G^*$ is UDG.
\end{theorem}

\begin{proof} 
We will show how to obtain UDG-representation $f$ of $G^*$ for the lobster $G$ with the vertex set $V(G)=\{g_i, b_i, r_i : 1 \leq i \leq n\}$ and edge set $E(G)=\{g_ig_{i+1}: 1 \leq i \leq n-1\} \cup \{g_ib_i, b_ir_i : 1\leq i \leq n\}$. We will refer to the vertices $\GG=\{g_i: 1 \leq i \leq n\}$, $\BB=\{b_i: 1 \leq i \leq n\}$ and $\RR=\{r_i :  1 \leq i \leq n \}$ and to their images in the plane as to green, blue and red, respectively (see Figure~\ref{c4freelobster}, for the visualization of the proof). Let us denote the parts of bipartition of $G$ by $C_1=\{b_i, g_{j}, r_{k}: 1 \leq i, j, k \leq n, i $ -- odd$, j, k $ -- even$\}$ and $C_2=\{b_i, g_{j}, r_{k}: 1 \leq i, j, k \leq n, i $ -- even$, j, k $ -- odd$\}$. Finally, denote by $\GG_1, \BB_1, \RR_1$ and $\GG_2, \BB_2, \RR_2$, the green,
blue and red vertices belonging to parts $C_1$ and $C_2$, respectively.

To put the points on the plane, we first fix some $\mm \in \left(0, \frac{1}{n}\right)$ and draw parallel lines $L_1, L_2, L_3, L_4$ such that $L_2$ and $L_3$ are
between $L_1$ and $L_4$ and $\dist(L_1, L_4)=1$, $\dist(L_2, L_3)=\sqrt{1- \mm^2}$, $\dist(L_1, L_2)=\dist(L_3, L_4)=(1-\sqrt{1- \mm^2}) / 2$. Then we draw $k$ lines $\{M_i: 1\leq i \leq n\}$ perpendicular to line $L_1$ and evenly spaced with distance $\mm$ between consecutive ones, i.e. $\dist(M_1, M_i) = (i-1) \mm$ for all $1 \leq i \leq n$ and all $M_i$'s are on one side of $M_1$. The intersections between $M_i$'s an $L_j$'s define the points of our UDG-representation of $G^*$ as follows:
\begin{itemize}
\item if $i$ is odd,  then $f(b_i)=M_i \cap L_1$, $f(r_i)=M_i \cap L_4$, $f(g_i) = M_i \cap L_3$;
\item if $i$ is even, then $f(b_i)=M_i  \cap L_4$, $f(r_i)=M_i \cap L_1$, $f(g_i)=M_i \cap L_2$. 
\end{itemize}

It is not hard to see that $f(C_1) \subseteq L_1 \cup L_2$ and $f(C_2) \subseteq L_3 \cup L_4$. 
The diameter of $f(C_1)$ is bounded by 
$\sqrt{\dist(L_1, L_2)^2 + \dist(M_1, M_n)^2}$. As $\dist(L_1, L_2)=\frac{1-\sqrt{1-\mm^2}}{2} \leq \frac{1-(1-\mm^2)}{2}=\frac{\mm^2}{2}$ and $\dist (M_1, M_n) = (n-1)\mm$, we deduce that the diameter of $f(C_1)$ is at most $n\mm$. By symmetry, the diameter of $f(C_2)$ is also
bounded by $n\mm$. Thus, as $\mm \leq 1/n$, the diameter of each of $f(C_1)$ and $f(C_2)$ is at most 1, which correspond to cliques $C_1$ and $C_2$ in $G^*$. 
It remains to show that $u \in C_1$ and $v \in C_2$ are adjacent in $G^*$ if and only if
the distance between the corresponding points $f(u)$ and $f(v)$ is at most 1. 
Since this is mostly technical task, we moved the confirming calculations to Appendix \ref{app:T15}.
\end{proof}

\begin{figure}[h]
    		
    	\begin{subfigure}[t]{0.5\textwidth}
       	\includegraphics[scale=0.4]{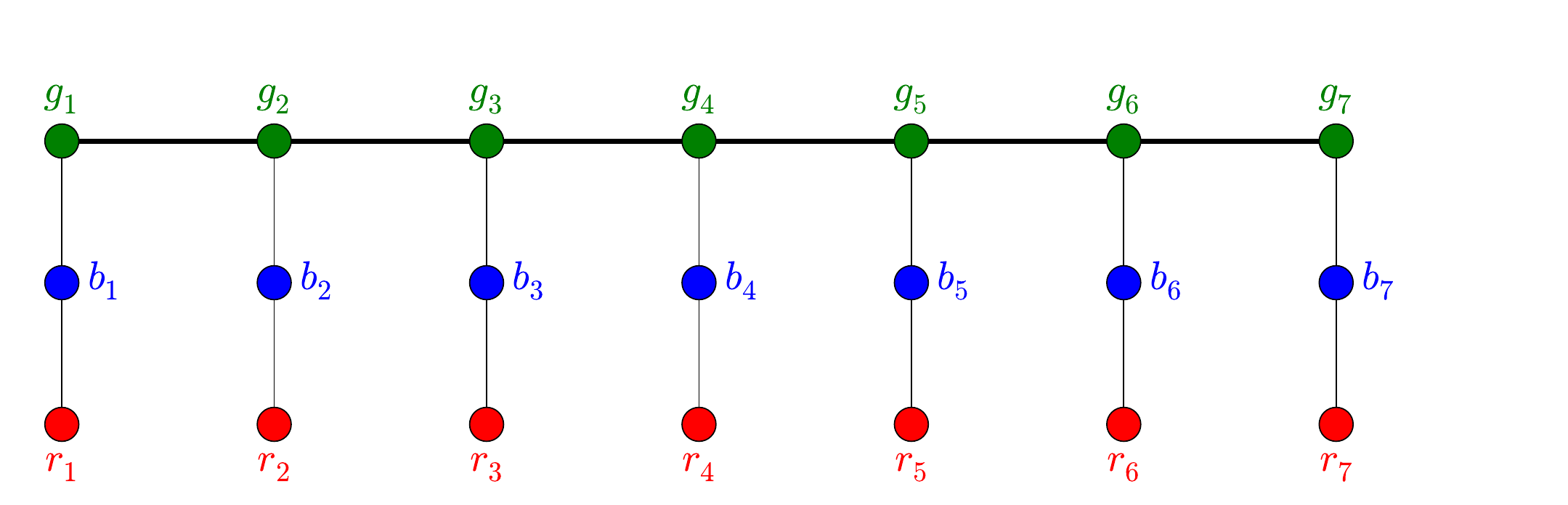}
       	\caption{Basic lobster $G$ of length 7}
		\label{fig:simple_lobster}
	\end{subfigure}
	~
	\begin{subfigure}[t]{0.5\textwidth}
       	\includegraphics[scale=0.4]{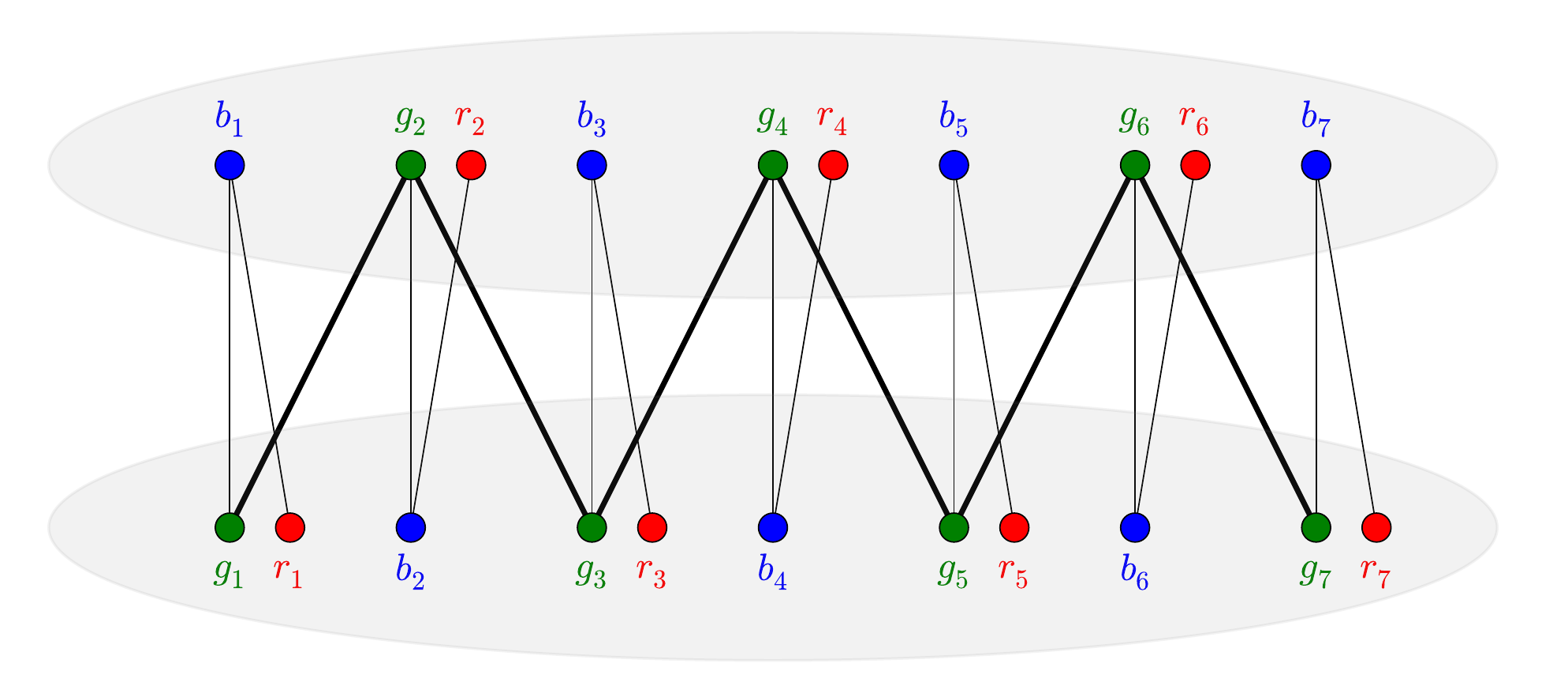}
              \caption{The graph $G^*$. Vertices in a gray area form a clique}
		\label{fig:simple_lobster_star}
	\end{subfigure}

	\begin{subfigure}[t]{\textwidth}
		\centering
		
       	\includegraphics[scale=0.8]{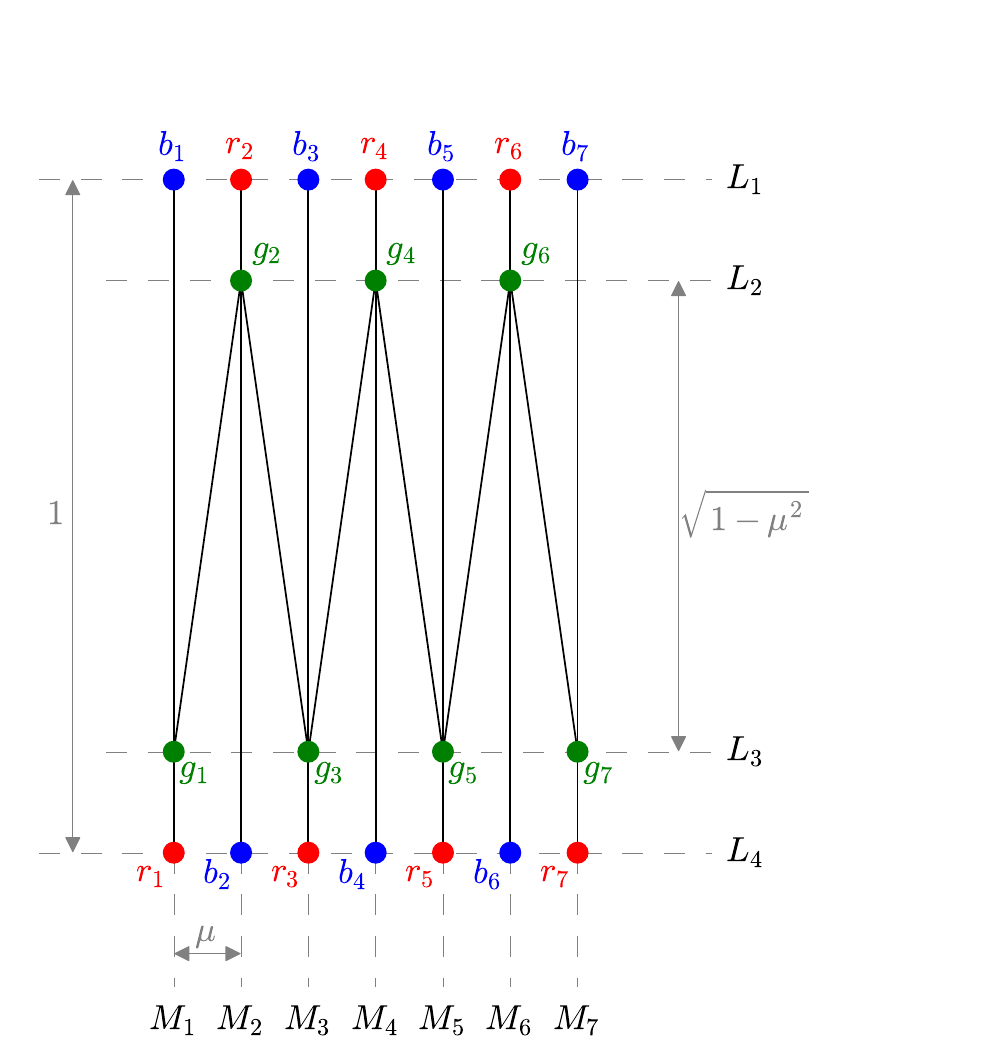}
              \caption{The UDG-representation of $G^*$}
		\label{fig:simple_lobster_star_repr}
	\end{subfigure}
			
    	\caption{}
    	\label{c4freelobster}
\end{figure}


\begin{figure}[h]
      	\begin{subfigure}[t]{0.5\textwidth}
    		\centering
       	\includegraphics[scale=0.8]{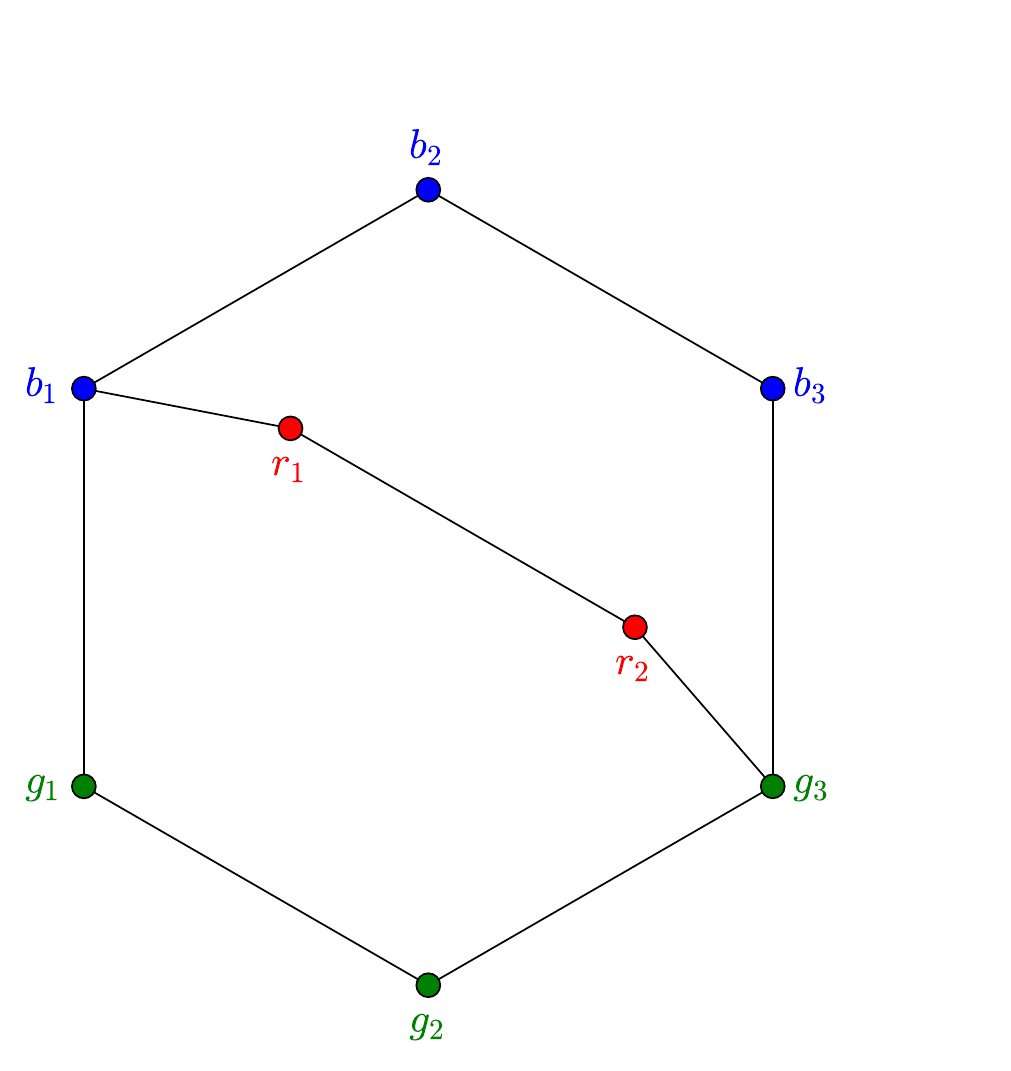}
       	\caption{The graph $C_{6,1}$}
		\label{fig:C6*}
	\end{subfigure}
	~
	\begin{subfigure}[t]{0.5\textwidth}
		\centering
       	\includegraphics[scale=0.9]{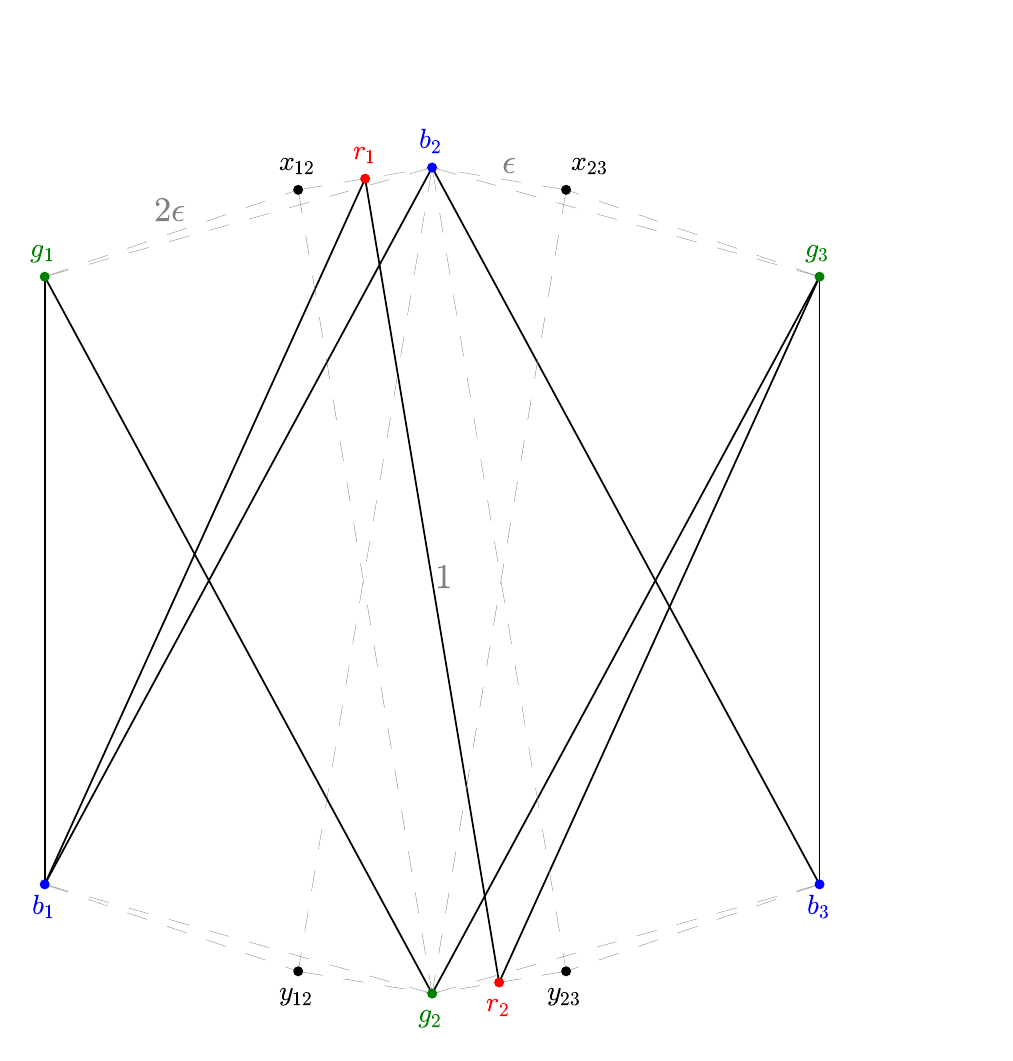}
              \caption{The UDG-representation of $C_{6,1}^*$}
		\label{fig:C6*representation}
	\end{subfigure}

    	\caption{}
    	\label{C6*representation}
\end{figure}

\begin{theorem}\label{th:repBasicGraphs}
	Let $G$ be a basic graph in $\mathcal{X}$. Then $G^*$ is UDG.
\end{theorem}

\begin{proof}
We will abuse the notation and instead of calling the image of a vertex $v$ by $f(v)$, we will refer to 
it simply by $v$. Thus, when talking about adjacencies, $v$ will be considered as a vertex of the 
graph $G^*$, and when talking about distances, or some other geometric properties, $v$ will mean 
the point $f(v)$. 
We denote $n=|V(G)|$ and we fix a positive parameter $\epsilon < \min\left\{\frac{1}{15n}, \frac{1}{128}\right\}$. 

\vskip7px \noindent
\textbf{Single hexagon.}
We start by representing graph $G^*$ when $G$ is isomorphic to $C_{6,1}$ (see Figure~\ref{fig:C6*}). 
Let
$V(G)=\{g_1, g_2, g_3, b_1, b_2, b_3, r_1, r_2 \}$ and
$E(G)=\{g_1g_2, g_2g_3, g_3b_3, b_3b_2, b_2b_1, b_1g_1, b_1r_1, r_1r_2, r_2g_3\}$.
To construct a UDG-representation $f$ of $G^*$, first, we place 6 points 
$\{x_{12}, b_2, x_{23}, y_{12}, g_2, y_{23}\}$ in the plane forming two rectangles as 
follows (see Figure~\ref{fig:C6*representation}): 
\begin{itemize}
\item $\{x_{12},b_2,y_{23},g_2\}$ forms a $1 \times \epsilon$ rectangle, where $\dist(x_{12}, g_2)=
\dist (b_2, y_{23})=1$, $\dist (x_{12},b_2)= \dist(g_2, y_{23}) = \epsilon$ and $[x_{12},g_2]$ is perpendicular to $[g_2,y_{23}]$.
\item $\{b_2, x_{23}, g_2, y_{12}\}$ forms a $1 \times \epsilon$ rectangle, where $\dist(b_2, y_{12})=
\dist (x_{23},g_2)=1$, $\dist (b_2, x_{23})= \dist(y_{12}, g_2) = \epsilon$ and $[b_2,x_{23}]$ is perpendicular to $[x_{23},g_2]$, and $x_{23} \neq x_{12}$.   
\end{itemize}

\noindent
Further, we place points $\{g_1, g_3, b_1, b_3\}$ as shown in Figure~\ref{fig:C6*representation} such that:
\begin{itemize}
\item $\dist(g_1, g_2)=\dist(g_2, g_3)=\dist(b_1, b_2)=\dist(b_2, b_3)=1$.
\item $\dist(g_1, x_{12})=\dist(g_3, x_{23})=\dist(b_1, y_{12})=\dist(b_3, y_{23})=2 \epsilon$.
\end{itemize} 

\noindent
Finally, we place the points $\{r_1, r_2\}$ as follows:
\begin{itemize}
\item $r_1$ in the middle of the segment $[x_{12},b_2]$, $r_2$ in the middle of the segment $[g_2,y_{23}]$.
\end{itemize}

We argue that this is indeed a UDG-representation of $G^*$. First of all, observe that the two parts of bipartition of $G$ are 
	$C_1=\{g_1, g_3, b_2, r_1\}$ and 
	$C_2=\{b_1, b_3, g_2, r_2\}$.
By triangle inequalities, one may obtain that the distances between the points in
$f(C_1)$ (resp. $f(C_2)$) are at most $6 \epsilon<1$. Hence we only need to deal with distances between $f(C_1)$ and $f(C_2)$.  Note that $\{b_2, r_1, g_2, r_2\}$ belongs to the rectangle $\textup{Conv}(x_{12}, b_2, y_{23}, g_2)$ and then it is easy to see that 
$\dist(r_1, r_2)=1$ and the other distances $\dist(r_1, g_2), \dist(r_2, b_2), \dist(b_2, g_2)$ between these points are at least 
$\sqrt{1+(\epsilon/2)^2} \geq 1+ \epsilon^2/16$. 
The rest of the pairs of vertices in the different parts $C_1$ and $C_2$ include a ``corner'' vertex -- $g_1, g_3, b_1$ or  $b_3$, and by symmetry, it is enough to show

\begin{itemize}[leftmargin=*]
	\item[] \textbf{Claim 1.} $\dist(g_1, y) \leq 1$ for all $y \in [b_1, y_{12}] \cup [y_{12}, g_2]$; and
	\item[] \textbf{Claim 2.} $\dist(g_1, y)>1$ for all $y \in (g_2, y_{23}] \cup [y_{23}, b_3]$.
\end{itemize}

\noindent
One can see that Claim 1 holds, by extending the segment $[g_1, b_1]$ to $[g_1,b]$ such that $\triangle g_1bg_2$ is a right-angled triangle with diagonal $[g_1,g_2]$ of length $1$, and $[g_1,b]$ being
perpendicular to $[b,g_2]$. Then, noticing that $y_{12}$ and $b_1$ lie inside this triangle, we conclude, that all the points of $[b_1,y_{12}] \cup [y_{12},g_2]$ lie inside this triangle and have distances at most 1 to any vertex of the triangle (in particular to $g_1$). 
We note that one can be more precise and by estimating the projections calculate that the distance between $g_1$ and $y_{12}$ is at most $\sqrt{1-(\epsilon/2)^2}\leq 1- \epsilon^2/8 $ and the distance between $g_1$ and the midpoint of $[y_{12},g_2]$ is at most $\sqrt{1-(\epsilon/4)^2} \leq 1-\epsilon^2/32$. Also, one can calculate that the distance between $g_1$ and $b_1$ is between $1-10\epsilon^2$ and $1-9\epsilon^2$ (this estimate holds for $\epsilon<1/5$). 

Regarding Claim 2, one should first observe that $\textup{Conv}(g_1, g_2, b_2, x_{12})$ is a quadrilateral, i.e. $x_{12}$ indeed lies above the 
line $g_1b_2$ in the Figure~\ref{C6*representation} (and by symmetry $y_{23}$ lies below $g_2b_3$). This follows from the fact that $\triangle g_1x_{12}g_2$ is an isosceles triangle with $\angle g_2g_1x_{12} =\angle g_1x_{12}g_2 =\alpha<90$. Thus $\angle g_1x_{12}b_2 =\alpha+90<180$. From this we deduce that $\angle b_2g_1g_2 <\alpha<90$ and hence $\angle g_1g_2b_3 >90$.
The latter inequality implies that any point on $(g_2, b_3]$ has distance greater than $\dist(g_1,g_2)=1$ and hence we are done. Indeed, it is not hard to evaluate using the Pythagorean theorem, that $\dist(g_1, y_{23}) \geq \sqrt{1+\epsilon^2} \geq 1+\epsilon^2/4$ and 
$\dist(g_1, r_2) \geq \sqrt{1+(\epsilon/2)^2}\geq{1+\epsilon^2/16}$. 

\vskip7px \noindent
\textbf{Two hexagons sharing an edge.}
Now we proceed to showing how to represent $G^*$, where $G$ consists of two $C_{6,1}$ sharing an
edge, and two additional pending vertices $a_3$ and $h_3$ (see Figure \ref{fig:2hex_edge}).
The corresponding representation is illustrated in Figure \ref{fig:2hex_edge_representation}.
Points $a_3, h_3$ are placed in such a way that:
\begin{itemize}
	\item $a_3, h_3 \in L(b_3, g_3)$ and $\dist(a_3,b_3)=1$, $\dist(g_3, h_3)=1$, as shown in the picture. 
\end{itemize}
It is easy to see that distance from $a_3$ to every point in the opposite part except $b_3$ is larger than 1, or indeed larger than $\sqrt{1+(2\epsilon)^2}$. By symmetry, the same holds for points $h_3$ and $g_3$. The distances involving points $g_3$ and $b_3$ are as needed for UDG-representation, because these points belong to both hexagons. For the rest distances, it is enough to show the following

\begin{itemize}[leftmargin=*]
	\item[] \textbf{Claim 3.} For any $x \in [x_{34}, b_4] \cup [b_4, x_{45}] \cup [x_{45},g_5]$ and any 
	$y \in [y_{23}, g_2] \cup [g_2, y_{12}] \cup [y_{12}, b_1]$, $\dist(x, y)>1$.   
\end{itemize}

\noindent
For any $y \in [b_1,y_{12}] \cup [y_{12}, g_2)$ and any $x \in [g_3, x_{34}] \cup [x_{34}, b_4] \cup [b_4, x_{45}] \cup [x_{45}, y_5]$ we have that $\dist(x,y) \geq \dist(g_3, y)>1$ by argument in Claim 2. 
Thus, to prove Claim 3, we can restrict ourselves to $y \in [g_2, y_{23}]$. By symmetry, we can also restrict 
to $x \in [x_{34}, b_4]$, for which it is enough to prove that $\dist(y_{23}, x_{34})>1$. One can easily convince oneself that $\dist(x_{12}, x_{23})<\dist(x_{23}, x_{34})$, hence 
$\dist(y_{23}, x_{34})=\sqrt{\dist(y_{23}, x_{23})^2+\dist(x_{23},x_{34})^2} > \sqrt{\dist(y_{23}, x_{23})^2+\dist(x_{12},x_{23})^2} = \dist(x_{12}, y_{23})=\sqrt{1+\epsilon^2}$. 
Therefore Claim 3 holds and we are done with joining two hexagons by an edge. Moreover, it is not hard to see that the arguments in Claims 1-3 extend to any collection of edge-adjacent hexagons, i.e. to a basic hexagonal
caterpillar with attached pendant vertices.

\begin{figure}[!h]
      	\begin{subfigure}[t]{\textwidth}
    		\centering
       	\includegraphics[scale=0.5]{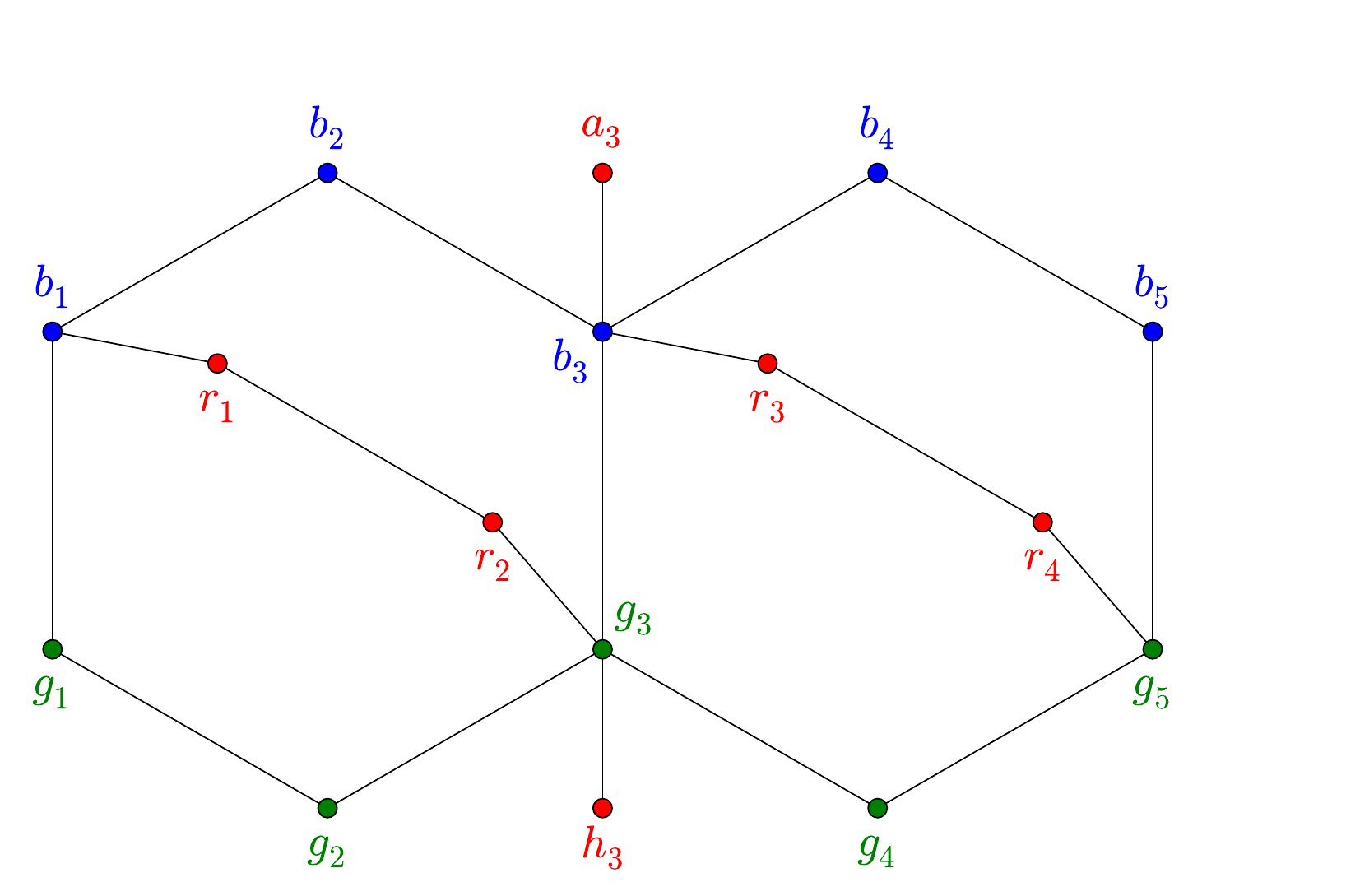}
       	\caption{Two hexagons joined along an edge}
		\label{fig:2hex_edge}
	\end{subfigure}
	
	\begin{subfigure}[t]{\textwidth}
		\centering
       	\includegraphics[scale=1]{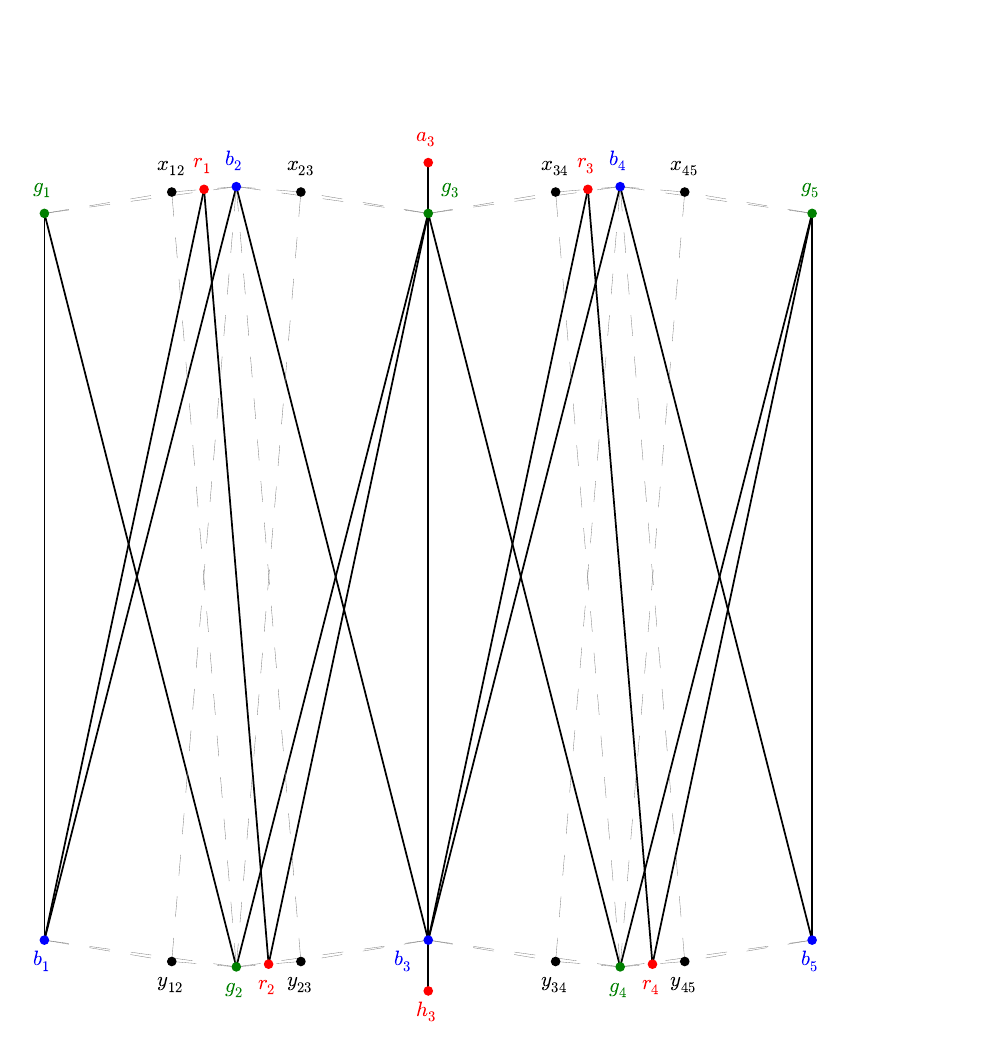}
              \caption{The UDG-representation of two hexagons joined along an edge}
		\label{fig:2hex_edge_representation}
	\end{subfigure}

    	\caption{}
    	\label{2hex_edge_representation}
\end{figure}

\begin{figure}[!h]
      	\begin{subfigure}[t]{\textwidth}
    		\centering
       	\includegraphics[scale=0.65]{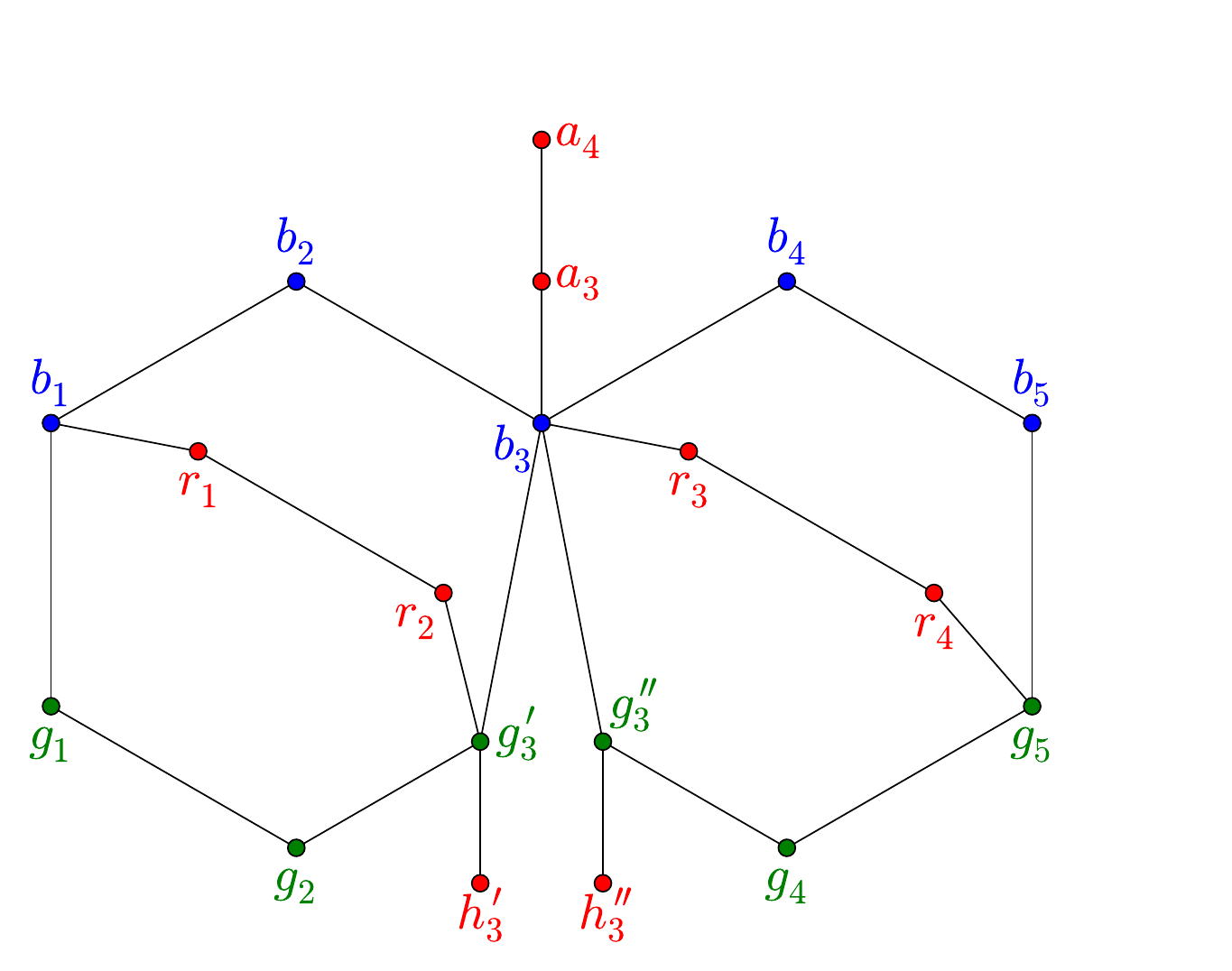}
       	\caption{Two hexagons sharing a vertex}
		\label{fig:2hex_vert}
	\end{subfigure}
	
	\begin{subfigure}[t]{\textwidth}
		\centering
       	\includegraphics[scale=1.22]{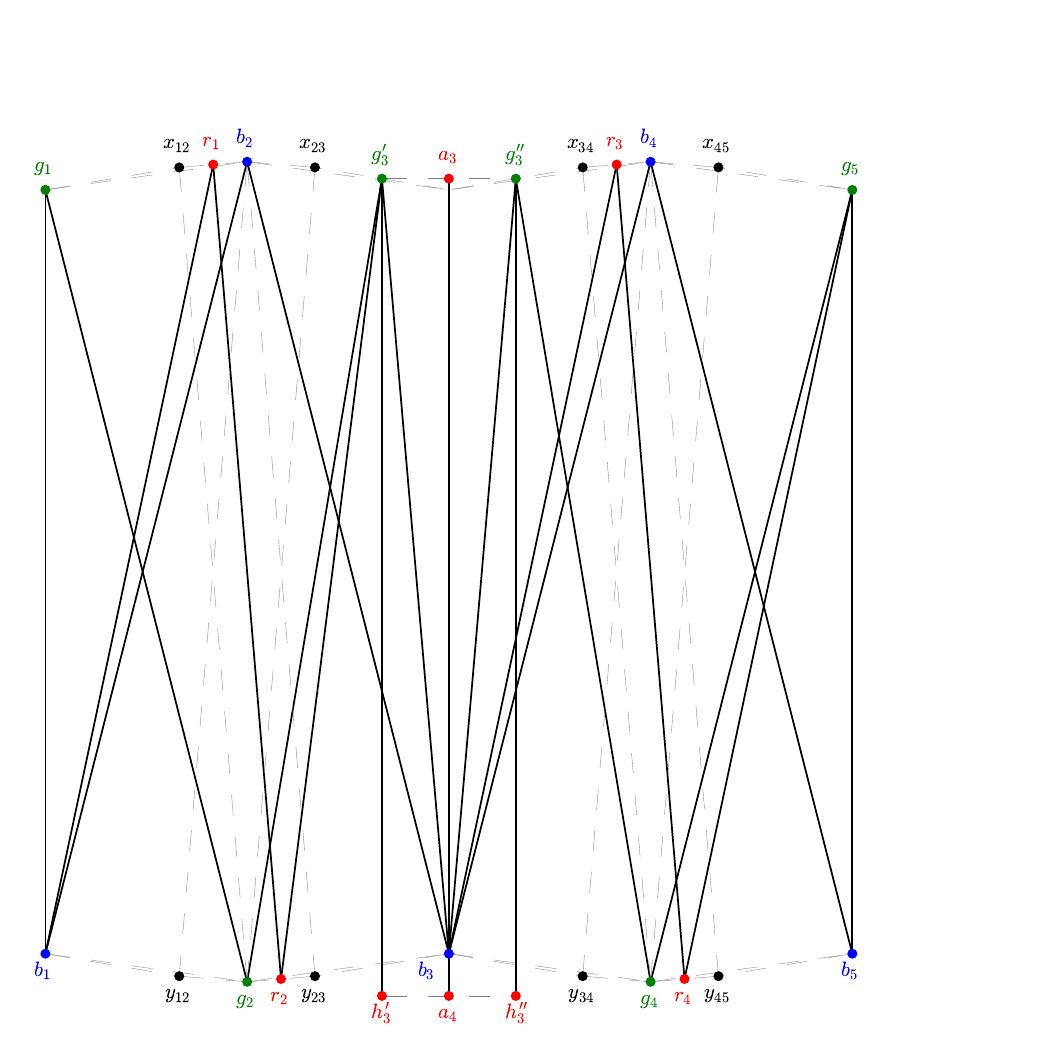}
              \caption{The UDG-representation of two hexagons sharing a vertex}
		\label{fig:2hex_vert_representation}
	\end{subfigure}

    	\caption{}
    	\label{2hex_edge_representation2}
\end{figure}

\vskip7px \noindent
\textbf{Two hexagons sharing a vertex.}
Now we will show how to  represent $G^*$, where $G$ consists of two $C_{6,1}$ sharing a vertex
(see Figure~\ref{fig:2hex_vert}).
The representation is obtained from the above representation for two hexagons 
sharing an edge, but we replace the vertex $g_3$ by two vertices $g_3'$ and $g_3''$ 
(see Figure \ref{fig:2hex_vert_representation}) such that:

\begin{itemize}
	\item $g_3'$ is the midpoint of $[g_3, x_{23}]$ and $g_3''$ is the midpoint of $[g_3, x_{34}]$. 
\end{itemize} 

\noindent
To prove that the points $g_3'$ and $g_3''$ have proper distances in the two adjacent hexagons, and in a chain of hexagons sharing a vertex or an edge, we will show the following

\begin{itemize}[leftmargin=*]
	\item[] \textbf{Claim 4.} Denote the midpoints of $[b_1, y_{12}]$, $[y_{12}, g_2]$, $[y_{34}, g_4]$, 
	$[y_{45}, b_5]$ by $b_1'', R_2, R_4, b_5'$, respectively. Then, for 
	$x \in \{b_1, b_1'', R_2, R_4, g_4, r_4, b_5',b_5\}$ we have $\dist(g_3', x)>1$ and for 
	$x \in \{g_2, r_2, b_3\}$ we have $\dist (g_3', x)<1$.    
\end{itemize}

\noindent
The proof of Claim 4 is given in Appendix \ref{app:claim4}. We notice that the proof for distances from $g_3'$ to $R_2, R_4, b_1''$ and $b_5'$, ensures that $g_3'$ has correct adjacencies with respect to all possible choices of direction of diagonals in hexagons and with possibility of having further hexagons adjacent at a single vertex $g_1$ or $g_5$. Finally, we place the points $\{h_3', h_3'', a_3, a_4\}$ in the plane as follows:

\begin{itemize}
\item $a_3$ is the midpoint of $[g_3', g_3'']$ and $h_3', a_4$, $h_3''$ are distance 1 below the points $g_3', a_3$ and $g_3''$, respectively.  
\end{itemize} 

It is clear that each of $h_3', a_4, h_3''$ is distance more than 1 away from all the vertices in the upper part except $g_3', a_3, g_3''$, respectively. Hence we only need to verify
the distances involving $a_3$. Observe that for all $y \in [b_1, y_{12}] \cup [y_{12}, g_2]$ we have 
$\dist(a_3, y)>\dist(g_3,y) \geq 1$. Also $\dist(a_3, b_3)<\dist(g_3',b_3)<1$. 
Hence, it is enough to show that $\dist(a_3, r_2)>1$. 
The proof of the latter fact can be found in Appendix \ref{app:dist_a3_r2}.

\vskip7px \noindent
\textbf{Connecting a chain of hexagons with a lobster.}
To finish the proof, we will show how a chain of hexagons can be attached to a lobster. 
An example is pictured in Figure~\ref{hex_cat_representation}.
To attach hexagon to a lobster at vertex $g_7$, we use the representation of hexagon obtained for joining hexagons at one vertex, i.e. 
we use point $b_7'$ with an attached pending vertex and point $g_7$ with a leg of size 2 attached exactly as 
in construction of hexagon joined to another hexagon at a vertex. This ensures that the first leg of lobster 
is attached correctly. 
Then we use the construction of lobster obtained in Theorem~\ref{basiclobster}.
In Theorem~\ref{basiclobster}, the lobster was uniquelly determined by a parameter $\mm$. The distance between two inner lines $L_2$ and $L_3$ was $\sqrt{1-\mm^2}$. Here, we choose $\mm$ to be such that this distance
is equal to $\dist(g_1, b_1)$. For the record, as we noted before, $1-10\epsilon^2 \leq \dist(g_1, b_1) \leq 1-9 \epsilon^2$ implies that  $(1-10\epsilon^2)^2 \leq 1-\mm^2 \leq (1-9\epsilon^2)^2$. Expanding, one can estimate that $1-25 \epsilon^2 \leq (1-10\epsilon^2)^2$ and for $\epsilon<1/9$, we can obtain the estimate $(1-9\epsilon^2)^2 \leq 1-16\epsilon^2$. Hence it follows that $4\epsilon \leq \mm \leq 5\epsilon$, which is roughly represented as the spacing between the lobster legs in the Figure~\ref{fig:hex_cat_representation}. It easily follows that the inner vertices of lobster are more than 1 away from any inner (not belonging to lobster) vertices of any hexagon. 
This completes the proof that for any basic graph $G \in \mathcal{X}$, $G^*$ is representable as a UDG. 

\end{proof}
\begin{figure}[!h]
      	\begin{subfigure}[t]{\textwidth}
    		\centering
       	\includegraphics[scale=0.82]{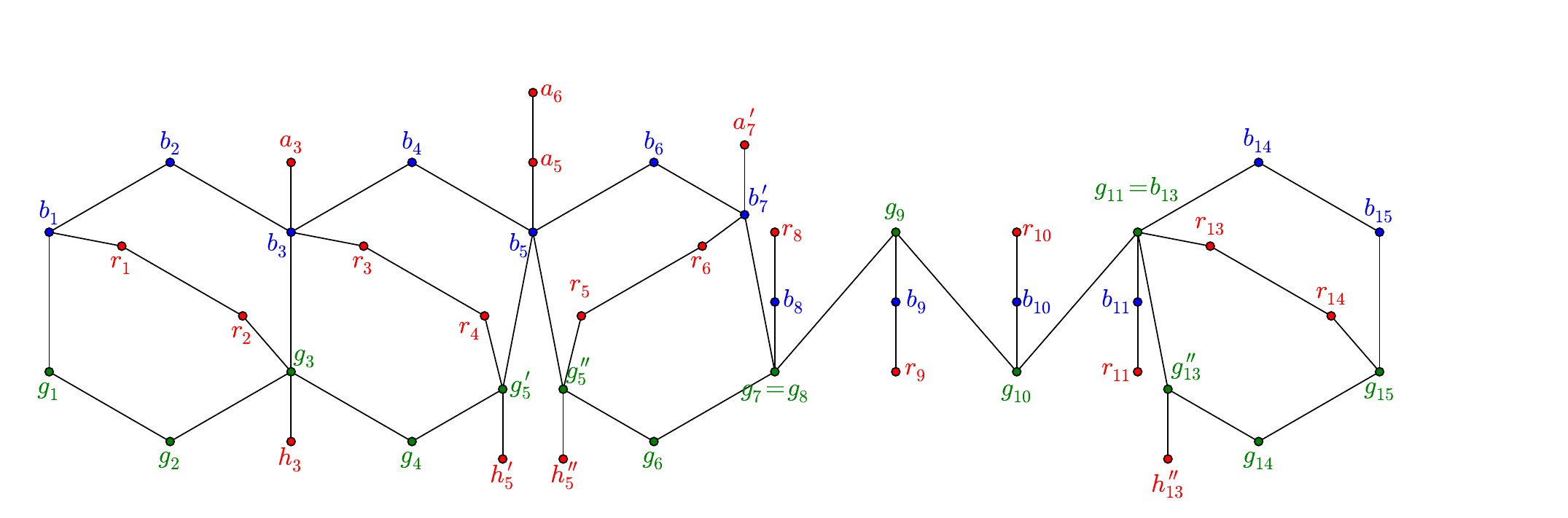}
       	\caption{Hexagonal caterpillar}
		\label{fig:hex_cat}
	\end{subfigure}
	
	\begin{subfigure}[t]{\textwidth}
		\centering
       	\includegraphics[scale=0.82]{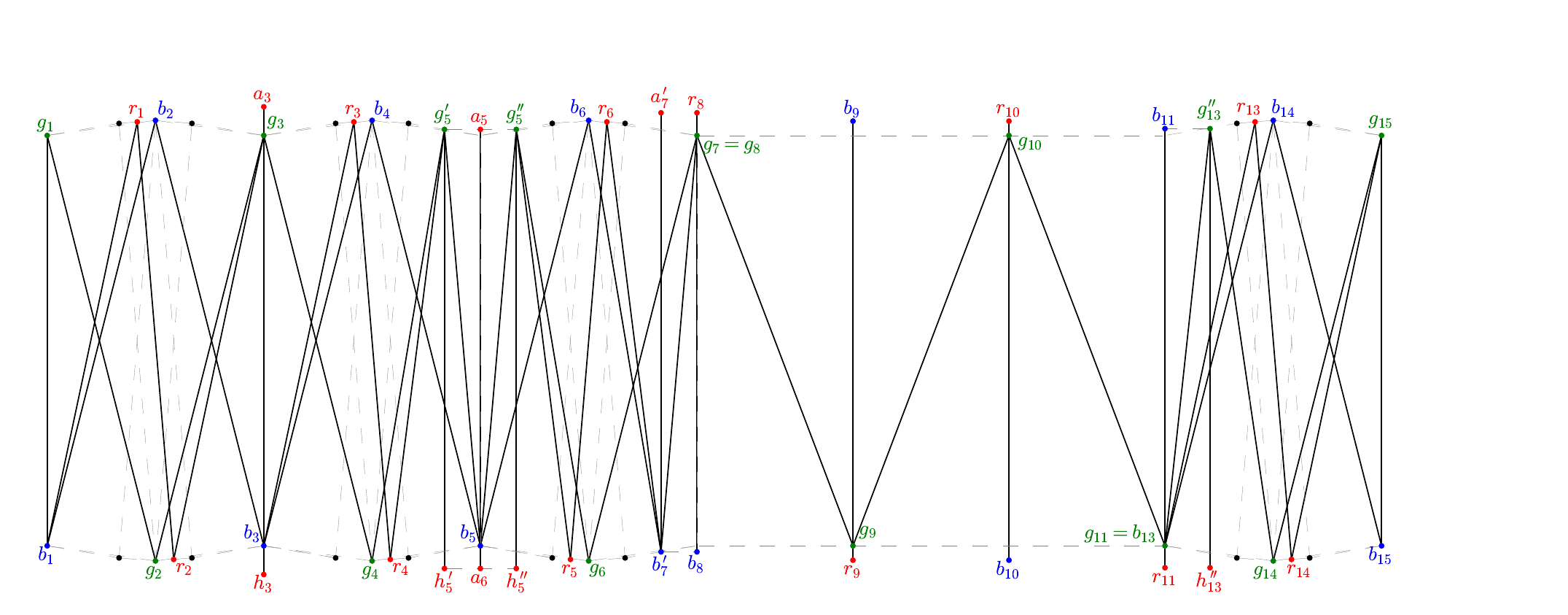}
              \caption{The UDG-representation of hexagonal caterpillar}
		\label{fig:hex_cat_representation}
	\end{subfigure}

    	\caption{}
    	\label{hex_cat_representation}
\end{figure}

In the following theorem we prove several properties of representation of basic graphs that
are important for representation of general graphs.

\begin{theorem}[Properties of basic graph representation]\label{lem:basic_rep}
	Let $G=(U,W,E)_c$ be a basic $C_4^*$-free co-bipartite UDG with $n=|V(G)|$. 
	Then for every positive $\epsilon < \min\{\frac{1}{15n}, \frac{1}{128}\}$ there exist
	$\Delta, q, r$ with $0 < \Delta < \frac{1}{3}$, $0 < q \leq r < 6$, and
	a UDG-representation $f: V(G) \rightarrow \mathbb{R}^2$ of $G$ with the following properties:
	\begin{enumerate}
		\item[(1)] $f(U) \subseteq D_1$ and
		$f(W) \subseteq D_2$,
		where $D_1 = [0, \Delta] \times [\sh, \sh]$,
		$D_2 = [0, \Delta] \times [1-\sh, 1+\sh]$
		with $\sh = r\epsilon^2$; 
		
		\item[(2)] For any vertices $a \in U$ and $b \in W$ either $\dist_f(a,b) = 1$ or 
		$|\dist_f(a,b)-1| \geq q\epsilon^2$;

		\item[(3)] For every parallel edge $ab$ of $G$, $\dist_f(a, b) = 1$. Moreover, 	
		$\dist_f(a, c) \neq 1$ and $\dist_f(c, b) \neq 1$ for any vertex $c \in V(G)$ different 	
		form $a$ and $b$.
			
	\end{enumerate}
\end{theorem}
\begin{proof}
	Let $f$ be a representation obtained in Theorem \ref{th:repBasicGraphs}. To the assumption that $\epsilon<\frac{1}{128}$ made in the theorem, we also add $\epsilon<\frac{1}{15n}$ to ensure that the representation lies in the strip of length $\Delta<\frac{1}{3}$. One can obtain such estimate by noting that the distance in $x$-coordinate between two consecutive points is less than $5\epsilon$, hence, the $n$ points will fit in the strip of length $5n\epsilon<\frac{1}{3}$. 
Observe that the shortest distance in $y$-coordinate between two points from different parts is obtained by $\dist_f(g_1, b_1)$ and it is at least $1-10\epsilon^2$. Also observe that every point has at least 1 neighbour in the other part, i.e. distance at most 1 from some point in another part. From these two observations, we can conclude that all the points lie in two strips of width $10\epsilon^2$ which are distance $1-10\epsilon^2$ away from each other. Hence, it follows that $r=5$ satisfies the conditions. From the proof of the theorem it is also not hard to see that we can take $q=\frac{1}{64}$. Finally, notice that every parallel edge satisfies property (3), so we are done.
\end{proof}

\subsubsection{Representation of general graphs}\label{sss:general}

Let $G=(U,W,E)_c$ be an arbitrary graph from $\mathcal{X}^*$ and let $H$ be a basic graph in $\mathcal{X}^*$ such that $G$ is obtained from $H$ by duplicating some of its parallel edges.
In this section we show how to extend a representation of $H$ described in the previous section
to a representation of $G$. 
We also prove that the resulting representation possesses certain properties, that will be
important in Section \ref{ss:comp}.

Let $e_i = a_i b_i, i = 1, \ldots, s$ be parallel edges of $H$ that have twins in $G$. 
For $i = 1, \ldots, s$ let
$e_i^j = a_i^j b_i^j, j = 1, \ldots, k_i$,
be twin edges of $e_i$ in $G$, and let $k = k_1 + \cdots + k_s$. We say that vertices in $V(G) \setminus V(H)$
are \textit{new} vertices.
For convenience we let $a_i^0 = a_i$ and $b_i^0 = b_i$.
Let $h$ be a representation of $H$ with
chosen positive $\epsilon < \min \left\{ \frac{1}{15|V(H)|}, \frac{1}{128}\right \}$ and 
parameters $\Delta$, $q$ and $r$ guaranteed by Theorem \ref{lem:basic_rep}.

First we define an extension $g: V(G) \rightarrow \mathbb{R}^2$ of $h$ and then show that $g$ is a representation of $G$.
To define $g$ we choose $t=\frac{q}{64\sqrt{r}} > 0$ and let $t_1 = \frac{t}{k}$. 
Since $g$ is an extension of $h$, it maps all vertices of $H$ to the same points as $h$ does,
that is $g(x) = h(x)$ for every $x \in V(H)$. Further, we define mapping of new vertices. 
Informally, for the edge $e_i = a_i b_i$ we place $a_i^j, b_i^j, j = 1, \ldots, k_i$ in the plane in such a way
that $a_i, b_i, a_i^{k_i}, b_i^{k_i}$ form a ``narrow'' rectangle with $[a_i, b_i]$ and $[a_i^{k_i}, b_i^{k_i}]$ 
being parallel
sides, and $[a_i^j, b_i^j], j = 1, \ldots, k_i-1$ are segments parallel to $[a_i, b_i]$ and evenly spaced within
the rectangle. Formally, for every $i=1, \ldots, s$ and $j=1, \ldots, k_i$ we define 
$g(a_i^j)$ and $g(b_i^j)$ in such a way that:
\begin{enumerate}
	\item the segment $[g(a_i^j),g(b_i^j)]$ is parallel to the segment $[g(a_i),g(b_i)]$;
	\item $\dist_g( a_i^j, b_i^j ) = 1$;
	\item $\dist_g( a_i, a_i^j ) = \dist_g( b_i, b_i^j )  = j \frac{t}{k} \epsilon = j t_1 \epsilon$;
	\item $\dist_g( a_i^j, a_i^{j+1} ) = \dist_g( b_i^j, b_i^{j+1} ) = \frac{t}{k} \epsilon = 
	t_1 \epsilon$
	(for $j=0, \ldots, k_i-1$);
	\item each of the segments $[g(a_i),g(a_i^j)]$ and $[g(b_i),g(b_i^j)]$ is perpendicular to
	the segment $[g(a_i), g(b_i)]$;
	\item (for definiteness) $g(a_i^j)$ and $g(b_i^j)$ have larger $x$-coordinate than $g(a_i)$
	and $g(b_i)$, respectively.
\end{enumerate}

%
%

\noindent
To prove that $g$ is a UDG-representation of $G$ we need several auxiliary statements.

\begin{lemma}
Suppose $ab$ is a parallel edge. Then the angle $\alpha$ between $[g(a),g(b)]$ and the vertical line, 
satisfies $\sin(\alpha) \leq 2 \sqrt{r} \epsilon$. 
\end{lemma}
\begin{proof}
Let $g(a)=(x,y)$, $g(b)=(x', y')$. As $\dist_g(a,b)$=1, we have $\sin(\alpha)=|x-x'|$. Notice that since $a$ and $b$
are in different parts, we get $|y-y'| \geq 1-2r\epsilon^2$. From this it follows that
$$
	\sin(\alpha)=|x-x'|=\sqrt{1-|y-y'|^2} \leq \sqrt{1-1+4r\epsilon^2} \leq 2\sqrt{r}\epsilon.
$$ 
\end{proof}

\begin{lemma}\label{twinlemma}
Let $a, a' \in V(G)$ be twins, and
let $g(a)=(x,y)$, $g(a')=(x',y')$. Then $|x-x'| \leq t \epsilon$ and $|y-y'| \leq 2t\sqrt{r}\epsilon^2$.  
\end{lemma}

\begin{proof}
Clearly, the first inequality holds because $|x-x'| \leq \dist_g(a, a') \leq t\epsilon$. 
Now, $|y-y'|=\dist_g(a,a') \sin(\alpha)$ where $\alpha$ is the angle between segment $[a,a']$ and the
horizontal line. This angle is equal to the angle of the parallel edge and vertical line, thus, by previous lemma
we can deduce that $\sin(\alpha) \leq 2\sqrt{r}\epsilon$. Hence, 
$$
	|y-y'| =\dist_g(a,a') \sin(\alpha) \leq 2t\sqrt{r} \epsilon^2.
$$
\end{proof}

The following is an important lemma which will be used for proving that
the defined map $g$ is indeed a UDG-representation of $G$.  

\begin{lemma} \label{maintwin}
Suppose $a, b$ are two vertices in different parts of $G$ with $|\dist_g(a,b) - 1| \geq q \epsilon^2$. 
Let $a'$ be either a twin of $a$ or $a'=a$ and let $b'$ be either a twin of $b$ or $b'=b$. 
Then $\dist_g(a',b') > 1$ iff $\dist_g(a,b) > 1$ and $|\dist_g(a', b')-1| \geq q\epsilon^2/2$. 
\end{lemma}
\begin{proof}
Let $a=(x,y)$, $a'=(x',y')$, $b=(z,u)$, $b'=(z', u')$. To get the bounds of the distance $\dist_g(a',b')$,
we will compare projections of $[g(a'),g(b')]$ and $[g(a),g(b)]$ onto $x$ and $y$ axes and then apply the Pythagorean theorem.  

First of all, triangle inequalities can be used to obtain that $|x-z| \leq |x-x'| + |x'-z'| + |z-z'|$ and $|x'-z'| \leq |x'-x| + |x-z|+|z-z'|$. Further, by Lemma~\ref{twinlemma}, we have $|x-x'| \leq t\epsilon$ and $|z-z'| \leq t \epsilon$, 
and hence
\begin{equation} \label{xz} 
|x-z|-2t\epsilon \leq |x'-z'| \leq |x-z| + 2t \epsilon .
\end{equation}

\noindent
Similarly, projecting onto $y$-axis, from triangle inequalities we obtain $|y'-u'| \leq |y'-y|+|y-u|+|u-u'|$ and
$|y-u| \leq |y-y'|+|y'-u'|+|u'-u|$. Also from Lemma~\ref{twinlemma} we know that $|y-y'| \leq 2t \sqrt{r} \epsilon^2$ 
and $|u-u'| \leq 2t \sqrt{r} \epsilon^2$. This gives us   
\begin{equation} \label{yu} 
|y-u|-4t \sqrt{r} \epsilon^2 \leq |y'-u'| \leq |y-u| + 4t \sqrt{r} \epsilon^2.
\end{equation}

\noindent
Now, we split our analysis into two cases.

\vskip7px
\noindent
\textbf{Case 1.} $|x-z|> 4 \sqrt{r} \epsilon$. 

\noindent
Since $|y-u| \geq 1-2r\epsilon^2$, we can easily obtain that 
$$
	\dist_g(a,b)^2= |y-u|^2+|x-z|^2 > (1-2r\epsilon^2)^2 + 
	(4\sqrt{r}\epsilon)^2=1+12r\epsilon^2+4r^2\epsilon^4>1.
$$ 
Hence, in this case our aim is to prove
that $\dist_g(a', b')>1$ and $|\dist_g(a',b')-1| \geq q\epsilon^2/2$, i.e. we have to prove that 
$\dist_g(a',b') \geq 1+q\epsilon^2/2$. For this, we use the estimates of the projections 
\begin{align*}
	|x'-z'| \geq & |x-z|-2t\epsilon \geq 4\sqrt{r}\epsilon-2t\epsilon; \\
	|y'-u'| \geq & |y-u|-4t\sqrt{r} \epsilon^2 \geq 1-2r\epsilon^2-4t\sqrt{r} \epsilon^2.
\end{align*}
As $q \leq r$, we have $t=\frac{q}{64\sqrt{r}} \leq \frac{\sqrt{r}}{2}$, and placing this upper bound of $t$ into 
the above inequalities we obtain 
\begin{align*}
	|x'-z'| \geq & 4\sqrt{r}\epsilon-\sqrt{r}\epsilon=3\sqrt{r}\epsilon; \\
	|y'-u'| \geq & 1 - 2 r\epsilon^2 - 2r \epsilon^2=1-4r\epsilon^2. 
\end{align*}
Applying the Pythagorean theorem, we obtain
$$
	\dist_g(a', b')^2 \geq  (3\sqrt{r}\epsilon)^2+ (1-4r\epsilon^2)^2 = 9r\epsilon^2+1-8r\epsilon^2+
	16r^2\epsilon^4 \geq 1+r\epsilon^2+ r^2\epsilon^4/4=(1+r\epsilon^2/2)^2.
$$
Hence, $\dist_g(a',b') \geq 1+r\epsilon^2/2$, and as $r \geq q$, we obtain the required inequality 
$\dist_g(a',b') \geq 1+q\epsilon^2/2$. 

\vskip7px
\noindent
\textbf{Case 2.} $|x-z| \leq 4 \sqrt{r} \epsilon$. 

\noindent
First, consider 
$$
	||x'-z'|^2-|x-z|^2|=||x'-z'|-|x-z|| \times ||x'-z'|+|x-z||.
$$ 
By (\ref{xz}) we have that the first term $||x'-z'|-|x-z||$  is upper bounded by $2t\epsilon$, and the second by 
$$
	|x'-z'|+|x-z| \leq 2|x-z|+2t\epsilon \leq 8\sqrt{r}\epsilon+2t\epsilon \leq 10\sqrt{r} \epsilon,
$$ 
where the latter inequality follows from the fact that $t=\frac{q}{64\sqrt{r}} \leq \frac{r}{64\sqrt{r}} \leq \sqrt{r}$.
This gives us an upper bound

\begin{equation} \label{xzeq}
	||x'-z'|^2-|x-z|^2| \leq 2t \epsilon \times 10 \sqrt{r} \epsilon=20t\sqrt{r} \epsilon^2.
\end{equation}

\noindent
Similarly, consider 
$$
	||y'-u'|^2-|y-u|^2|=||y'-u'|-|y-u|| \times ||y'-u'|+|y-u||.
$$
By (\ref{yu}), we have $||y'-u'|-|y-u|| \leq 4t\sqrt{r}\epsilon^2$ and 
$$
	|y'-u'|+|y-u| \leq 2|y-u|+ 4t\sqrt{r}\epsilon^2 \leq 2(1+2r\epsilon^2)+4t\sqrt{r}\epsilon^2 \leq 3.
$$ 
This gives us an upper bound 

\begin{equation} \label{yueq}
	||y'-u'|^2-|y-u|^2| \leq 4t\sqrt{r}\epsilon^2 \times 3 \leq 12t\sqrt{r}\epsilon^2.
\end{equation}

\noindent
Adding (\ref{xzeq}) and (\ref{yueq}), we get 
$$
	|\dist_g(a',b')^2-\dist_g(a,b)^2| \leq 32t\sqrt{r}\epsilon^2.
$$
One can easily check, for example by projecting to $y$-axis, that $\dist_g(a,b)+\dist_g(a',b')\geq 1$. 
Hence, 
$$
	|\dist_g(a',b')-\dist_g(a,b)| \leq 32t\sqrt{r}\epsilon^2/(\dist_g(a,b)+\dist_g(a',b)) \leq 32t\sqrt{r}\epsilon^2.
$$
Inserting $t=\frac{q}{64\sqrt{r}}$ we have 
$$
	|\dist_g(a',b')-\dist_g(a,b)| \leq q\epsilon^2/2.
$$ 
As $|\dist_g(a,b)-1| \geq q\epsilon^2$, it follows that 
$|\dist_g(a',b')-1| \geq q\epsilon^2/2$ and that $\dist_g(a',b')>1$ iff $\dist_g(a,b)>1$.
This completes the proof of the lemma.

\end{proof}

\noindent
We are now ready to prove the main results of this section.

\begin{theorem} \label{basicgeneral}
Suppose $h$ is a UDG-representation of the basic graph $H$ which satisfies the conditions
outlined in Theorem~\ref{lem:basic_rep}. Let $\epsilon, r, q, \Delta$ be as in Theorem~\ref{lem:basic_rep}. 
Then $g$ is a UDG-representation of $G=(U,W,E)_c$.
Moreover, the representation $g$ satisfies the following conditions:
\begin{itemize}
\item[(1)] $g(U) \subseteq D_1$, $g(W) \subseteq D_2$ where $D_1=[0, \Delta'] \times [-\sh,\sh]$, 
$D_2=[0,\Delta'] \times [1-\sh, 1+\sh]$ with $\sh=r'\epsilon^2$, $r'=2r$, and $\Delta'=\Delta+\epsilon$. 
\item[(2)]
For every $u \in U$, $w \in W$, we have either $\dist_g(u, w)=1$ or $|\dist_g(u, w)-1| \geq q'\epsilon^2$ for
$q'=\frac{q^2}{64^2r\times4k^2}$.
\end{itemize}
\end{theorem}

\begin{proof}
The condition (2) is satisfied for all the vertices of $H$ by Theorem~\ref{lem:basic_rep}.  
Further, by Lemma~\ref{maintwin}, the condition is satisfied between any new vertex and 
a vertex of $H$, or between two new vertices that are twins of vertices in different parallel edges. 
So we only need to consider pairs of new vertices $a_i^l$, $b_i^m$, $l, m \in \{1,2, \ldots, k_i\}$ 
that are twins to two vertices of the same parallel edge $a_ib_i$. In this case, clearly, the distances 
that are not equal to $1$ are at least 
$$\sqrt{1+\left(\frac{t}{k}\epsilon\right)^2} \geq 1+\frac{t^2}{4k^2}\epsilon^2 =1 + \frac{q^2}{64^2r\times4k^2}\epsilon^2.$$

The condition (1) is clearly satisfied for the representation $h$ of $H$, and by Lemma~\ref{twinlemma},
we can get out of the strip horizontally by at most $t\epsilon<\epsilon$ and vertically by at most $4t\sqrt{r}\epsilon^2<r\epsilon^2$.
This completes the proof. 
\end{proof}

As for any basic graph $G \in \mathcal{X}^*$ we have a UDG-representation satisfying the conditions of
Theorem~\ref{lem:basic_rep}, Theorem~\ref{basicgeneral} shows that every 
graph in $\mathcal{X}^*$ is a UDG. 
Moreover, the representation has several properties, that allow us to transform these 
UDG-representations to UDG-representations of the bipartite complements of these graphs, which we will do 
in the next section. For completeness we state the result for general graphs in $\mathcal{X}^*$. 

%

\begin{theorem} \label{c4representation}
Let $G=(U,W,E)_c$ be an $n$-vertex graph in $\mathcal{X}^*$.
Then for every sufficiently small $\sh$ there exists $\Delta' \in (0, 1/3)$, and a UDG-representation
$g$ of $G$ possessing the following properties:

\begin{itemize}
	\item[(1)] $g(U) \subseteq D_1$, $g(W) \subseteq D_2$, where $D_1=[0, \Delta'] \times [-\sh,\sh]$, 
	$D_2=[0,\Delta'] \times [1-\sh, 1+\sh]$. 
	\item[(2)] For every $u \in U$, $w \in W$, we have either $\dist_g(u, w)=1$ or $|\dist_g(u, w)-1| \geq q''\sh$,
	where $q''=\frac{1}{64^4 \times 200 n^2}$.
\end{itemize}
\end{theorem}   

\begin{proof}
Let $g$ be a UDG-representation of $G$ with parameters $\epsilon, r', q', \Delta',\sh$ guaranteed by 
Theorem \ref{basicgeneral}.
First we can assume that $\Delta' \in (0,1/3)$, which is true for every sufficiently small $\epsilon$,
since $\Delta' = \Delta+\epsilon$ and $\Delta \in (0,1/3)$.
Now for arbitrary $u \in U$ and $w \in W$, if $\dist_g(u, w) \neq 1$, then
$|\dist_g(u, w)-1| \geq q'\epsilon^2$ and from $\sh=r'\epsilon^2$ we derive

$$
	|\dist_g(u, w)-1| \geq \frac{q'}{r'} \sh = \frac{q^2}{64^2r\times4k^2 \times 2r} \sh \geq
	\frac{1}{64^4 \times 200 n^2} \sh = q'' \sh,
$$
where we take $q=\frac{1}{64}$ and $r=5$, which is eligible as shown in the proof of Theorem~\ref{lem:basic_rep}.
%
\end{proof}


\subsection{On bipartite self-complementarity of the class of co-bipartite UDGs}\label{ss:comp}
 
Notice that all the forbidden subgraphs (and, more generally, substructures) for co-bipartite unit disk graphs, 
which were revealed in Sections \ref{sec:tools} and \ref{sec:forb} are self-complementary in bipartite
sense, i.e. if $G$ is a bipartite graph and $G^*$ is a forbidden subgraph, then $\overline{G}$ is also 
a forbidden subgraph. This in turn motivates to explore, whether the class of co-bipartite UDGs 
is indeed self-complementary in bipartite sense. 
In this section, we show that if a UDG-representation of a co-bipartite graph $G^{*}$ satisfies certain
conditions, then it can be transformed into a UDG-representation of the graph $\overline{G}$. 
Loosely speaking, the conditions tells us that the parts of $G^*$ are mapped into two narrow strips
being distance approximately 1 away from each other.
In the next section we will apply this result to show that the bipartite complement of a $C_4^*$-free 
co-bipartite UDG is also UDG. This will settle the fact that $\mathcal{Z} = \overline{\mathcal{X}}$.

In this section we will often use polar coordinates. Let us recall that a point $(r,\alpha)_p$ in polar coordinates 
is a point $(r\cos(\alpha), r\sin(\alpha))$ in standard Cartesian coordinates. 
We begin by describing the transformation. For this we fix $0<\sh<\frac{1}{12}$ and $0<\Delta<\frac{1}{3}$ and
let $D_1=[0, \Delta] \times [-\sh, \sh] \subseteq \mathbb{R}^2$, 
$D_2=[0, \Delta] \times [1-\sh,1+ \sh] \subseteq \mathbb{R}^2$.
Let $D=D_1 \cup D_2$ be the domain where the points of the representation of $G^*$ lie. 
The transformation $\tau: D \rightarrow \mathbb{R}^2$ is defined as follows 
(see Figure \ref{fig:tau_transformation} for illustration):

$$
\text{for all $\alpha \in [0,\Delta]$ and $y \in [-\sh, \sh]$, define }
\begin{cases}
\tau(\alpha, y):=\left(\frac{1}{2}+y, -\frac{\pi}{2}+2\alpha\right)_p & \\
\tau(\alpha, 1+y):=\left(\frac{1}{2}-y, \frac{\pi}{2}+2\alpha\right)_p & 
\end{cases}
$$

Notice first that this transformation maps set of points on a horizontal line to a line through $(0,0)$.  
That is for a fixed $\alpha$ the points $D_1(\alpha)=\{(\alpha, y_1): y_1 \in [-\sh, \sh]\}$ and 
$D_2(\alpha)=\{(\alpha, 1+y_2): y_2 \in [-\sh, \sh]\}$ are mapped to the line 
$$
	L(\alpha)=\left\{\left(y,\frac{\pi}{2}+2\alpha\right)_p: y \geq 0\right\} \bigcup 
	\left\{\left(y,-\frac{\pi}{2}+2\alpha\right)_p: y > 0\right\}.
$$ 
To closer examine what happens on the line, take two points $A = (\alpha, y_1) \in D_1(\alpha)$ and 
$B=(\alpha, 1+y_2) \in D_2(\alpha)$ for some $y_1, y_2 \in [-\sh, \sh]$. 
Then $\dist(A, B)=1+y_2-y_1$ and $\dist(\tau(A), \tau(B))=\frac{1}{2}+y_1+\frac{1}{2}-y_2=1+y_1-y_2$. 
Therefore

\begin{itemize}
	\item[-] if $y_1-y_2=0$, then $\dist(A, B)=1$ and $\dist(\tau(A), \tau(B))=1$;
	\item[-] if $y_1-y_2=a>0$ then  $\dist(A, B)=1-a<1$ and $\dist(\tau(A), \tau(B))=1+a>1$;
	\item[-] if $y_1-y_2=-a<0$ then  $\dist(A, B)=1+a>1$ and $\dist(\tau(A), \tau(B))=1-a<1$.
\end{itemize}

\begin{figure}[!h]
      	\begin{subfigure}[t]{0.4\textwidth}
    		
       	\includegraphics[scale=1.2]{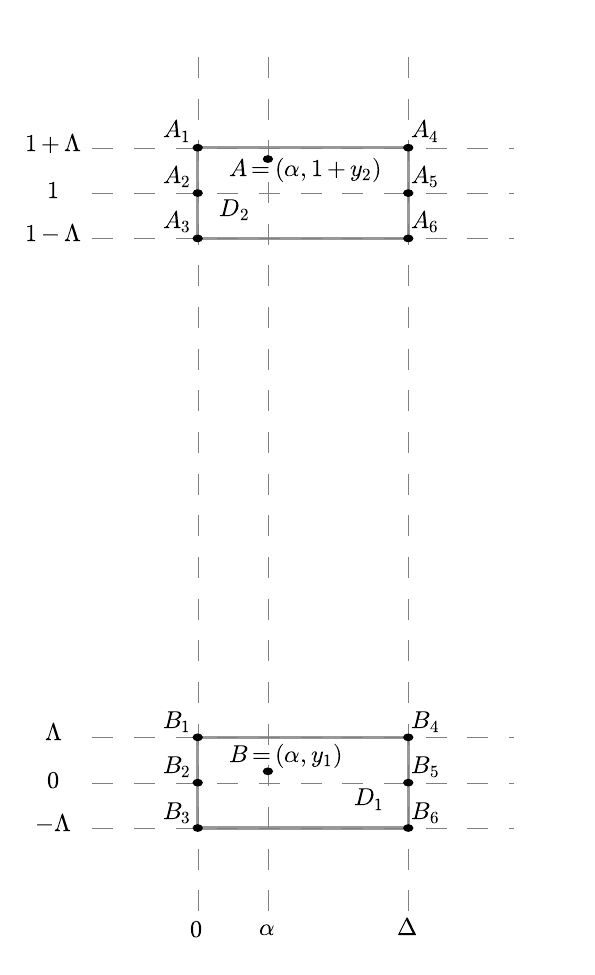}
       	\caption{}
		\label{fig:udg_reg_G_star}
	\end{subfigure}
	~
	\begin{subfigure}[t]{0.6\textwidth}
		
       	\includegraphics[scale=1.2]{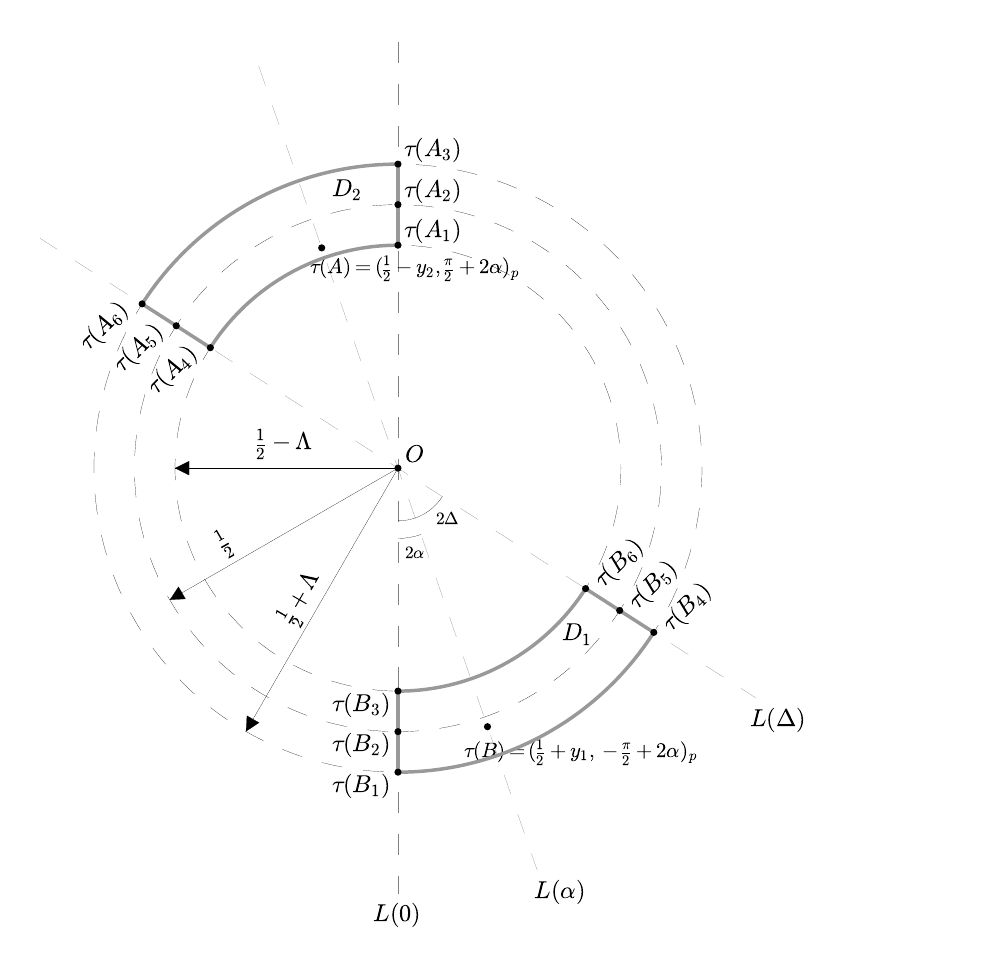}
              \caption{}
		\label{fig:udg_reg_G_compl}
	\end{subfigure}

    	\caption{Transformation $\tau$}
    	\label{fig:tau_transformation}
\end{figure}

\noindent
Thus, for two points $A \in D_1(\alpha)$, $B \in D_2(\alpha)$ on the same horizontal line but in different parts, transformation $\tau$ swaps the distances that are less than 1 with the distances that are greater than 1, i.e. $\dist(\tau(A), \tau(B))>1$ iff  $\dist(A, B)<1$ and $\dist(\tau(A), \tau(B))<1$ iff  $\dist(A, B)>1$. Further, one can easily see that both $\tau(D_1)$ and $\tau(D_2)$ have diameter less than 1. Thus, if we have a 
UDG-representation $f$ of some co-bipartite graph $G^*$ which lies on one horizontal line, i.e. 
$f(G^*) \subseteq D_1(\alpha) \cup D_2(\alpha)$ and avoids distances equal to one, then the map 
$\tau \comp f$ is a UDG-representation of $\overline{G}$. 

We would like to extend this argument to the whole set $D_1 \cup D_2$. However, not all the distances, between the points in different parts $D_1$ and $D_2$, which are less than 1 will be swapped with distances that are greater than 1 by map $\tau$. Nevertheless, in the lemma below we will show that the distances that are smaller than $1- 100\sh^2$ or greater than $1 + 100\sh^2$ are mapped to distances greater than 1 or smaller than 1, respectively.  Thus, if $G^*$ has a UDG-representation $g$ with $g(G^*) \subseteq D_1 \cup D_2$ 
such that no distance lies in the interval $[1-100\sh^2, 1+100\sh^2]$, then $\tau \comp g$ is a 
UDG-representation of $\overline{G}$. Furthermore, it is worth noting that the distances of size 1 can be avoided by appropriate scaling of the initial UDG-representation, as we will see later. 
Now we are ready to prove the main result of this section. 

\begin{lemma}\label{lem:transf}
Let $D_1$ and $D_2$ be as described above. Suppose $G^*$ admits a UDG-representation $g$, 
such that $g(V(G^*)) \subseteq D_1 \cup D_2$ and for all $x \in D_1 \cap g(V(G^*)), y \in D_2 \cap g(V(G^*))$, 
$\dist(x,y) \notin [1-100\sh^2, 1+100\sh^2]$. Then $\tau \comp g$ is a UDG-representation of $\overline{G}$.
\end{lemma}

\begin{proof}
We will prove the lemma by showing that for any two points $A=(\alpha, 1+a)$, $B=(\beta, b)$, with $\alpha, \beta \in [0, \Delta]$ and $a, b \in [-\sh, \sh]$ the following statement holds:
\begin{itemize}
	\item[($\star$)] if $\dist(A,B) < 1-100\sh^2$ or  $\dist(A,B) > 1+100\sh^2$, then $\dist(\tau(A), \tau(B)) >1$ 
	or $\dist(\tau(A), \tau(B)) <1$, respectively.   
\end{itemize}

\noindent
First we observe that it is enough to show the statement ($\star$) for all pairs $A, B$ as above with $\alpha=0$. Indeed, let $\alpha'=\min(\alpha, \beta)$ and $\beta'=\max(\alpha, \beta)$ and let $A'=(\alpha', 1+a)$, $B'=(\beta',b)$.
It is not hard to see that $\dist(A', B')=\dist(A, B)$ and $\dist(\tau(A'), \tau(B'))=\dist(\tau(A),\tau(B))$. 
Further, let $\beta''=\beta'-\alpha'$ and let $A''=(0, 1+a)$, $B''=(\beta'', b)$. Again, it is not hard to see that $\dist(A', B')= \dist(A'',B'')$ and $\dist(\tau(A'), \tau(B'))=\dist(\tau(A''),\tau(B''))$, because horizontal shifting by distance $\alpha''$ and rotating around the origin by angle $2\alpha''$ are both isometries of the plane. 
Thus, the pair $A,B$ satisfies ($\star$) iff the pair $A'', B''$ satisfies ($\star$). 
Hence from now onwards we will assume $A=(0, 1+a), B=(\beta, b)$, with 
$\beta \in [0, \Delta]$ and $a, b \in [-\sh, \sh]$.   

Consider a special point $\CC=(\beta, \cc)$ with 
$\cc=\cc(\beta, a)<1$ such that $\dist(A, \CC)=1$ 
(see Figure \ref{fig:specialpointsgen}). This point is an intersection of the vertical line going through $(\beta, 0)$ and a unit circle centered at $A$ and one can easily calculate that $\cc=a+1-\sqrt{1-\beta^2}$. 
The importance of the point $\CC$ is that the distance between $A$ and a point 
$B=(\beta, b)$ is greater or smaller than 1 depending on whether $B$ is below (i.e. $b < \cc$)
or above (i.e. $b > \cc$) $\CC$, respectively. 

Similarly, consider a special point $\FF$ which lies on the ray 
$R(\beta)=\left\{\left(r,-\frac{\pi}{2}+2\beta\right)_p: r \geq 0\right\} \subseteq L(\beta)$ and is distance 1 away from $\tau(A)=\left(0, \frac{1}{2}-a\right)$. Such a point exists and is unique because the unit cycle centered at $\left(0, \frac{1}{2}-a\right)$ 
contains the origin O - the endpoint of the half-line. We denote the distance 
$\ff=\ff(\beta,a)=\dist(O,\FF)-\frac{1}{2}$.
The importance of the point $\FF$ is that it divides the ray $R(\beta)$ into two segments:
the points $\left(\frac{1}{2}+b,-\frac{\pi}{2}+2\beta\right)_p$ have distance less or more than 1 from $\tau(A)$ depending on whether $b<\ff$ or $b>\ff$, respectively. 
Let $\FF'=(\beta, \ff)$. 
As $\left(\frac{1}{2}+b,-\frac{\pi}{2}+2\beta\right)_p=\tau(\beta, b)$ for any $b \in [-\sh, \sh]$, we deduce that
$\dist(\tau(B), \tau(A))$ is greater or smaller than 1 depending on whether $B$
lies above or below the point $\FF'$, respectively.  
 
From the above discussion we deduce the following important criterion. 
If $b>\max\{\ff(\beta,a), \cc(\beta,a)\}$, then $\dist(A, B)<1$, and $\dist(\tau(A), \tau(B))>1$. Similarly, if $b<\min\{\ff(\beta,a), \cc(\beta,a)\}$, then $\dist(A, B)>1$, and $\dist(\tau(A), \tau(B))<1$. So, in both cases $B=(\beta, b)$ satisfies ($\star$). 
However, if $b \in [\min\{\ff, \cc\}, \max\{\ff, \cc\}]$, 
then the distances are not inverted by the map $\tau$, i.e. either $\dist(A,B)$ and $\dist(\tau(A), \tau(B))$ are both smaller or equal to 1 or both greater than 1. 
In what follows, we will show that in this case $\dist(A,B) \in [1-100\sh^2, 1+100\sh^2]$. 
In order to do so, we will estimate values of $\ff$ and $\cc$ more precisely. 

As we observed earlier $\cc=a+1-\sqrt{1-\beta^2}$. We can approximate the root part of the equation as follows: $1-\frac{\beta^2}{2}-\frac{\beta^4}{2} \leq \sqrt{1-\beta^2} \leq 1-\frac{\beta^2}{2}$. Hence, 

\begin{align*}
	a+\frac{\beta^2}{2} \leq \cc \leq a+\frac{\beta^2}{2}+\frac{\beta^4}{2}. 
\end{align*}

\noindent
For finding reasonable bounds of function $\ff$ the arguments are more involved, and we moved them
to Appendix \ref{app:LemTrans}, where we show
\begin{align*}
	a+ \frac{\beta^2}{2}-\frac{7\beta^4}{6}-2a^2\beta^2 \leq \ff \leq a+\frac{\beta^2}{2}+\frac{\beta^4}{2}.
\end{align*}

Having obtained these estimates, we are now ready to say something about \textit{non-invertible} points, that is the points $B=(\beta, b) \in D_1$ such that $\dist(A,B)$ and $\dist(\tau(A), \tau(B))$ are both greater or both smaller than 1. As we observed above, such $B$ must lie between $\FF'$ and $\CC$, i.e. must have 
$b \in [\min\{\ff,\cc\}, \max\{\ff,\cc\}]$. 
Further we consider two cases with respect to the value of $\beta$.

\begin{enumerate}
\item $\beta>\sqrt{6\sh}$. 
The obtained bounds on the functions $\ff$ and $\cc$ 
imply that 
\begin{align*}
	\min\{\ff(\beta,a), \cc(\beta,a)\} \geq & a+\frac{\beta^2}{2} -\frac{7\beta^4}{6}-2a^2 \beta^2
	\geq -\sh+\beta^2\left(\frac{1}{2}-\frac{7\beta^2}{6}-2\sh^2\right) \\ 
	\geq & -\sh+\beta^2\left(\frac{1}{2}-\frac{7}{6} \times \frac{1}{9}-\frac{2}{12^2}\right)
	\geq -\sh+\frac{\beta^2}{3} \\
	> & -\sh+\frac{\left( \sqrt{6\sh} \right)^2}{3} 
	= -\sh+2\sh =\sh. 
\end{align*}
Hence there is no point in 
$[\min\{\ff,\cc\}, \max\{\ff,\cc\}] \cap [-\sh, \sh]$,
which means that every point 
$B \in \beta \times [-\sh, \sh]$ satisfies ($\star$): $\dist(A,B)>1$ but $\dist(\tau (A), \tau(B))<1$. 

\item $\beta \leq \sqrt{6\sh}$. 
In this region, we have: 
\begin{align*}
|\ff(\beta, a) - \cc(\beta,a)| \leq & a+\frac{\beta^2}{2}+\frac{\beta^4}{2}-
					\left( a+\frac{\beta^2}{2}-\frac{7\beta^4}{6}-2a^2\beta^2 \right) \\
= &\frac{5\beta^4}{3}+2a^2\beta^2 \leq \frac{5\left( \sqrt{6\sh} \right)^4}{3}+ \frac{2\sh^2}{9}\\
= & \left(60+\frac{2}{9}\right)\sh^2 \leq 100 \sh^2
\end{align*}
If $B$ is non-invertible, then it satisfies 
$\min\{\ff,\cc\} \leq b \leq \max\{\ff,\cc\}$ and we have 
$\dist(B,\CC) \leq |\ff-\cc| \leq 100\sh^2$. The triangle inequalities 
$\dist(B,\CC) + \dist(A,B) \geq \dist(A,\CC)$ and 
$\dist(B,\CC)+\dist(A,\CC) \geq \dist(A,B)$ imply 
$$
	\dist(A,\CC)-\dist(B,\CC) \leq \dist(A,B) \leq \dist(A,\CC)+\dist(B,\CC),
$$ 
and as $\dist(A,\CC)=1$, we deduce 
$1- 100\sh^2 \leq \dist(A,B) \leq 1+100\sh^2$. This finishes the proof that any $A=(0,1+a) \in D_2$ and any $B=(\beta, b) \in D_1$ satisfies ($\star$) and hence the proof of the lemma. 
\end{enumerate}

\begin{figure}[!h]
      	\begin{subfigure}[t]{0.4\textwidth}
    		
       	\includegraphics[scale=1.2]{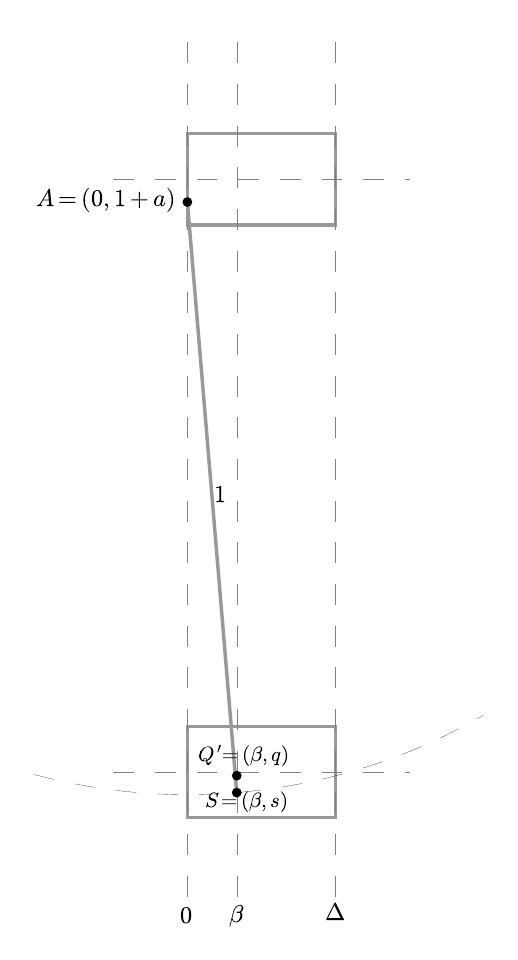}
       	\caption{}
		\label{fig:special_before}
	\end{subfigure}
	~
	\begin{subfigure}[t]{0.6\textwidth}
		
       	\includegraphics[scale=1.2]{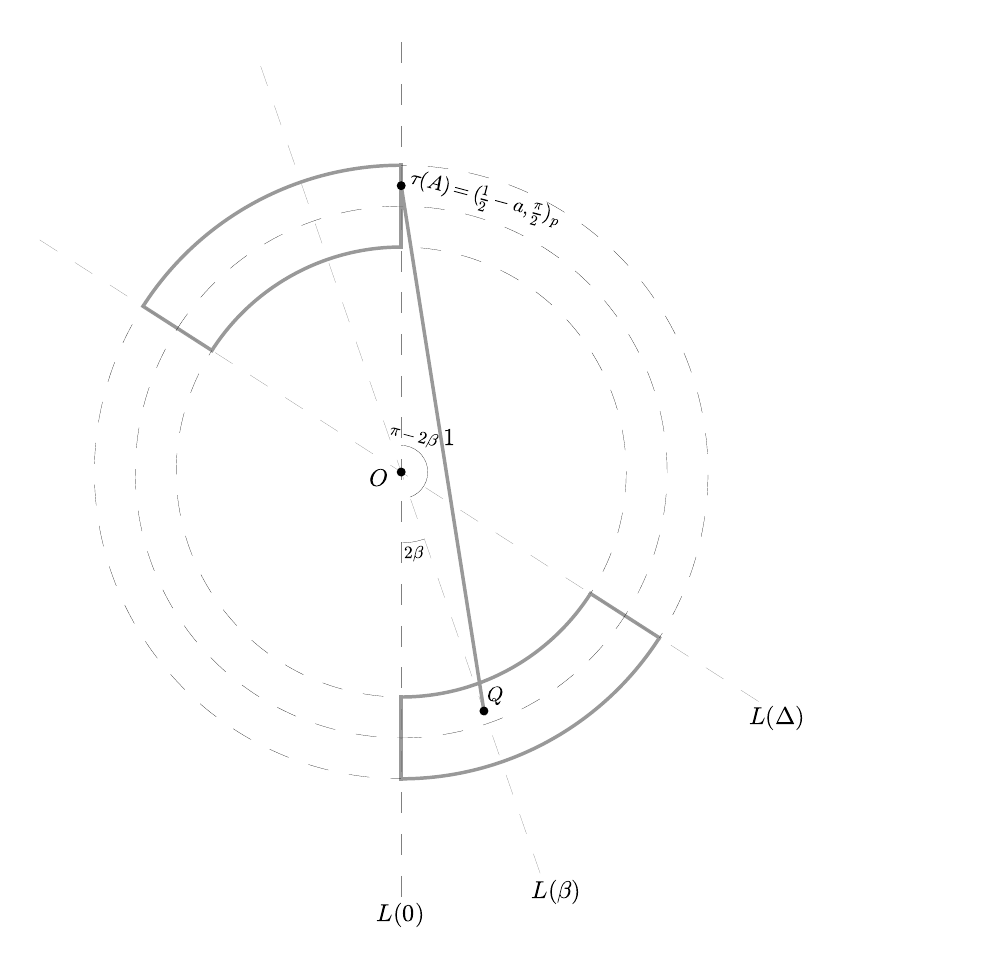}
              \caption{}
		\label{fig:special_after}
	\end{subfigure}

    	\caption{Special points $\CC$ and $\FF$ (here $a<0$)}
    	\label{fig:specialpointsgen}
\end{figure}

\end{proof}


\subsection{$2K_2$-free co-bipartite unit disk graphs}\label{ss:2K2rep}

Now we are ready to use the results of the above section to transform the representation
of a $C_4^*$-free co-bipartite unit disk graph into a representation of its bipartite complement, 
which is a $2K_2$-free co-bipartite graph.

\begin{theorem}
Let $G=(U,W,E)$ be a graph in $\mathcal{X}$. Then $\overline{G}$ is a UDG. 
\end{theorem}

\begin{proof}
First let us choose some $\sh<\frac{q''}{1600}$ satisfying the conditions of Theorem~\ref{c4representation} 
with $q''$ as in the theorem.
By Theorem~\ref{c4representation} we know that ${G}^*$ has a UDG-representation $g$ such that:
\begin{enumerate}
	\item[1)] $g(U) \subseteq D_1$, $g(W) \subseteq D_2$, where $D_1=[0, \Delta] \times [-\sh,\sh]$,
	$D_2=[0,\Delta] \times [1-\sh, 1+\sh]$, for some $\Delta \in (0, 1/3)$; 
	\item[2)] for any two vertices $u \in U$ and $w \in W$, we have either 
	$\dist_g(u, w)=1$ or $|\dist_g(u, w)-1| \geq q''\sh$. 
\end{enumerate}
To employ Lemma \ref{lem:transf} for transforming the UDG-representation $g$ to a
UDG-representation of $\overline{G}$, we must get rid of unit distances. 
To this end we first apply scaling transformation 
$$
	h: (x,y) \rightarrow \left(\left(1-\frac{q''}{2}\sh\right)x, \left(1-\frac{q''}{2}\sh\right)y\right),
$$ 
which scales the whole map by a factor of $1-\frac{q''}{2}\sh$. 
One can observe that distance between images of any two vertices in different parts of $G^*$ 
under the map $h \comp g$ is either at most $1-\frac{q''\sh}{2}$ or at least
$$
	(1 + q''\sh) \times \left(1-\frac{q''}{2}\sh\right)=1 + \frac{q''}{2}\sh - \frac{q''^2}{2}\sh^2 >
	1 + \frac{q''}{2}\sh-\frac{q''}{4}\sh = 1+\frac{q''}{4}\sh,
$$
where the latter inequality is valid because $q'' < 1$ and $\sh < 1/2$.
Therefore, for any vertices $u, w$ in different parts of $G^*$, we 
have $|\dist_{h \comp g}( u, w )-1| \geq \frac{q''}{4}\sh$. Also note that 
$\dist_{h \comp g}(u, w) > 1$ iff $\dist_g(u, w) > 1$, hence $h \comp g$ is a UDG-representation of $G^*$.
We must also note that the scaling affected the strips $D_1$ and $D_2$ as well. 
Though, it is not hard to check that the images of $D_1$ and $D_2$ under the map $h \comp g$ fall into 
the strips $[0, \Delta] \times [-2\sh, 2\sh]$ and $[0, \Delta] \times [1-2\sh, 1+2\sh]$, respectively. 

Finally, the choice of $\sh$ guarantees that $2\sh<\frac{1}{12}$ and 
$|\dist_{h \comp g}( u, w )-1| \geq \frac{q''}{4} \sh > 100 (2\sh)^2$.
Hence Lemma~\ref{lem:transf} applies to the UDG-representation $h \comp g$ of $G^*$ and gives us 
a transformation map $\tau$, such that 
$\tau \comp h \comp g$ is a UDG-representation of $\overline{G}$. This finishes the proof of the theorem.
\end{proof}


\section{Concluding remarks and open problems}\label{sec:conclusion}

In this work we identified infinitely many new minimal forbidden induced subgraphs for
the class of unit disk graphs. Using these results we provided structural characterization of some subclasses of co-bipartite UDGs. Obtaining structural characterization of the 
whole class of co-bipartite UDGs is a challenging research problem. 
An open problem for which such a characterization may be useful is the problem of implicit representation of UDGs.
A hereditary class $\mathcal{G}$ admits an implicit representation if there exists a positive integer $k$
and a polynomial algorithm $A$ such that the vertices of every $n$-vertex graph $G \in \mathcal{G}$
can be assigned labels (binary strings) of length at most $k \log n$ such that given two vertex labels
of $G$ algorithm $A$ correctly decides adjacency of the corresponding vertices in $G$ \cite{KNR92}.
Notice that a class $\mathcal{G}$ admitting an implicit representation has $2^{O(n \log n)}$
$n$-vertex graphs as only $O(n \log n)$ bits is used for encoding each of these graphs.
In \cite{KNR92} Kannan et al. asked whether converse is true, i.e. is it true that every hereditary class
having $2^{O(n \log n)}$ $n$-vertex graphs admits an implicit representation? In \cite{Spinrad03} Spinrad
restated this question as a conjecture, which nowadays is known as \textit{the implicit graph conjecture}.
The class of UDGs satisfies the conditions of the conjecture, i.e. it is hereditary and contains
$2^{\Theta(n \log n)}$ $n$-vertex graphs (see \cite{Spinrad03} and \cite{McDM14}).
Though, no adjacency labeling scheme for the class is known \cite{Spinrad03}. 
A natural approach for such labeling 
would be to associate with every vertex the coordinates of its image under an UDG-representation in $\mathbb{Q}^2$. For this idea to work the integers (numerators and denominators) involved in 
coordinates of points in the UDG-representation
should be bounded by a polynomial of $n$. However, as shown in \cite{McDM13} this can not be 
guaranteed as there are $n$-vertex UDGs for which every UDG-representation necessarily uses
at least one integer of order $2^{2^{\Omega(n)}}$. Therefore some further ideas required
for tackling the problem. For example one may try to combine geometrical and structural properties
of UDGs maybe together with some additional tools (see e.g. \cite{ACLZ15}) to attack the problem of 
implicit representation of UDGs. In particular, from our structural results one can derive an implicit
representation for $C_4^*$-free co-bipartite UDGs and for $2K_2$-free co-bipartite UDGs. 
However, it remains unclear how to get an implicit representation for the whole class of co-bipartite UDGs, 
and it would be very interesting to see such results.

Interestingly, for every discovered co-bipartite forbidden subgraph and substructure its
bipartite complementary counterpart is also forbidden. This suggests that the class of
co-bipartite UDGs may be closed under bipartite complementation. This intuition is further
supported by the result that the bipartite complement of a $C_4^*$-free co-bipartite UDG 
is also (co-bipartite) UDG. These facts lead us to pose the following

\begin{conjecture*}
	For every co-bipartite UDG its bipartite complement is also co-bipartite UDG.
\end{conjecture*}

\noindent
One of the possible approaches to prove this conjecture is, similarly to the proof of 
Lemma \ref{lem:transf}, to show that a representation of a co-bipartite UDG can be 
transformed into a representation of its bipartite complementation.

Another interesting research direction is to investigate systematically properties of edge 
asteroid triple free graphs
as it was done for asteroidal triple free graphs \cite{COS97}. Similarly to co-bipartite UDGs
edge asteroid triples arose in forbidden subgraph characterizations of several other graph
classes such as co-bipartite circular arc graphs \cite{FHH99} and bipartite 2-directional
orthogonal ray graphs \cite{STU10}. However, knowledge about edge asteroid triple free graphs is sporadic, and it would be interesting to study in a consistent manner properties of these graphs, especially, of those graphs which are bipartite.

\section*{Appendices}

\renewcommand{\thesubsection}{\Alph{subsection}}

\subsection{Addendum to the proof of Theorem \ref{basiclobster}}\label{app:T15}

\subsubsection*{Distances between points in $f(C_1)$ and $f(C_2)$}

We split the arguments into cases where we argue about pairs of vertices in $S_1 \times S_2$
for different subsets $S_1 \subseteq C_1, S_2 \subseteq C_2$.
For each pair of vertices we show that the distance between their images is at most 1
if and only if the vertices are adjacent in $G^*$ (see Figure \ref{c4freelobster}).

\begin{enumerate}
	\item $S_1 = \GG_1, S_2 = \GG_2$
	\begin{enumerate}
		\item Edges of $G^*$ between the vertices in $S_1$ and $S_2$: $\{g_i g_j  : j = i \pm 1 \}$.
		\begin{enumerate}
			\item For $j = i \pm 1$, $\dist_f(g_i, g_{j}) = \sqrt{1-\mm^2 + \mm^2} =1$.
			\item For $j \neq i \pm 1$, $\dist_f(g_i, g_{j}) \geq \dist_f(g_i, g_{i+3})=\sqrt{1-\mm^2+(3\mm)^2}=\sqrt{1+8\mm^2} \geq 1+2\mm^2 $.
		\end{enumerate}
	\end{enumerate}
	
	\item $S_1 = \GG_1, S_2 = \BB_2 \cup \RR_2$ or $S_1 = \BB_1 \cup \RR_1, S_2 = \GG_2$
	\begin{enumerate}
		\item Edges of $G^*$ between the vertices in $S_1$ and $S_2$: $\{g_i b_i  : i = 1, \ldots, n \}$.
		\begin{enumerate}
			\item $\dist_f(g_i, b_i)=1-\frac{1-\sqrt{1-\mm^2}}{2}=\frac{1}{2}+\frac{\sqrt{1-\mm^2}}{2} \leq \frac{1}{2}+\frac{1-\mm^2/2}{2}=1-\frac{\mm^2}{4}$.

			\item The distances between $f(g_i)$ and $f(b_j)$ with $j \neq i$ or 
			between $f(g_i)$ and $f(r_k)$ are at least 
			\begin{align*}
			\dist_f(g_i,r_{i+1}) = &\sqrt{\dist_f(g_i,b_i)^2+\dist_f(b_i,r_{i+1})^2} 
			\geq \sqrt{\left(1-\frac{\mm^2}{4}-\frac{\mm^4}{4}\right)^2+\mm^2} \\
			 \geq & 
			 \sqrt{\left(1-\frac{\mm^2}{4}-\frac{\mm^2}{8}\right)^2+\mm^2}
			 \geq \sqrt{1-\frac{3}{4}\mm^2 + \mm^2}=\sqrt{1+\frac{1}{4}\mm^2} 
			 \geq 1+\frac{1}{16}\mm^2.
			 \end{align*}
           		 Note that the first inequality uses our basic inequality (1) and for the second we used the fact that $\mm \leq \sqrt{\frac{1}{2}}$. 	
	\end{enumerate}	
	\end{enumerate}
	\item $S_1 = \RR_1 \cup \BB_1, S_2 = \RR_2 \cup \BB_2$
	\begin{enumerate}
		\item Edges of $G^*$ between the vertices in $S_1$ and $S_2$: $\{r_i b_i  : i = 1, \ldots, n \}$.
		\begin{enumerate}
			\item $\dist_f(r_i,b_i)=1$.
			\item The distances between any other two points one on $L_1$ and another on $L_4$ are 
			at least $\sqrt{1+\mm^2} \geq 1+\frac{1}{4}\mm^2$.
		\end{enumerate}
	\end{enumerate}
\end{enumerate}

\noindent
Observe that we have proved that for any two vertices $v \in S_1$ and $w \in S_2$, 
either $\dist_f(v,w)=1$ or $|\dist_f(v,w) -1| \geq \frac{1}{16} \mu^2$.

\subsection{Addendum to the proof of Theorem \ref{th:repBasicGraphs}}

\subsubsection{Proof of Claim 4}\label{app:claim4}

Here we verify the following claim from the proof of Theorem \ref{th:repBasicGraphs} 
(see Figure \ref{fig:2hex_vert_representation} for illustration)

\begin{itemize}[leftmargin=0px]
	\item[] \textbf{Claim 4.} Denote the midpoints of $[b_1, y_{12}]$, $[y_{12}, g_2]$, $[y_{34}, g_4]$, 
	$[y_{45}, b_5]$ by $b_1'', R_2, R_4, b_5'$, respectively. Then, for 
	$x \in \{b_1, b_1'', R_2, R_4, g_4, r_4, b_5',b_5\}$ we have $\dist(g_3', x)>1$ and for 
	$x \in \{g_2, r_2, b_3\}$ we have $\dist (g_3', x)<1$.    
\end{itemize}

\noindent
\textbf{Proof.}
We prove the claim by direct estimation of distances for different pairs $(x,y)$ of points:

\begin{enumerate}
	\item $(g'_3, b_1)$: 
	$\dist(g_3', b_1)> \dist(x_{23}, b_1)=\dist(g_1,y_{23}) \geq \sqrt{1+\epsilon^2} \geq 1+ \epsilon^2/4$ by Claim 2.
	
	\item $(g'_3, b_1'')$: 
	as $\textup{Conv}(x_{23},b_1,b_1'', g_3')$ is a parallelogram, 
	$\dist(g_3', b''_1)=\dist(x_{23}, b_1)=\dist(g_1,y_{23}) \geq \sqrt{1+\epsilon^2} \geq 1+ \epsilon^2/4$ by Claim 2.

	\item $(g'_3, g_2)$: 
	$\dist(g_3', g_2)=\sqrt{1-\epsilon^2} \leq 1-\epsilon^2/2$.
	
	\item $(g'_3, b_3)$: 
	$\dist(g_3', b_3)<\dist(x_{23}, b_3) = \dist(g_1,y_{12}) \leq \sqrt{1-(\epsilon/2)^2} \leq 1-\epsilon^2/8$ by Claim 1.
	
	\item $(g_3', y)$, where $y \in \{g_4, r_4, b_5', b_5\}$:
	 $\dist(g_3', y) > \dist(g_3', g_4) \geq \sqrt{1+\epsilon^2} \geq 1+\epsilon^2/4$, 
	 follows by applying the Law of cosines to triangle $\triangle g_3'g_3g_4$ as $\dist(g_4,g_3)=1$,
	 $\dist(g_3,g_3')=\epsilon$ and $90 < \angle g_3'g_3g_4 < 180$.
	
	\item $(g'_3, R_2)$: 
	denote $\angle x_{23} g_2 g_3'= \alpha$, and notice that $\sin(\alpha)= \epsilon$, 
	$\dist(g_2, g_3')=\sqrt{1-\epsilon^2}$, and $\angle g_3' g_2 R_2=\alpha+90$, then
	 
	\begin{align*}
		\dist(g_3', R_2)^2 = & (\epsilon/2)^2+1-\epsilon^2-2\cos(\alpha+90) 
		(\epsilon/2)\sqrt{1-\epsilon^2} \\
		= & 1-3\epsilon^2/4+\sin(\alpha)\epsilon\sqrt{1-\epsilon^2} \\
		= & 1-3\epsilon^2/4+\epsilon^2\sqrt{1-\epsilon^2} \\
		> & 1+\epsilon^2/8
	\end{align*}
	whenever $\sqrt{1-\epsilon^2}>7/8$, which holds for $\epsilon < \sqrt{15}/8$. 
	Hence, $\dist(g_3', R_2)> \sqrt{1+\epsilon^2/8} \geq 1+\epsilon^2/32$.
	
	\item $(g'_3, r_2)$:
	notice that $\angle g_3' g_2 r_2=\gamma<90$, thus 
	\begin{align*}
		\dist(g_3',r_2)^2 = & \dist(g_2, g_3')^2+\dist(g_2, r_2)^2 -2\cos(\gamma)\dist(g_2,g_3')\dist(g_2, r_2) \\
		< &\dist(g_2, g_3')^2+\dist(g_2, r_2)^2 \\
		= & 1-\epsilon^2+(\epsilon/2)^2 \\
		= & 1 - 3\epsilon^2/4,
	\end{align*}
	 that is $\dist(g_3',r_2) < \sqrt{1 - 3\epsilon^2/4} \leq 1-3\epsilon^2/8 $.
	 
	\item $(g'_3, R_4)$:
	by comparing the slope of $[y_{34}, g_4]$ and $[g_3, g_3']$ and denoting the point $x$ to be the middle
	point of $[g_3, g_3']$, one can easily see that 
	$$
		\dist(g_3', R_4) \geq \dist(x, g_4) > \dist (g_3, g_4)=1.
	$$
	Indeed, one can obtain $\dist(g_3', R_4) \geq \dist(x, g_4) > \sqrt{1+(\epsilon/2)^2} \geq 1+ \epsilon^2/16$.

\end{enumerate}

\noindent
Notice, in particular, that for $x$ as in the statement of Claim 4, we have proved 
$|\dist(g_3', x)-1|>\epsilon^2/32$.

\qed

\subsubsection{Proof that $\dist(a_3, r_2)>1$} \label{app:dist_a3_r2}

First, we observe that $[g_3',a_3]$ is parallel to $[r_2, R_2]$. 
Intuitively, both of the intervals have length close to $\epsilon$, and we also know that 
$\dist(g_3', R_2) \geq \sqrt{1+\epsilon^2/8}$. By the triangle inequality, we can deduce that 
$$
	\dist(a_3, r_2) \geq \dist(g_3', R_2) - |\dist(g_3', a_3)-\dist(r_2,R_2)|.
$$
We would like to show that $|\dist(g_3', a_3)-\dist(r_2,R_2)|$ is small.
To calculate these distances let us denote $\beta = \angle b_2 g_2 x_{23}$ and 
$\alpha = \angle x_{23} g_2 g_3'$. Then, $\angle g_2 g_3 x_{23} =90-\alpha$ and
$\angle g_2 g_3 b_3 =\angle g_3 g_2 b_2 =2\alpha+\beta$. Thus, $\angle b_3 g_3 g_3' =90-\alpha+2\alpha+\beta=90+\alpha+\beta$. Hence, 
$\angle g_3' g_3 a_3 =90-\alpha-\beta$ and we can calculate 
$$
	\dist(g_3', a_3)=\sin(90-\alpha-\beta)\epsilon=\cos(\alpha+\beta)\epsilon.
$$
Further, by noticing that $\angle b_2 g_2 r_2 =90-\beta$ we calculate 
$$
	\dist(r_2, R_2)=2\sin(90-\beta)(\epsilon/2)=\cos(\beta)\epsilon.
$$
Now, $\cos(\beta)=\frac{1}{\sqrt{1+\epsilon^2}}$, $\sin(\beta)=\frac{\epsilon}{\sqrt{1+\epsilon^2}}$, $\cos(\alpha)=\sqrt{1-\epsilon^2}$, $\sin(\alpha)=\epsilon$, and therefore
$$
	\cos(\alpha+\beta)=\cos(\alpha)\cos(\beta)-\sin(\alpha)\sin(\beta)=\frac{\sqrt{1-\epsilon^2}}{\sqrt{1+\epsilon^2}}-\frac{\epsilon^2}{\sqrt{1+\epsilon^2}}.
$$
Thus
\begin{align*}
	|\dist(g_3', a_3)-\dist(r_2,R_2)| =& \epsilon \cos(\beta) - \epsilon \cos(\alpha+\beta) \\
	= & \epsilon \left( \frac{1}{\sqrt{1+\epsilon^2}}-\frac{\sqrt{1-\epsilon^2}}{\sqrt{1+\epsilon^2}}+
	\frac{\epsilon^2}{\sqrt{1+\epsilon^2}} \right) \\
	\leq &\epsilon \left(\frac{1-(1-\epsilon^2)+\epsilon^2}{\sqrt{1+\epsilon^2}} \right) \\
	= & \frac{2\epsilon^3}{\sqrt{1+\epsilon^2}} \leq 2 \epsilon^3.
\end{align*} 
Finally, we conclude that 
$$
	\dist(a_3, r_2) \geq \sqrt{1+\epsilon^2/8} - 2 \epsilon^3 \geq 1+
	\epsilon^2/32  - 2\epsilon^3 \geq 1+ \epsilon^2/64,
$$ whenever $\epsilon < 1/128$. 

\qed

\subsection{Addendum to the proof of Lemma \ref{lem:transf}}\label{app:LemTrans}

\subsubsection*{Lower and upper bounds on $\ff$}

Below we derive the following bounds on $\ff$
$$
a+ \frac{\beta^2}{2}-\frac{7\beta^4}{6}-2a^2\beta^2 \leq \ff \leq a+\frac{\beta^2}{2}+\frac{\beta^4}{2}.
$$

\noindent
\textbf{Proof.}
One can apply the law of cosines to the triangle $\triangle \tau(A) \FF O$ and obtain the equation 
$$
\dist(\tau(A),O)^2 + \dist(O,\FF)^2-2\cos(\angle\tau(A)O\FF)\dist(\tau(A),O)\dist(O,\FF)= \dist(\tau(A),\FF)^2.
$$
Inserting the values $\dist(\tau(A),O)=\frac{1}{2}-a$, $\dist(\tau(A),\FF)=1$ and $\cos(\angle\tau(A)O\tau(B))=\cos(\pi-2\beta)=-\cos(2\beta)$, we get the equation 
$$
	\left(\frac{1}{2}-a\right)^2+\dist(O,\FF)^2+2\cos(2\beta)\left(\frac{1}{2}-a\right)\dist(O,\FF)=1.
$$
Solving the quadratic equation yields 
$$
	\dist(O,\FF)=-\cos(2\beta)\left(\frac{1}{2}-a\right)\pm\sqrt{1-\left(\frac{1}{2}-a\right)^2
	+\left(\cos(2\beta)\left(\frac{1}{2}-a\right)\right)^2}.
$$ 
This equation has one positive and one negative root, and therefore we must choose the positive sign. Hence, 

\begin{align*}
\ff = & \dist(O,\FF)-\frac{1}{2} \\
= & -\frac{1}{2}-\frac{\cos(2\beta)}{2}+a\cos(2\beta)+\sqrt{1+\left(\frac{1}{2}-a\right)^2(\cos(2\beta)^2-1)}\\
= & a-\cos^2(\beta)-2a\sin^2(\beta)+\sqrt{1-\left(\frac{1}{2}-a\right)^2\sin^2(2\beta)}.
\end{align*}

\noindent
Consider now
\begin{align*}
	K =  \left(\frac{1}{2}-a\right)^2 \sin^2(2\beta) 
	=  \left(\frac{1}{2}-a\right)^2 4\sin^2(\beta)\cos^2(\beta)
	=  (1-2a)^2 \sin^2(\beta)(1-\sin^2(\beta)).
\end{align*}
Expanding the brackets, one deduces that 
$$
	K=\sin^2(\beta)-4a \sin^2(\beta)+4a^2 \sin^2(\beta)-\sin^4(\beta)+4a\sin^4(\beta)-4a^2\sin^4(\beta) \geq \sin^2(\beta)-4a\sin^2(\beta)-\sin^4(\beta),
$$
because both $4a^2 \sin^2(\beta) $ and $4a\sin^4(\beta)-4a^2\sin^4(\beta)$ are non-negative. 
This allows us to obtain the desired upper bound for $\ff$: 

\begin{align*} 
\ff & = a-\cos^2(\beta)-2a\sin^2(\beta)+\sqrt{1-K} \\
  & \leq a-\cos^2(\beta)-2a\sin^2(\beta)+1-\frac{K}{2} \\
   &   \leq a-\cos^2(\beta)+1-2a\sin^2(\beta)-\frac{\sin^2(\beta)}{2}+2a\sin^2(\beta)+\frac{\sin^4(\beta)}{2}\\
   &   =a+\sin^2(\beta)-\frac{\sin^2(\beta)}{2}+\frac{\sin^4(\beta)}{2} \\
   &    \leq a+\frac{\beta^2}{2}+\frac{\beta^4}{2}.
\end{align*} 

\noindent
It is also easy to derive that $K \leq (1-2a)^2 \sin^2(\beta)$ and in particular 
$$
	K^2 \leq (1-2a)^4\sin^4(\beta) \leq (1+2\sh)^4\sin^4(\beta) \leq (1+2/12)^4 \sin^4(\beta) \leq 2\sin^4(\beta).
$$ 
This allows us to deduce the lower bound for $\ff$:

\begin{align*}
\ff  & = a-\cos^2(\beta)-2a\sin^2(\beta)+\sqrt{1-K} \\
   & \geq a-\cos^2(\beta)-2a\sin^2(\beta)+1-\frac{K}{2}-\frac{K^2}{2} \\
   & \geq a-\cos^2(\beta)+1-2a\sin^2(\beta)-\frac{\sin^2(\beta)}{2}+2a\sin^2(\beta)-2a^2\sin^2(\beta)-
   \frac{2 \sin^4(\beta)}{2}\\
     &    \geq a+\sin^2(\beta)-\frac{\sin^2(\beta)}{2}-2a^2\sin^2(\beta)-\sin^4(\beta)\\
       &  \geq a+\frac{\sin^2(\beta)}{2}-2a^2\sin^2(\beta)-\sin^4(\beta)\\
         &\geq a+ \frac{1}{2}\left(\beta-\frac{\beta^3}{6}\right)^2-2a^2\beta^2-\beta^4\\
          &\geq a + \frac{\beta^2}{2}-\frac{\beta^4}{6}-2a^2\beta^2-\beta^4\\
	& \geq a+ \frac{\beta^2}{2}-\frac{7\beta^4}{6}-2a^2\beta^2.
\end{align*}

\qed

\end{document}